\documentclass[12pt]{amsart}
\usepackage{amsmath,amsthm,amssymb,latexsym,hyperref}
\usepackage[utf8]{inputenc}
\usepackage{amsfonts}
\usepackage{bbm}

\numberwithin{equation}{section}
\numberwithin{equation}{subsection}

\renewcommand*{\theequation}{%
  \ifnum\value{subsection}=0 %
    \thesection
  \else
    \thesubsection
  \fi
  .\arabic{equation}%
}

\newtheorem{thm}{Theorem}[section]
\newtheorem{lem}{Lemma}[section]

\newtheorem{prop}{Proposition}[section]

\title[Quadratic Dirichlet $L$-functions of prime conductor]{Dirichlet $L$-functions of quadratic characters of prime conductor at the central point}

\author{Siegfred Baluyot}
\address{Department of Mathematics \\
 University of Illinois at Urbana-Champaign \\
 1409 West Green Street, Urbana, IL 61801}
\email{\href{mailto:sbaluyot@illinois.edu}{sbaluyot@illinois.edu}}
\author{Kyle Pratt}
\email{{\href{mailto:kpratt4@illinois.edu}{kpratt4@illinois.edu}},{\href{mailto:kvpratt@gmail.com}{kvpratt@gmail.com}}}


\begin{document}
\date{}

\maketitle

\begin{abstract}
We prove that more than nine percent of the central values $L(\frac{1}{2},\chi_p)$ are non-zero, where $p\equiv 1 \pmod{8}$ ranges over primes and $\chi_p$ is the real primitive Dirichlet character of conductor $p$. Previously, it was not known whether a positive proportion of these central values are non-zero. As a by-product, we obtain the order of magnitude of the second moment of $L(\frac{1}{2},\chi_p)$, and conditionally we obtain the order of magnitude of the third moment. Assuming the Generalized Riemann Hypothesis, we show that our lower bound for the second moment is asymptotically sharp.
\end{abstract}

\tableofcontents

\section{Introduction and results}

The values of $L$-functions at special points on the complex plane are of great interest. At the fixed point of the functional equation, called the central point, the question of non-vanishing is particularly important. For instance, the well-known Birch and Swinnerton-Dyer conjecture~\cite{Wiles} relates the order of vanishing of certain $L$-functions at the central point to the arithmetic of elliptic curves. Katz and Sarnak~\cite{KS99} discuss several examples of families of $L$-functions and describe how the zeros close to $s=\frac{1}{2}$ give evidence of some underlying symmetry group for each of these families. They suggest that understanding these symmetries may in turn lead to finding a natural spectral interpretation of the zeros of the $L$-functions. The analysis of each family they discuss leads to a \textit{Density Conjecture} that, if true, would imply that almost all $L$-functions in the family do not vanish at the central point. Iwaniec and Sarnak~\cite{iwaniecsarnak} show that the non-vanishing of $L$-functions associated with holomorphic cusp forms is closely related to the Landau-Siegel zero problem. Thus the question of non-vanishing at the central point is connected to many deep arithmetical problems.

A considerable amount of research has been done towards answering this question for families of Dirichlet $L$-functions. Chowla conjectured that $L(\frac{1}{2},\chi) \neq 0$ for $\chi$ a primitive quadratic Dirichlet character \cite[p. 82, problem 3]{Chow}. It has since become a sort of folklore conjecture that $L(\frac{1}{2},\chi)\neq 0$ for all primitive Dirichlet characters $\chi$. One family that has attracted a lot of attention is the family of $L(s,\chi)$ with $\chi$ varying over primitive characters modulo a fixed conductor. This family is widely believed to have a unitary symmetry type, as in the philosophy of Katz and Sarnak. Balasubramanian and Murty~\cite{BalasubramanianMurty} were the first to prove that a (small) positive proportion of this family does not vanish at the central point. They used the celebrated technique of mollified moments, a method that has been highly useful in other contexts (see, for example,~\cite{BL14,CGG,Sel42}). Iwaniec and Sarnak \cite{IS99} developed a simpler, stronger method and improved this proportion to $\frac{1}{3}$. The approach of Iwaniec and Sarnak has since become standard in the study of non-vanishing of $L$-functions at the central point. Bui \cite{Bui12} and Khan and Ngo \cite{KN16} introduced new ideas and further improved the lower bound $\frac{1}{3}$. The second author \cite{Pratt} has shown that more than fifty percent of the central values are non-vanishing when one additionally averages over the conductors. For further interesting research on this and other families of $L$-functions, see \cite{BM11,DK18,Khan10,Khan12,KDN16, KM00, KMV1, KMV2, MV00,   OnoSkinner}.

The family of $L(s,\chi)$ with $\chi$ varying over all real primitive characters has also been extensively studied. This family is of particular significance because it seems to be of symplectic rather than unitary symmetry. Thus we encounter new phenomena not seen in the unitary case. For $d$ a fundamental discriminant, set $\chi_d(\cdot) = \left(\frac{d}{\cdot} \right)$, the Kronecker symbol. Then $\chi_d$ is a real primitive character with conductor $|d|$. The hypothetical positivity of central values $L(\frac{1}{2},\chi_d)$ has implications for the class number of imaginary quadratic fields \cite[p. 514]{IK}. Jutila~\cite{Jut81} initiated the study of non-vanishing at the central point for this family and proved that $L(\frac{1}{2},\chi_d) \neq 0$ for infinitely many fundamental discriminants $d$. His methods show that $\gg X/\log X$ of the quadratic characters $\chi_d$ with $|d|\leq X$ have $L(\frac{1}{2},\chi_d) \neq 0$. \"{O}zl\"{u}k and Snyder~\cite{OS93} examined the low-lying zeros of this family, and found the first evidence of its symplectic behavior. Assuming the Generalized Riemann Hypothesis (GRH), they showed that more than $\frac{15}{16}$ of the central values $L(\frac{1}{2},\chi_d)$ are non-zero~\cite{OS99}. Katz and Sarnak independently obtained the same result in unpublished work (see~\cite{KS99,Sou00}).

Soundararajan~\cite{Sou00} made a breakthrough when he proved unconditionally that more than $\frac{7}{8}$ of the central values $L(\frac{1}{2},\chi_d)$ with $d\equiv 0$ (mod $8$) are non-zero. The biggest difficulty lies in analyzing the contribution of the ``off-diagonal'' terms in the evaluation of a mollified second moment. Soundararajan discovered that there is, in fact, a main contribution arising from these off-diagonal terms. (See Section~\ref{sec: outline of pos prop thm} for more discussion.)

The case of real primitive characters with prime conductor is more difficult still. Jutila~\cite{Jut81} initiated the study of $L(\frac{1}{2},\chi_p)$, where $p$ is a prime. His methods yield that $\gg X/(\log X)^3$ of the primes $p \leq X$ satisfy $L(\frac{1}{2},\chi_p) \neq 0$. The difficulty in studying this family is that its moments involve sums over primes, and thus are more complicated to investigate. In fact, Jutila only evaluated the first moment of this family. As far as the authors are aware, no asymptotic evaluation of the second moment has appeared in the literature. However,  Andrade and Keating \cite{AK12} asymptotically evaluated the second moment of an analogous family over function fields. Andrade and the first author~\cite{AB} have continued the study of the family of $L(\frac{1}{2},\chi_p)$, showing that it is likely governed by a symplectic law. Conditionally on GRH, they prove that more than 75\% of primes $p \leq X$ satisfy $L(\frac{1}{2},\chi_p) \neq 0$.

We prove an unconditional positive proportion result for the central values $L(\frac{1}{2},\chi_p)$. In fact, we prove that more than nine percent of these central values are non-zero.

\begin{thm}\label{thm: pos prop}
There exists an absolute, effective constant $X_0$ such that if $X \geq X_0$ then
\begin{align*}
\sum_{\substack{p \leq X \\ p \equiv 1 \, (\textup{mod }8) \\ L(\frac{1}{2},\chi_p)\neq 0}}1 &\geq .0964 \sum_{\substack{p \leq X \\ p \equiv 1 \, (\textup{mod }8)}} 1.
\end{align*}
\end{thm}

The proof of Theorem \ref{thm: pos prop} proceeds via the mollification method, which we discuss briefly in Section \ref{sec: outline of pos prop thm} below. Our methods build on those of Jutila \cite{Jut81} and Soundararajan \cite{Sou00}. As in the work of Soundararajan, the main difficulty lies in evaluating the contribution of certain off-diagonal terms. The difference now is that we are summing over primes instead of over square-free integers, and so we cannot directly use his approach. A key idea in the proof of Theorem \ref{thm: pos prop} is the use of upper bound sieves to turn intractable sums over primes into manageable sums over integers. The use of sieves in studying central values of $L$-functions has also appeared in some other contexts (see \cite{HL97}, also \cite[p. 1035]{RS15}).

The tools developed for the proof of Theorem \ref{thm: pos prop} allow us to obtain the order of magnitude of the second moment of $L(\frac{1}{2},\chi_p)$.

\begin{thm}\label{thm: 2nd moment upper bound}
Let $\mathfrak{c}$ be the positive constant
\begin{align*}
\mathfrak{c} &:= \left(144 \zeta(2) \left(1 - \frac{1}{\sqrt{2}}\right)^2\right)^{-1}=.0492\ldots \, .
\end{align*}
For large $X$ we have
\begin{align*}
(\mathfrak{c} - o(1)) \frac{X}{4}(\log X)^3 \leq \sum_{\substack{p \leq X \\ p \equiv 1 \, (\textup{mod }8)}} (\log p) L\left( \tfrac{1}{2},\chi_p\right)^2 \leq (4\mathfrak{c} + o(1)) \frac{X}{4} (\log X)^3.
\end{align*}
\end{thm}

One would rather have an upper bound in Theorem \ref{thm: 2nd moment upper bound} that asymptotically matches the lower bound, but this seems difficult to prove unconditionally. By adapting a method of Soundararajan and Young \cite{SY10} we are able, however, to prove such an asymptotic formula on GRH.

\begin{thm}\label{thm: 2nd moment GRH}
Let $\mathfrak{c}$ be as in Theorem \ref{thm: 2nd moment upper bound}. Assume the Riemann Hypothesis for $\zeta(s)$ and for all Dirichlet $L$-functions $L(s,\chi_p)$ with $p \equiv 1 \pmod{8}$. Then
\begin{align*}
\sum_{\substack{p \leq X \\ p \equiv 1 \, (\textup{mod }8)}} (\log p) L \left( \tfrac{1}{2},\chi_p \right)^2 = \mathfrak{c} \frac{X}{4}(\log X)^3 + O(X (\log X)^{11/4}).
\end{align*}
\end{thm}
After we completed this paper, Maksym Radziwi\l\l\ informed us about work in progress with Julio Andrade, Roger Heath-Brown, Xiannan Li, and K. Soundararajan in which they derive an unconditional asymptotic formula for the second moment of $L(\frac{1}{2},\chi_p)$. Their approach similarly introduces sieve weights, and they also observed that this idea could lead to a non-vanishing result.

Our methods further yield the order of magnitude of the third moment of $L(\frac{1}{2},\chi_p)$, assuming that the central values $L(\frac{1}{2},\chi_n)$ are non-negative for certain fundamental discriminants $n$. This non-negativity hypothesis follows, of course, from GRH.

\begin{thm}\label{thm: third moment}
Assume that for all positive square-free integers $n$ with $n \equiv 1 \pmod{8}$ it holds that $L(\frac{1}{2},\chi_n) \geq 0$. Then for large $X$
\begin{align*}
\sum_{\substack{p \leq X \\ p \equiv 1 \, (\textup{mod }8)}} (\log p) L\left( \tfrac{1}{2},\chi_p\right)^3 \asymp X(\log X)^6.
\end{align*}
\end{thm}

Throughout this paper, we work exclusively with $p \equiv 1 \pmod{8}$ for convenience, but our methods are not specific to this residue class. With some modifications one could state similar results for other residue classes modulo 8. See the end of Section~\ref{sec: outline of pos prop thm} for more details.

Our work indicates that Soundararajan's lower bound \cite{Sou00} for the proportion of non-vanishing for fundamental discrimimants $d \equiv 0 \pmod{8}$ also holds for the case of fundamental discriminants $d \equiv 1\pmod{8}$. Proving this involves re-doing the calculations in Section~\ref{sec:mollified second moment}, but without applying an upper bound sieve. To complete the proof, one would also need a first moment calculation. We omit the details and instead refer the reader to \cite[Section~4]{Sou00}.

It is natural to ask about the limitations of our method, and how much we can increase the lower bound in Theorem~\ref{thm: pos prop}. If we assume that we can use arbitrarily long mollifiers~\cite{Farmer}, then we obtain a higher percentage of non-vanishing. However, in view of the parity problem of sieve theory \cite{FI09}, we could not reach a proportion greater than $\frac{1}{2}$ via our method. On the other hand, by a different method~\cite{AB}, the Density Conjecture of Katz and Sarnak would imply that 100\% of the central values $L(\tfrac{1}{2},\chi_p)$ are nonzero.

The outline of the rest of the paper is as follows. In Section~\ref{sec:notation} we establish some notation and conventions that hold throughout this work. Section~\ref{sec: outline of pos prop thm} outlines the basic strategy for the proof of Theorem~\ref{thm: pos prop}. In Sections~\ref{sec:lemmas} and \ref{sec:sieves} we state a number of important technical results which are used in the proofs of our theorems. The proof of Theorem~\ref{thm: pos prop} is spread across Sections~\ref{sec:mollified first moment}, \ref{sec:mollified second moment}, and \ref{sec:finish proof of pos prop thm}. In Section~\ref{sec:mollified first moment} and its subsections we study the mollified first moment problem. The very long Section~\ref{sec:mollified second moment} and its subsections handle the mollified second moment. We choose our mollifier and finish the proof of Theorem~\ref{thm: pos prop} in Section~\ref{sec:finish proof of pos prop thm}. We prove Theorems~\ref{thm: 2nd moment upper bound} and \ref{thm: 2nd moment GRH} in Section~\ref{sec:second moment theorems}, and we prove Theorem~\ref{thm: third moment} in Section~\ref{sec:third moment}.

\section{Notation and conventions}\label{sec:notation}

We define $\chi_n(\cdot)=\left(\frac{n}{\cdot}\right)$, the Kronecker symbol, for all nonzero integers $n$, even if $n$ is not a fundamental discriminant. Note that this means $\chi_n$ has conductor $|n|$ only when $n$ is a fundamental discriminant. We write $S(Q)$ for the set of all real primitive characters $\chi$ with conductor $\leq Q$. For an integer $n$, we write $n = \square$ or $n \neq \square$ according to whether or not $n$ is a perfect square.

We let $\varepsilon > 0$ denote an arbitrarily small constant whose value may vary from one line to the next. When $\varepsilon$ is present, in some fashion, in an inequality or error term, we allow implied constants to depend on $\varepsilon$ without necessarily indicating this in the notation. At times we indicate the dependence of implied constants on other quantities by use of subscripts: for example, $Y \ll_A Z$.

Throughout this paper, we denote by $\Phi(x)$ a smooth function, compactly supported in $[\frac{1}{2},1]$, which satisfies $\Phi(x) = 1$ for $x \in [\frac{1}{2} + \frac{1}{\log X},1 - \frac{1}{\log X}]$ and $\Phi^{(j)}(x) \ll_j (\log X)^j$ for all $j \geq 0$. We could state our results for arbitrary smooth functions supported in $[\frac{1}{2},1]$, but we avoid this in an attempt to achieve some simplicity.

We write $e(x) = e^{2\pi i x}$. For $g$ a compactly supported smooth function, we define the Fourier transform $\hat{g}(y)$ of $g$ by
\begin{align*}
\hat{g}(y) = \int_\mathbb{R} g(x) e(-xy) dx.
\end{align*}
At times, however, we find it convenient to use a slightly different normalization of the Fourier transform (see Lemma \ref{Selbergsieve}).

We define the Mellin transform $g^{\dagger}(s)$ of $g$ by
\begin{align*}
g^{\dagger}(s) = \int_0^\infty g(x) x^{s-1} dx.
\end{align*}
It is also helpful to define a modified Mellin transform $\check{g}(w)$ by
\begin{align*}
\check{g}(w) = \int_0^\infty g(x) x^w dx.
\end{align*}
Observe that $\check{g}(w) = g^\dagger(1+w)$. Lastly, for a complex number $s$, we define
\begin{align*}
g_s(t) = g(t) t^{s/2}.
\end{align*}

Note that
\begin{align*}
\hat{\Phi}(0) = \Phi^{\dagger}(1) = \check{\Phi}(0)= \frac{1}{2} + O\left( \frac{1}{\log X}\right).
\end{align*}

The letter $p$ always denotes a prime number. We write $\varphi$ for the Euler phi function, and $d_k$ for the $k$-fold divisor function. If $a$ and $b$ are integers we write $[a,b]$ for their least common multiple and $(a,b)$ for their greatest common divisor. It will always be clear from context whether $[a,b]$, say, denotes a least common multiple or a real interval.

Given coprime integers $a$ and $q$, we write $\overline{a}\pmod{q}$ for the multiplicative inverse of $a$ modulo $q$.

\section{Outline of the proof of Theorem \ref{thm: pos prop}}\label{sec: outline of pos prop thm}

The proof of Theorem \ref{thm: pos prop} proceeds through the mollification method. The method was introduced by Bohr and Landau \cite{BL14}, but later greatly refined in the hands of Selberg \cite{Sel42}. The idea is to introduce a Dirichlet polynomial $M(p)$, known as a mollifier, which dampens the occasional wild behavior of the central values $L(\frac{1}{2},\chi_p)$. We study the first and second moments
\begin{equation}\label{eq: defn moment sums S1 and S2}
\begin{split}
S_1 &:= \sum_{\substack{p \equiv 1 \, (\text{mod }8)}} (\log p)\Phi \left( \frac{p}{X}\right) L\left(\tfrac{1}{2},\chi_p\right) M(p), \\
S_2 &:= \sum_{\substack{p \equiv 1 \, (\text{mod }8)}} (\log p)\Phi \left( \frac{p}{X}\right) L\left(\tfrac{1}{2},\chi_p\right)^2 M(p)^2.
\end{split}
\end{equation}
If the mollifier is chosen well then $S_1 \gg X$ and $S_2 \ll X$. By the Cauchy-Schwarz inequality we have
\begin{align}\label{eq: Cauchy Schwarz first second moment inequality}
\sum_{\substack{p \equiv 1 \, (\text{mod }8) \\ L(\frac{1}{2},\chi_p) \neq 0}} (\log p)\Phi \left( \frac{p}{X}\right) \geq \frac{S_1^2}{S_2},
\end{align}
and this implies that a positive proportion of $L(\frac{1}{2},\chi_p)$ are non-zero. 

Our mollifier takes the form
\begin{align}\label{eq: defn of mollifier}
M(p) := \sum_{\substack{m \leq M \\ m \text{ odd}}} \frac{b_m}{\sqrt{m}} \chi_p(m),
\end{align}
for some coefficients $b_m$ we describe shortly. Here we set 
\begin{align}\label{eq:outline section, defn of M, length of mollifier}
M = X^\theta, \ \ \ \ \ \ \ \ \ \ \ \theta \in \left(0 ,\tfrac{1}{2}\right) \ \  \text{ fixed}.
\end{align}
The larger one can take $\theta$, the better proportion of non-vanishing one can achieve.

The coefficients $b_m$ are a smoothed version of the M\"obius function $\mu(m)$. Specifically, we choose 
\begin{align}\label{eq: defn of mollifier coeffs bm}
b_m = \mu(m) H \left(\frac{\log m}{\log M} \right),
\end{align}
where $H(t)$ is smooth function compactly supported in $[-1,1]$ which we choose in Section \ref{sec:finish proof of pos prop thm}. It will be convenient in a number of places that $b_m$ is supported on square-free integers.

We outline our strategy for estimating $S_1$ and $S_2$. We simplify the presentation here in comparison to the actual proofs. The sum $S_1$ is by far the simpler of the two, so we start here (see Section \ref{sec:mollified first moment}). Using an approximate functional equation for the central value $L(\frac{1}{2},\chi_p)$ (Lemma \ref{lem: approx func eq}), we write $S_1$ as
\begin{align*}
S_1 &\approx \sum_{m \leq M} \frac{b_m}{\sqrt{m}} \sum_{k \leq X^{1/2+\varepsilon}} \frac{1}{\sqrt{k}} \sum_{\substack{p \equiv 1 \, (\text{mod }8)}} (\log p)\Phi \left( \frac{p}{X}\right)\chi_p(mk).
\end{align*}
The main term arises from the ``diagonal'' terms $mk = \square$. The character values $\chi_p(mk)$ are then all equal to one, and we simply use the prime number theorem in arithmetic progressions modulo eight to handle the sum on $p$. The sum on $k$ contributes a logarithmic factor, but this logarithmic loss is canceled out by a logarithmic gain coming from a cancellation in the mollifier coefficients. This yields the main term for $S_1$, which is of size $\asymp X$ (Proposition \ref{prop: asymptotic for S1}).

The ``off-diagonal'' terms $mk \neq \square$ contribute only to the error term. After some manipulations the off-diagonal terms are essentially of the form
\begin{align*}
\mathcal{E} := \sum_{\substack{q \leq MX^{1/2+\varepsilon} \\ q \neq \square}} \frac{\alpha(q)}{q^{\frac{1}{2}}} \sum_{\substack{p}} (\log p)\Phi \left( \frac{p}{X}\right)\chi_q(p),
\end{align*}
where $\alpha(q)$ is some function satisfying $|\alpha(q)| \ll_\varepsilon q^\varepsilon$. We assume here for simplicity that all of the characters $\chi_q$ are primitive characters. We bound the character sum over primes in $\mathcal{E}$ in three different ways, depending on the size of $q$. These three regimes  correspond to small, medium, and large values of $q$. Some of the arguments are similar to those of Jutila \cite{Jut81}.

In the regime of small $q$ we appeal to the prime number theorem in arithmetic progressions with error term. The sum on primes $p$ is small, except in the case where one of the characters $\chi_{q^*}$ is exceptional: that is, the associated $L$-function $L(s,\chi_{q^*})$ has a real zero $\beta_*$ very close to $s=1$. Siegel's theorem gives $q^* \geq c(B) (\log X)^B$ with $B>0$ arbitrarily large. This would immediately dispatch any exceptional characters, but unfortunately the constant $c(B)$ is not effectively computable. To get an effective estimate we use Page's theorem, which states that at most one such exceptional character $\chi_{q^*}$ exists. We then study carefully the contribution of this one exceptional character and show it is acceptably small.

In regimes of medium and large $q$, we take advantage of the averaging over $q$ present in $\mathcal{E}$. We bound $\mathcal{E}$ in terms of instances of
\begin{align*}
\mathcal{E}(Q) := Q^{-\frac{1}{2}+\varepsilon} \sum_{\substack{Q/2 < q \leq Q \\ q \neq \square}} \left|\sum_{p} (\log p)\Phi \left( \frac{p}{X}\right) \chi_q(p) \right|,
\end{align*}
where $Q$ is of moderate size, or is large. 

When $Q$ is medium-sized, we use the explicit formula to bound $\mathcal{E}(Q)$ by sums over zeros of the $L$-functions $L(s,\chi_q)$. We then use zero-density estimates.

We are left with the task of bounding $\mathcal{E}(Q)$ when $Q$ is large, which means $Q$ is larger than $X^\delta$ for some small, fixed $\delta > 0$. Rather than treating the sum on primes analytically, as we did when $Q$ was small or medium-sized, we treat the sum on primes combinatorially. We use Vaughan's identity to write the character sum over the primes as a linear combination of linear and bilinear sums. The linear sums are handled easily with the P\'olya-Vinogradov inequality. We bound the bilinear sums by appealing to a large sieve inequality for real characters due to Heath-Brown (Lemma \ref{lem: estimates for character sums}).

We now describe our plan of attack for $S_2$ (see Section \ref{sec:mollified second moment}). Recall that
\begin{align*}
S_2 &= \sum_{\substack{p \equiv 1 \, (\text{mod }8)}} (\log p)\Phi \left( \frac{p}{X}\right) L\left(\tfrac{1}{2},\chi_p\right)^2 M(p)^2.
\end{align*}
As we see from Theorem \ref{thm: 2nd moment GRH}, we only barely obtain an asymptotic formula for the second moment
\begin{align*}
\sum_{\substack{p \leq X \\ p \equiv 1 \, (\text{mod }8)}} (\log p) L\left(\tfrac{1}{2},\chi_p\right)^2
\end{align*}
under the assumption of the Generalized Riemann Hypothesis. Thus, it might seem doubtful that one can say anything useful about $S_2$, since the central value $L(\frac{1}{2},\chi_p)^2$ is further twisted by the square of a Dirichlet polynomial. The key idea is that we do not need an asymptotic formula for $S_2$, but only an upper bound of the right order of magnitude (with a good constant). We therefore avail ourselves of sieve methods (see Section \ref{sec:sieves}). By positivity we have
\begin{align*}
S_2 &\leq (\log X) \sum_{\substack{n \equiv 1 \, (\text{mod }8)}} \mu^2(n) \Phi \left( \frac{n}{X}\right) \left(\sum_{d \mid n} \lambda_d \right) L \left( \tfrac{1}{2},\chi_n \right)^2 M(n)^2,
\end{align*}
where
\begin{align*}
\sum_{d \mid n} \lambda_d
\end{align*}
is an upper bound sieve supported on coefficients with $d \leq D$. Since we are now working with ordinary integers instead of prime numbers, the analysis for $S_2$ becomes similar to the second moment problem considered in \cite{Sou00} (see \cite[Section 5]{Sou00}).

We begin by writing
\begin{equation}\label{mu2approx}
\mu^2(n) = N_Y(n) + R_Y(n),
\end{equation}
where
\begin{equation}\label{MYRY}
N_Y(n) := \sum_{\substack{\ell^2 \mid n \\ \ell \leq Y}} \mu(\ell), \ \ \ \ \ R_Y(n) := \sum_{\substack{\ell^2 \mid n \\ \ell > Y}} \mu(\ell),
\end{equation}
and $Y$ is a small power of $X$. The sum
\begin{align*}
\sum_{\substack{n \equiv 1 \, (\text{mod }8)}} \Phi \left( \frac{n}{X}\right) R_Y(n) \left(\sum_{d \mid n} \lambda_d \right) L \left( \tfrac{1}{2},\chi_n \right)^2 M(n)^2
\end{align*}
is an error term, and is shown to be small in a straightforward fashion by applying moment estimates for $L(\frac{1}{2},\chi_n)$ due to Heath-Brown (Lemma \ref{lem: moment estimates}).

The main task is therefore to asymptotically evaluate the sum
\begin{align*}
\sum_{\substack{n \equiv 1 \, (\text{mod }8)}} \Phi \left( \frac{n}{X}\right) N_Y(n) \left(\sum_{d \mid n} \lambda_d \right) L \left( \tfrac{1}{2},\chi_n \right)^2 M(n)^2.
\end{align*}
We use an approximate functional equation to represent the central values $L \left( \frac{1}{2},\chi_n \right)^2$ and arrive at expressions of the form
\begin{align*}
\sum_{\ell \leq Y} \mu(\ell)\sum_{d \leq D} \lambda_d \mathop{\sum \sum}_{m_1,m_2 \leq M} \frac{b_{m_1} b_{m_2}}{\sqrt{m_1m_2}} \sum_{\nu=1}^\infty \frac{d(\nu)}{\sqrt{\nu}} \sum_{\substack{n \equiv 1 \, (\text{mod }8) \\ d \mid n \\ \ell^2 \mid n}} \left(\frac{m_1m_2\nu}{n} \right) \Phi \left( \frac{n}{X}\right) \omega \left( \frac{\nu}{n}\right),
\end{align*}
where $\omega(x)$ is some rapidly decaying smooth function that satisfies $\omega(x) \approx 1$ for small $x$. We then change variables $n = m [d,\ell^2]$.

We use Poisson summation to transform the sum on $m$ into a sum basically of the form
\begin{align*}
\sum_{k \in \mathbb{Z}} \left(\frac{[d,\ell^2]k}{m_1m_2\nu} \right) e \left(\frac{k \overline{[d,\ell^2]m_1m_2\nu}}{8} \right) \hat{F}_\nu \left( \frac{kX}{[d,\ell^2]m_1m_2\nu} \right),
\end{align*}
for some smooth function $F_\nu$. The zero frequency $k = 0$ gives rise to a main term. Since $(\frac{0}{h})= 1$ or 0 depending on whether $h$ is a square, the $k=0$ contribution represents the expected ``diagonal'' contribution from $m_1m_2\nu = \square$. There is an additional, off-diagonal, main term which arises, essentially, from the terms with $[d,\ell^2]k = \square$. We adapt here the delicate off-diagonal analysis of \cite{Sou00}. The situation is complicated by the presence of the additive character $e(\cdot)$, which is not present in \cite{Sou00}. The additive character necessitates a division of the integers $k$ into residue classes modulo 8. We then use Fourier expansion to write the additive character as a linear combination of multiplicative characters. After many calculations the off-diagonal main term arises as a sum of complex line integrals. When we combine the various pieces the integrand becomes an even function, exhibiting a symmetry which none of the pieces separately possessed. This fact proves to be very convenient in the final steps of the main term analysis.

One intriguing feature of the main term in $S_2$ is a kind of ``double mollification''. We must account for the savings coming from the mollifier $M(n)$, but must also account for the savings coming from the sieve weights $\lambda_d$, which act as a sort of mollifier on the natural numbers. It is crucial that we get savings in both places, and therefore our sieve process must be very precise. We find that a variation on the ideas of Selberg (see e.g. \cite[Section 6.5]{IK}) is sufficient.

At length we arrive at an upper bound $S_{2,U}$, say, for $S_2$ of size $S_{2,U}\ll X$. We make an optimal choice of the function $H(x)$ in Section \ref{sec:finish proof of pos prop thm} to maximize the ratio $S_1^2/S_{2,U}$. The resulting mollifier is not the optimal mollifier, but it gives results that are asymptotically equivalent to those attained with the optimal mollifier. This yields Theorem \ref{thm: pos prop}.

To treat other residue classes of $p \pmod{8}$, we make the following changes. First, we change the definition of $\chi_p(\cdot )$ to $\left( \frac{(-1)^{\mathfrak{a}}p}{\cdot}\right)$, where $\mathfrak{a}=0$ if $p\equiv 1\pmod{4}$ and $\mathfrak{a}=1$ if $p\equiv 3\pmod{4}$. Thus $\chi_p$ is still a primitive character of conductor $p$. Second, we use a variant of the approximate functional equation (Lemma~\ref{lem: approx func eq}) with $\omega_j$, defined in \eqref{eq: defn of omega j}, replaced by
$$
\frac{1}{2\pi i} \int_{(c)} \frac{\Gamma \left( \frac{s}{2} + \frac{1+2\mathfrak{a}}{4}\right)^j}{\Gamma \left( \frac{1+2\mathfrak{a}}{4}\right)^j} \left(1 - \frac{\chi_p(2)}{2^{\frac{1}{2}-s}}\right)^j \xi^{-s} W(s) \frac{ds}{s}.
$$
The function $W(s)$ here is $16\left(s^2-\tfrac{1}{4}\right)^2$. Its purpose is to cancel potential poles at $s=\frac{1}{2}$ in the analysis.

\section{Lemmata}\label{sec:lemmas}

We represent the central values of $L$-functions by using an approximate functional equation. We first investigate some properties of the smooth functions which appear in our approximate functional equations. For $j = 1,2$ and $c > 0$, define
\begin{align}\label{eq: defn of omega j}
\omega_j(\xi) &= \frac{1}{2\pi i} \int_{(c)} \frac{\Gamma \left( \frac{s}{2} + \frac{1}{4}\right)^j}{\Gamma \left( \frac{1}{4}\right)^j} \left(1 - \frac{1}{2^{\frac{1}{2}-s}}\right)^j \xi^{-s} \frac{ds}{s}.
\end{align}

\begin{lem}\label{lem: properties of omega j}
Let $j = 1,2$. The function $\omega_j(\xi)$ is real-valued and smooth on $(0,\infty)$. If $\xi > 0$ we have
\begin{align*}
\omega_j(\xi) = \left(1 - \frac{1}{\sqrt{2}}\right)^j + O_\varepsilon(\xi^{\frac{1}{2}-\varepsilon}).
\end{align*}
For any fixed integer $\nu \geq 0$ and $\xi \geq 4\nu+10$, we have
\begin{align*}
\omega_j^{(\nu)}(\xi) \ll (\xi/2)^{\nu+3} \exp\left(-\frac{1}{4}\xi^{\frac{2}{j}} \right) \ll_\nu \exp\left(-\frac{1}{8}\xi^{\frac{2}{j}} \right).
\end{align*}
\end{lem}
\begin{proof}
The proof is similar to \cite[Lemma 2.1]{Sou00}, but we give details for completeness. The function $\omega_j(s)$ is real-valued because the change of variable $\text{Im}(s) \rightarrow -\text{Im}(s)$ shows that $\omega_j$ is equal to its complex conjugate. Moreover, uniform convergence for $\xi$ in compact subintervals of $(0,\infty)$ shows that $\omega_j$ is smooth.

To prove the first estimate of the lemma, move the line of integration in the definition of $\omega_j(\xi)$ to $c = -\frac{1}{2}+\varepsilon$. The pole at $s=0$ contributes $\left(1 - \frac{1}{\sqrt{2}}\right)^j$, and the new integral is $O_\varepsilon(\xi^{\frac{1}{2}-\varepsilon})$.

Let us turn to the last estimate of the lemma. We may suppose $\xi^{\frac{2}{j}} \geq 4\nu + 10$. By differentiation under the integral sign we find
\begin{align*}
\omega_j^{(\nu)}(\xi) = \frac{(-1)^\nu}{2\pi i}\int_{(c)}\frac{\Gamma \left( \frac{s}{2} + \frac{1}{4}\right)^j}{\Gamma \left( \frac{1}{4}\right)^j} \left(1 - \frac{1}{2^{\frac{1}{2}-s}}\right)^j s (s+1) \cdots (s+\nu - 1) \xi^{-s-\nu} \frac{ds}{s}.
\end{align*}
Recall that $|\Gamma(x+iy)| \leq \Gamma(x)$ for $x \geq 1$ and $z \Gamma(z) = \Gamma(z+1)$. Thus, for $c \geq 2$ we obtain
\begin{align*}
|\omega_j^{(\nu)}(\xi)| &\ll \Gamma \left( \frac{c}{2} + \frac{5}{4} + \nu\right)^j \left(1 + \frac{2^c}{\sqrt{2}}\right)^j \xi^{-c-v} \int_{(c)} \frac{1}{|s| |\frac{s}{2} + \frac{1}{4}+\nu|} \prod_{k=0}^{\nu-1} \frac{|s+k|}{|\frac{s}{2}+\frac{1}{4} + k|} |ds| \\
&\ll\Gamma \left( \frac{c}{2} + \frac{5}{4} + \nu\right)^j \left( \frac{2^j}{\xi}\right)^c \left(\frac{2}{\xi} \right)^\nu c^{-1},
\end{align*}
where the implied constants are absolute. By Stirling's formula this is
\begin{align*}
\ll \left(\frac{c+2\nu + 3}{2e} \right)^{\frac{j}{2}(c+2\nu + 3)}\left(\frac{2^j}{\xi} \right)^c \left( \frac{2}{\xi}\right)^\nu.
\end{align*}
We choose $c = \frac{1}{2} \xi^{\frac{2}{j}} - 2\nu - 3$, which we note is $> 2$. Thus, the quantity in question is
\begin{align*}
\ll \left(\frac{\xi}{2}\right)^{\nu+3} \exp\left( - \frac{1}{4} \xi^{\frac{2}{j}}\right),
\end{align*}
as desired.
\end{proof}

We will find it technically convenient to use an approximate functional equation in which the variable of summation is restricted to odd integers.

\begin{lem}\label{lem: approx func eq}
Let $n \equiv 1 \pmod{8}$ be square-free and satisfy $n > 1$. Let $\chi_n(\cdot) = \left(\frac{n}{\cdot}\right)$ denote the real primitive character of conductor $n$. Then for $j = 1,2$ we have
\begin{align*}
L\left( \tfrac{1}{2},\chi_n\right)^j = \frac{2}{\left(1 - \frac{1}{\sqrt{2}}\right)^{2j}} \sum_{\substack{\nu = 1 \\ \nu \textup{ odd}}}^\infty \frac{\chi_n(\nu) d_j(\nu)}{\sqrt{\nu}} \omega_j \left(\nu \left(\frac{\pi}{n} \right)^{j/2} \right) =: \mathcal{D}_j(n).
\end{align*}
\end{lem}
\begin{proof}
The proof follows along standard lines (e.g. \cite[Theorem 5.3]{IK}), but we give a proof since our situation is slightly different.

Let $\Lambda(z,\chi_n) = \left(\frac{n}{\pi}\right)^{z/2} \Gamma\left( \frac{z}{2}\right) L(z,\chi_n)$. Since $n \equiv 1 \pmod{4}$ we have $\chi_n(-1) = 1$, and therefore we have the functional equation (see \cite[Proposition 2.2.24]{Coh}, \cite[Chapter 9]{Dav})
\begin{align*}
\Lambda(z,\chi_n) = \Lambda(1-z,\chi_n).
\end{align*}
Recall also that $\Lambda(z,\chi_n)$ is entire because $\chi_n$ is primitive.

Now consider the sum
\begin{align*}
I := \sum_{\nu \text{ odd}} \frac{\chi_n(\nu) d_j(\nu)}{\sqrt{\nu}} \omega_j \left(\nu \left(\frac{\pi}{n} \right)^{j/2} \right).
\end{align*}
We use the definition of $\omega_j$ and interchange the order of summation and integration. Since $\chi_n(2) = 1$ we have
\begin{align*}
I &= \frac{1}{2\pi i}\int_{(c)} \frac{\Gamma \left( \frac{s}{2} + \frac{1}{4}\right)^j}{\Gamma \left( \frac{1}{4}\right)^j} \left(1 - \frac{1}{2^{\frac{1}{2}-s}}\right)^j\left(1 - \frac{1}{2^{\frac{1}{2}+s}}\right)^j \left(\frac{n}{\pi} \right)^{js/2} L\left( \frac{1}{2}+s,\chi_n\right)^j \frac{ds}{s} \\
&= \frac{1}{2\pi i} \int_{(c)} \frac{(\frac{n}{\pi})^{-j/4}}{\Gamma \left( \frac{1}{4}\right)^j}\left(1 - \frac{1}{2^{\frac{1}{2}-s}}\right)^j\left(1 - \frac{1}{2^{\frac{1}{2}+s}}\right)^j \Lambda \left( \frac{1}{2}+s,\chi_n\right)^j \frac{ds}{s}.
\end{align*}
We move the line of integration to $\text{Re}(s) = -c$, picking up a contribution from the simple pole at $s=0$:
\begin{align*}
I &= \frac{(\frac{n}{\pi})^{-j/4}}{\Gamma \left( \frac{1}{4}\right)^j}\left(1 - \frac{1}{\sqrt{2}}\right)^{2j} \Lambda\left( \frac{1}{2},\chi_n\right)^j \\ 
&+ \frac{1}{2\pi i} \int_{(-c)} \frac{(\frac{n}{\pi})^{-j/4}}{\Gamma \left( \frac{1}{4}\right)^j}\left(1 - \frac{1}{2^{\frac{1}{2}-s}}\right)^j\left(1 - \frac{1}{2^{\frac{1}{2}+s}}\right)^j \Lambda \left( \frac{1}{2}+s,\chi_n\right)^j \frac{ds}{s}.
\end{align*}
In this latter integral we change variables $s \rightarrow -s$ and then apply the functional equation $\Lambda\left(\frac{1}{2}-s,\chi_n\right)=\Lambda\left(\frac{1}{2}+s,\chi_n\right)$ to obtain
\begin{align*}
\frac{(\frac{n}{\pi})^{-j/4}}{\Gamma \left( \frac{1}{4}\right)^j}\left(1 - \frac{1}{\sqrt{2}}\right)^{2j} \Lambda\left( \frac{1}{2},\chi_n\right)^j &= 2I = 2\sum_{\nu \text{ odd}} \frac{\chi_n(\nu) d_j(\nu)}{\sqrt{\nu}} \omega_j \left(\nu \left(\frac{\pi}{n} \right)^{j/2} \right).
\end{align*}
We then rearrange to obtain the desired conclusion.
\end{proof}

We frequently encounter exponential sums which are analogous to Gauss sums. Given an odd integer $n$, we define for all integers $k$
\begin{equation}\label{eq: defn of Gk}
G_k(n) = \left(\frac{1-i}{2} + \left( \frac{-1}{n}\right) \frac{1+i}{2} \right)\sum_{a(\text{mod } n)} \left( \frac{a}{n} \right) e \left( \frac{ak}{n}\right)
\end{equation}
and
\begin{equation}\label{Gaussdef}
\tau_k(n) =\sum_{a(\text{mod } n)} \left( \frac{a}{n} \right) e \left( \frac{ak}{n}\right) = \left(\frac{1+i}{2} + \left( \frac{-1}{n}\right) \frac{1-i}{2} \right) G_k(n).
\end{equation}
We require knowledge of $G_k(n)$ for all $n$.

\begin{lem}\label{lem: properties of Gkn}
(i) (Multiplicativity) Suppose $m$ and $n$ are coprime odd integers. Then $G_k(mn) = G_k(m) G_k(n)$. \\
\noindent (ii) Suppose $p^\alpha$ is the largest power of $p$ dividing $k$. (If $k=0$ set $\alpha = \infty$.) Then for $\beta \geq 1$
\begin{align*}
G_k(p^\beta) = 
\begin{cases}
0 \ \ \ \ \ \ \ \ \ \ &\text{if } \beta \leq\alpha \text{ is odd}, \\
\varphi(p^\beta) &\text{if } \beta \leq \alpha \text{ is even}, \\
-p^\alpha &\text{if } \beta=\alpha+1 \text{ is even}, \\
(\frac{kp^{-\alpha}}{p}) p^\alpha\sqrt{p} &\text{if } \beta = \alpha+1 \text{ is odd}, \\
0 &\text{if } \beta \geq \alpha + 2.
\end{cases}
\end{align*}
\end{lem}
\begin{proof}
This is \cite[Lemma 2.3]{Sou00}.
\end{proof}

The following two results are useful for bounding various character sums that arise. Both results are corollaries of a large sieve inequality for quadratic characters developed by Heath-Brown \cite{HeaB95}.

\begin{lem}\label{lem: estimates for character sums}
Let $N$ and $Q$ be positive integers, and let $a_1,\ldots,a_N$ be arbitrary complex numbers. Then
\begin{align*}
\sum_{\chi \in S(Q)} \left| \sum_{n \leq N} a_n \chi(n) \right|^2 \ll_\varepsilon (QN)^\varepsilon (Q+N) \sum_{n_1n_2 = \square} |a_{n_1}a_{n_2}|,
\end{align*}
for any $\varepsilon > 0$. Let $M$ be a positive integer, and for each $|m| \leq M$ write $4m = m_1m_2^2$, where $m_1$ is a fundamental discriminant, and $m_2$ is positive. Suppose the sequence $a_n$ satisfies $|a_n| \ll n^\varepsilon$. Then
\begin{align*}
\sum_{|m| \leq M} \frac{1}{m_2} \left|\sum_{n \leq N} a_n \left( \frac{m}{n}\right) \right|^2 \ll (MN)^\varepsilon N (M+N).
\end{align*}
\end{lem}
\begin{proof}
This is \cite[Lemma 2.4]{Sou00}.
\end{proof}

\begin{lem}\label{lem: moment estimates}
Suppose $\sigma + it$ is a complex number with $\sigma \geq \frac{1}{2}$. Then
\begin{align*}
\sum_{\chi \in S(Q)} \left|L(\sigma+it,\chi) \right|^4 \ll Q^{1+\varepsilon} (1+|t|)^{1+\varepsilon}
\end{align*}
and
\begin{align*}
\sum_{\chi \in S(Q)} \left|L(\sigma+it,\chi) \right|^2 \ll Q^{1+\varepsilon} (1+|t|)^{\frac{1}{2}+\varepsilon}.
\end{align*}
\end{lem}
\begin{proof}
This is \cite[Lemma 2.5]{Sou00}.
\end{proof}

\section{Sieve estimates}\label{sec:sieves}

Our main sieve will be a variant of the Selberg sieve (see \cite[Chapter 7]{FI10}). To lessen the volume of calculations, we also use Brun's pure sieve \cite[Chapter 6]{FI10} as a preliminary sieve to handle small prime factors. We set
\begin{align}\label{eq:sieve section, defn of z0}
z_0 := \exp((\log X)^{1/3})
\end{align}
and
\begin{align}\label{eq:sieve section, defn of R}
R := X^\vartheta, \ \ \ \ \ \ \ \ \ \ \ \vartheta \in \left(0 ,\tfrac{1}{2}\right) \ \  \text{ fixed}.
\end{align}

Given a set $\mathcal{A}$ of integers we write $\mathbf{1}_\mathcal{A}(n)$ for the indicator function of this set. For $y > 2$ we define
\begin{align*}
P(y) = \prod_{p \leq y} p.
\end{align*}
Then, for $n \asymp X$, our basic sieve inequality is
\begin{equation}\label{basicsieveinequality}
\mathbf{1}_{\{n:n\text{ prime}\}} \leq \mathbf{1}_{\{n:(n,P(z_0))=1\}}\mathbf{1}_{\{n:(n,P(R)/P(z_0))=1\}},
\end{equation}

We write $\omega(n)$ for the number of distinct prime factors of $n$. To bound the first factor on the right-hand side of \eqref{basicsieveinequality}, we use Brun's upper bound sieve condition (see \cite[(6.1)]{FI10})
\begin{equation}\label{Brun}
\mathbf{1}_{\{n:(n,P(z_0))=1\}}(n)\leq \sum_{\substack{b|(n,P(z_0)) \\ \omega(b)\leq 2r_0 }} \mu(b),
\end{equation}
where 
\begin{align*}
r_0 := \lfloor (\log X)^{1/3} \rfloor.
\end{align*}
We use an ``analytic'' Selberg sieve (e.g. \cite{Poly14}) for the second factor of \eqref{basicsieveinequality}. We introduce a smooth, non-negative function $G(t)$ which is supported on the interval $[-1,1]$. We further require $G(t)$ to satisfy $|G(t)| \ll 1, |G^{(j)}(t)| \ll_j (\log \log X)^{j-1}$ for $j$ a positive integer, and on the interval $[0,1]$ we require $G(t) = 1-t$ for $t \leq 1 - (\log \log X)^{-1}$. Then
\begin{align}\label{Selberg}
\mathbf{1}_{\{n:(n,P(R)/P(z_0))=1\}}(n) &\leq \Bigg(\sum_{\substack{d \mid n \\ (d,P(z_0))=1}} \mu(d) G \left( \frac{\log d}{\log R}\right) \Bigg)^2 \\
&=\mathop{\sum\sum}_{\substack{j,k\leq R \\ [j,k]|n \\ (jk,P(z_0))=1 }} \mu(j)\mu(k)G\left( \frac{\log j}{\log R}\right)G\left( \frac{\log k}{\log R}\right). \nonumber
\end{align}
We mention also that the properties of $G$ imply
\begin{align}\label{eq:upper bound on int of G' squared}
\int_0^\infty G'(t)^2 dt = 1 + O \left( \frac{1}{\log \log X}\right)= 1+o(1).
\end{align}
Note that the fundamental theorem of calculus and Cauchy-Schwarz yield the lower bound
\begin{align*}
\int_0^\infty G'(t)^2 dt \geq 1.
\end{align*}

From \eqref{basicsieveinequality}, \eqref{Brun}, and \eqref{Selberg}, we arrive at the upper bound sieve condition
\begin{equation}\label{sieveinequality}
\mathbf{1}_{\{n:n\text{ prime}\}}(n) \leq \sum_{d|n} \lambda_d,
\end{equation}
where the coefficients $\lambda_d$ are defined by
\begin{equation}\label{lambda}
\lambda_d= \sum_{\substack{b|P(z_0) \\ \omega(b)\leq 2r_0 }} \mathop{\sum\sum}_{\substack{m,n\leq R \\ b[m,n]=d \\ (mn,P(z_0))=1 }} \mu(b)\mu(m)\mu(n)G\left( \frac{\log m}{\log R}\right)G\left( \frac{\log n}{\log R}\right).
\end{equation}
If $b|P(z_0)$ and $\omega(b)\leq 2r_0$, then $b\leq z_0^{2r_0}=\exp(2(\log X)^{2/3})$. Hence $\lambda_d\neq 0$ only for $d\leq D$, where
\begin{equation}\label{Ddef}
D=R^2 \exp(2(\log X)^{2/3}) \ll_{\varepsilon} R^2X^{\varepsilon}.
\end{equation}

In our evaluation of sums involving the sieve coefficients \eqref{lambda} we use the following version of the fundamental lemma of sieve theory (see also \cite[Section 6.5]{FI10}).
\begin{lem}\label{fundamentallemma}
Let $0 <\delta < 1$ be a fixed constant, $r$ a positive integer with $r \asymp (\log X)^{\delta}$, and $z_0$ as in \eqref{eq:sieve section, defn of z0}. Suppose that $g$ is a multiplicative function such that $|g(p)|\ll 1$ uniformly for all primes $p$. Then
\begin{equation*}
\sum_{\substack{b \mid P(z_0) \\ \omega(b)\leq r \\ (b,\ell)=1}}\frac{\mu(b)}{b}g(b) = \prod_{\substack{p\leq z_0\\ p\nmid \ell}}\Bigg( 1-\frac{g(p)}{p}\Bigg) +O\Big( \exp(-r\log \log r )\Big)
\end{equation*}
uniformly for all positive integers $\ell$.
\end{lem}
\begin{proof}
The proof is standard. Complete the sum on the left-hand side by adding to it all the terms with $\omega(b)>r$, dropping by positivity the condition $(b,\ell)=1$. The error introduced in doing so is $\ll \exp(-(1+o(1))r\log r)\ll \exp(-r \log \log r)$ (e.g. \cite[\S 6.3]{IK}). The completed sum is equal to the Euler product on the right-hand side.
\end{proof}

The basic tool in our application of the Selberg sieve is the following lemma.
\begin{lem}\label{Selbergsieve}
Let $z_0=\exp((\log X)^{1/3})$. Let $G$ be as above. Suppose $h$ is a function such that $|h(p)|\ll_\varepsilon p^{-\varepsilon}$ uniformly for all primes $p$. Let $A > 0$ be a fixed real number. Then there exists a function $E_0(X)$, which depends only on $X,G$, and $\vartheta$ (see \eqref{eq:sieve section, defn of R}) with $E_0(X) \rightarrow 0$ as $X \rightarrow \infty$, such that
\begin{equation}\label{Selbergsieve2}
\begin{split}
\mathop{\sum\sum}_{\substack{m,n \leq R \\ (mn, \ell P(z_0))=1}}
&  \frac{\mu(m)\mu(n)}{[m,n]} \ G\left( \frac{\log m}{\log R}\right)G\left( \frac{\log n}{\log R}\right)\prod_{p|mn}\Big( 1+h(p)\Big)\\
& = \frac{1+E_0(X)}{\log R}\prod_{p\leq z_0}\left( 1-\frac{1}{p}\right)^{-1} + \ O_{\varepsilon,A}\left( \frac{1}{(\log R)^{A}}\right),
\end{split}
\end{equation}
uniformly for $\ell \ll X^{O(1)}$.
\end{lem}
\begin{proof}
Let $\mathcal{S}$ denote the left-hand side of \eqref{Selbergsieve2}. If $m,n\leq R$ and $(mn,P(z_0))=1$, then $\omega(mn) \ll \log R$, and each prime dividing $mn$ is larger than $z_0$. Thus
$$
\prod_{p|mn}\Big( 1+h(p)\Big)=1+O_\varepsilon\Bigg( \frac{\log R}{z_0^{\varepsilon}}\Bigg),
$$
and so
\begin{equation}\label{Selbergsieve3}
\mathcal{S}=\mathop{\sum\sum}_{\substack{m,n \leq R \\ (mn, \ell P(z_0))=1}} \frac{\mu(m)\mu(n)}{[m,n]} \ G\left( \frac{\log m}{\log R}\right)G\left( \frac{\log n}{\log R}\right) + O\Bigg( \frac{(\log R)^4}{z_0^{\varepsilon}}\Bigg).
\end{equation}
We may ignore the condition $(mn,\ell)=1$ in \eqref{Selbergsieve3} because
$$
\mathop{\sum\sum}_{\substack{m,n \leq R \\ (mn, P(z_0))=1 \\ (mn,\ell)>1 }} \frac{1}{[m,n]} \leq \mathop{\sum\sum}_{\substack{m,n \leq R \\ (mn, P(z_0))=1 }} \frac{1}{[m,n]}\sum_{\substack{p|\ell \\ p|mn}} 1 \ll (\log R)^3 \sum_{\substack{p|\ell \\ p>z_0}} \frac{1}{p} \ll \frac{(\log\ell)(\log R)^3}{z_0}.
$$
We next insert the Fourier inversion formula
\begin{equation}\label{Fourierinversion1}
G(t) \ = \ \int_{-\infty}^{\infty} g(z) e^{-t(1+iz)}\,dz
\end{equation}
into \eqref{Selbergsieve3}, where
\begin{equation}\label{Fourierinversion2}
g(z) \ = \ \int_{-\infty}^{\infty} e^t G(t) e^{izt}\,dt.
\end{equation}
We then interchange the order of summation and integration and write the sum as an Euler product to deduce that
\begin{equation}\label{Selbergsieve4}
\mathcal{S}=\int_{-\infty}^{\infty}\int_{-\infty}^{\infty} g(z_1)g(z_2)\prod_{p>z_0} \left(1-\frac{1}{p^{1+\frac{1+iz_1}{\log R}}} -\frac{1}{p^{1+\frac{1+iz_2}{\log R}}} + \frac{1}{p^{1+\frac{2+iz_1+iz_2}{\log R}}} \right)\,dz_1dz_2+ O\Bigg( \frac{(\log R)^4}{z_0^{\varepsilon}}\Bigg).
\end{equation}
By integrating \eqref{Fourierinversion2} by parts repeatedly we see
\begin{align*}
g(z) \ll_A \left(\frac{\log\log X}{1+|z|}\right)^A,
\end{align*}
and we have the trivial bound
\begin{align*}
\prod_{p>z_0} \left(1-\frac{1}{p^{1+\frac{1+iz_1}{\log R}}} -\frac{1}{p^{1+\frac{1+iz_2}{\log R}}} + \frac{1}{p^{1+\frac{2+iz_1+iz_2}{\log R}}} \right) \ll (\log R)^{O(1)}.
\end{align*}
Therefore, we may truncate the double integral in \eqref{Selbergsieve4} to the region $|z_1|,|z_2|\leq \sqrt{\log R}$, with an error of size $O_A((\log R)^{-A})$. After doing so, we multiply and divide the integrand by Euler products of zeta-functions to arrive at
\begin{equation}\label{Selbergsieve5}
\begin{split}
\mathcal{S}=
& \mathop{\int \int}_{|z_i| \leq \sqrt{\log R}} g(z_1)g(z_2)\frac{\zeta\left(1+\frac{2+iz_1+iz_2}{\log R}\right)}{\zeta\left(1+\frac{1+iz_1}{\log R}\right) \zeta\left(1+\frac{1+iz_2}{\log R}\right)} \\
& \times \prod_{p\leq z_0} \frac{\left(1-\frac{1}{p^{1+\frac{2+iz_1+iz_2}{\log R}}} \right)}{\left(1-\frac{1}{p^{1+\frac{1+iz_1}{\log R}}} \right)\left(1-\frac{1}{p^{1+\frac{1+iz_2}{\log R}}} \right)} \prod_{p>z_0} \Bigg(1 + O \Bigg(\frac{1}{p^2}\Bigg) \Bigg)  \,dz_1dz_2+ O\Bigg( \frac{1}{(\log R)^{A}}\Bigg).
\end{split}
\end{equation}
The product over primes $p>z_0$ in \eqref{Selbergsieve5} is $1+O(1/z_0)$. To estimate the product over $p\leq z_0$, observe that if $|s|\ll \sqrt{\log R}$, then
\begin{equation*}
\sum_{p\leq z_0}\frac{1}{p-1}\left( 1-p^{-s}\right) \ll \sum_{p\leq z_0}\frac{|s| \log p}{p} \ll  |s|\log z_0 \ll  \frac{(\log X)^{1/3}}{(\log R)^{1/2}},
\end{equation*}
which implies that
\begin{equation*}
\begin{split}
\prod_{p\leq z_0} \left( 1-\frac{1}{p^{1+s}}\right) 
& = \exp\left( \sum_{p\leq z_0}\log\left( 1+\frac{1}{p-1}\left( 1-p^{-s}\right)\right)\right)\prod_{p\leq z_0}\left( 1-\frac{1}{p}\right)\\
& = \left( 1+O\left( \frac{(\log X)^{1/3}}{(\log R)^{1/2}}\right)\right)\prod_{p\leq z_0}\left( 1-\frac{1}{p}\right). 
\end{split}
\end{equation*}
We may also expand each zeta-function in \eqref{Selbergsieve5} into its Laurent series. With these approximations, we deduce from \eqref{Selbergsieve5} that
\begin{equation*}
\begin{split}
\mathcal{S}= \frac{1}{\log R}\prod_{p\leq z_0} \Bigg( 1-\frac{1}{p}\Bigg)^{-1}\mathop{\int\int}_{|z_i| \leq \sqrt{\log R}} g(z_1)g(z_2)
& \frac{(1+iz_1)(1+iz_2)}{2+iz_1+iz_2}  \left(1 + E(X,\vartheta,z_1,z_2) \right) \,dz_1dz_2\\
& + O\Big( (\log R)^{-A}\Big),
\end{split}
\end{equation*}
uniformly for $\log \ell\ll \log X$. Here $E(X,\vartheta,z_1,z_2)$ tends to zero as $X \rightarrow \infty$. By the rapid decay of $g(z)$, we may extend the range of integration to $\mathbb{R}^2$ without affecting our bound for the error term. By differentiating \eqref{Fourierinversion1} under the integral sign and Fubini's theorem, we find
\begin{align}\label{eq:sieve section, integral of G prime squared}
\mathop{\int \int}_{\mathbb{R}^2} g(z_1)g(z_2) \frac{(1+iz_1)(1+iz_2)}{2+iz_1+iz_2} dz_2 dz_1 = \int_0^\infty G'(t)^2 dt.
\end{align}
The lemma now follows from \eqref{eq:sieve section, integral of G prime squared} and \eqref{eq:upper bound on int of G' squared}.
\end{proof}

\begin{lem}\label{sieve}
Let $\lambda_d$ and $D$ be as defined in \eqref{lambda} and \eqref{Ddef}, respectively. Suppose that $g$ is a multiplicative function such that $g(p)=1+O(p^{-\varepsilon})$ for all primes $p$. Then with $E_0(X)$ as in Lemma \ref{Selbergsieve} we have
\begin{equation*}
\begin{split}
\sum_{\substack{d\leq D\\ (d,\ell)=1}}\frac{\lambda_d}{d} g(d) = \frac{1+E_0(X)}{\log R}\prod_{\substack{p\leq z_0\\ p\nmid \ell}}\Bigg( 1-\frac{g(p)}{p}\Bigg)\prod_{p\leq z_0}\left( 1-\frac{1}{p}\right)^{-1} + O_\varepsilon\left( \frac{1}{(\log R)^{2018}}\right),
\end{split}
\end{equation*}
uniformly in $\ell \ll X^{O(1)}$.
\end{lem}
\begin{proof}
The definitions \eqref{lambda} and \eqref{Ddef} of $\lambda_d$ and $D$ imply
$$
\sum_{\substack{d\leq D\\ (d,\ell)=1}}\frac{\lambda_d}{d} g(d) =\sum_{\substack{b \mid P(z_0) \\ \omega(b)\leq 2r_0 \\ (b,\ell)=1}}
\mathop{\sum\sum}_{\substack{m,n \leq R \\ (mn, \ell P(z_0))=1}}\frac{\mu(b)\mu(m)\mu(n)}{b[m,n]} \ G\left( \frac{\log m}{\log R}\right)G\left( \frac{\log n}{\log R}\right)g(b[m,n]).
$$ 
In the sum on the right-hand side, $g(b[m,n])=g(b)g([m,n])$ because $b$ and $mn$ are coprime. Thus we may apply Lemma~\ref{Selbergsieve} and then Lemma~\ref{fundamentallemma} to arrive at Lemma~\ref{sieve}.
\end{proof}

\begin{lem}\label{sievewithsum}
Let $\lambda_d,D,g$ be as in Lemma~\ref{sieve}. Suppose that $h$ is a function such that $|h(p)|\ll_\varepsilon p^{-1+\varepsilon}$ for all primes $p$. Then with $E_0(X)$ as in Lemma \ref{Selbergsieve} we have
\begin{equation*}
\begin{split}
\sum_{\substack{d\leq D\\ (d,\ell)=1}}\frac{\lambda_d}{d} g(d)\sum_{p|d} h(p)=
& - \frac{1+E_0(X)}{\log R}\prod_{p\leq z_0}\left( 1-\frac{1}{p}\right)^{-1} \\
& \ \times \sum_{\substack{p\leq z_0 \\ p\nmid \ell}} \frac{g(p)h(p)}{p} \prod_{\substack{q\leq z_0\\ q\nmid p\ell}}\Bigg( 1-\frac{g(q)}{q}\Bigg) + O_\varepsilon\left( \frac{1}{(\log R)^{2018}}\right),
\end{split}
\end{equation*}
uniformly for all integers $\ell$ such that $\log \ell\ll \log X$. (Here, the index $q$ runs over primes $q$.)
\end{lem}
\begin{proof}
The definitions \eqref{lambda} and \eqref{Ddef} of $\lambda_d$ and $D$ imply
\begin{equation*}
\begin{split}
\sum_{\substack{d\leq D\\ (d,\ell)=1}}\frac{\lambda_d}{d} g(d)\sum_{p|d} h(p)=\sum_{\substack{b \mid P(z_0) \\ \omega(b)\leq 2r_0 \\ (b,\ell)=1}}
\mathop{\sum\sum}_{\substack{m,n \leq R \\ (mn, \ell P(z_0))=1}}
& \frac{\mu(b)\mu(m)\mu(n)}{b[m,n]} \ G\left( \frac{\log m}{\log R}\right)G\left( \frac{\log n}{\log R}\right)\\
& \ \ \times g(b[m,n])\sum_{p|bmn} h(p).
\end{split}
\end{equation*}
Since $b$ and $mn$ are coprime, $g(b[m,n])=g(b)g([m,n])$ and
$$
\sum_{p|bmn} h(p)=\sum_{p|b} h(p)+\sum_{p|mn} h(p).
$$
We may ignore the sum over the $p|mn$ because the conditions $(mn,P(z_0))=1$ and $mn\leq R^2$ imply
$$
\sum_{p|mn} h(p)\ll \sum_{p|mn}p^{-1+\varepsilon}\ll \frac{\log R}{z_0^{1-\varepsilon}}.
$$
We factor out $g(b)$ and $\sum_{p|b} h(p)$ from the sum over $m,n$ and then apply Lemma~\ref{Selbergsieve} to deduce that
\begin{equation}\label{sievewithsum2}
\begin{split}
\sum_{\substack{d\leq D\\ (d,\ell)=1}}\frac{\lambda_d}{d} g(d)\sum_{p|d} h(p)=
& \frac{1+E_0(X)}{\log R}\prod_{p\leq z_0}\left( 1-\frac{1}{p}\right)^{-1}\\
& \ \times \sum_{\substack{b \mid P(z_0) \\ \omega(b)\leq 2r_0 \\ (b,\ell)=1}}\frac{\mu(b)}{b}g(b) \sum_{p|b} h(p) + O\left( \frac{1}{(\log R)^{2018}}\right).
\end{split}
\end{equation}
To estimate the $b$-sum, we interchange the order of summation and then relabel $b$ as $bp$ to write
$$
\sum_{\substack{b \mid P(z_0) \\ \omega(b)\leq 2r_0 \\ (b,\ell)=1}}\frac{\mu(b)}{b}g(b) \sum_{p|b} h(p) = \sum_{\substack{p\leq z_0 \\ p\nmid \ell}} h(p) \sum_{\substack{b \mid P(z_0) \\ \omega(b)\leq 2r_0 \\ (b,\ell)=1 \\ p|b }}\frac{\mu(b)}{b}g(b) = - \sum_{\substack{p\leq z_0 \\ p\nmid \ell}} \frac{g(p)h(p)}{p} \sum_{\substack{b \mid P(z_0) \\ \omega(b)\leq 2r_0-1 \\ (b,p\ell)=1 }}\frac{\mu(b)}{b}g(b).
$$
Lemma~\ref{sievewithsum} now follows from  Lemma~\ref{fundamentallemma} and \eqref{sievewithsum2}.
\end{proof}

\section{The mollified first moment}\label{sec:mollified first moment}

Our goal in this section is to asymptotically evaluate $S_1$. Recall from \eqref{eq: defn moment sums S1 and S2} that
\begin{align*}
S_1 &=\sum_{\substack{p \equiv 1 \, (\text{mod }8)}} (\log p)\Phi \left( \frac{p}{X}\right) L\left(\tfrac{1}{2},\chi_p\right) M(p).
\end{align*}
Recall the definition of $M(p)$ from \eqref{eq: defn of mollifier}, and the choice \eqref{eq: defn of mollifier coeffs bm} we made for the mollifier coefficients $b_m$. We shall prove the following result.

\begin{prop}\label{prop: asymptotic for S1}
Let $0 < \theta < \frac{1}{2}$ be fixed. If $X \geq X_0(\theta)$, then
\begin{align*}
S_1 &= \frac{1}{2(1 - \frac{1}{\sqrt{2}})} \left(H(0) - \frac{1}{2\theta}H'(0) \right) \frac{X}{4}  + O \left(\frac{X}{(\log X)^{1-\varepsilon}} \right).
\end{align*}
The implied constant in the error term is effectively computable.
\end{prop}

Let us begin in earnest, following the outline in Section~\ref{sec: outline of pos prop thm}. We apply Lemma~\ref{lem: approx func eq} to write $L(\frac{1}{2},\chi_p)$ as a Dirichlet series. We insert the definition of $M(p)$ and obtain
\begin{align*}
S_1 &= \frac{2}{\left( 1-\frac{1}{\sqrt{2}}\right)^2}\sum_{\substack{m\leq M\\ m \text{ odd}}} \frac{b_m}{\sqrt{m}} \sum_{\substack{n=1 \\ n \text{ odd}}}^{\infty} \frac{1}{\sqrt{n}}   \sum_{p\equiv 1 \, (\text{mod }8) } (\log p) \Phi\left(\frac{p}{X}\right)\omega_1\left( n\sqrt{\frac{\pi}{p}}\right) \left(\frac{mn}{p}\right).
\end{align*}
The main term arises from the terms with $mn = \square$. Let us denote this portion of $S_1$ by $S_1^{\square}$. We denote the complementary portion with $mn \neq \square$ by $S_1^{\neq}$. Therefore 
\begin{align*}
S_1 = S_1^\square + S_1^{\neq},
\end{align*}
where
\begin{equation}\label{eq: splitting up first moment into diagonal and off diagonal}
\begin{split}
S_1^\square &= \frac{2}{\left( 1-\frac{1}{\sqrt{2}}\right)^2}\mathop{\sum_{\substack{m\leq M\\ m \text{ odd}}} \sum_{\substack{n=1 \\ n \text{ odd}}}^{\infty}}_{mn = \square} \frac{b_m}{\sqrt{m}} \frac{1}{\sqrt{n}}   \sum_{p\equiv 1 \, (\text{mod }8) } (\log p) \Phi\left(\frac{p}{X}\right)\omega_1\left( n\sqrt{\frac{\pi}{p}}\right) \left(\frac{mn}{p}\right), \\
S_1^{\neq} &= \frac{2}{\left( 1-\frac{1}{\sqrt{2}}\right)^2}\mathop{\sum_{\substack{m\leq M\\ m \text{ odd}}} \sum_{\substack{n=1 \\ n \text{ odd}}}^{\infty}}_{mn \neq \square} \frac{b_m}{\sqrt{m}} \frac{1}{\sqrt{n}}   \sum_{p\equiv 1 \, (\text{mod }8) } (\log p) \Phi\left(\frac{p}{X}\right)\omega_1\left( n\sqrt{\frac{\pi}{p}}\right) \left(\frac{mn}{p}\right).
\end{split}
\end{equation}
We treat first the main term $S_1^\square$, and later we will bound the error term $S_1^{\neq}$.

\subsection{Main term}

Recall that $b_m$ is supported on square-free integers $m$. Therefore, $mn = \square$ if and only if $n = mk^2$, where $k$ is a positive integer. We make this change of variables and then interchange orders of summation to obtain
\begin{align*}
S_1^{\square} = \frac{2}{\left( 1-\frac{1}{\sqrt{2}}\right)^2}\sum_{p\equiv 1 \, (\text{mod }8) } (\log p) \Phi\left(\frac{p}{X}\right) \sum_{\substack{m\leq M\\ (m,2p)=1}} \frac{b_m}{m} \sum_{\substack{k=1 \\ (k,2p)=1 }}^{\infty} \frac{1}{k}   \omega_1\left( mk^2\sqrt{\frac{\pi}{p}}\right).
\end{align*}
By the rapid decay of $\omega_1$ (Lemma \ref{lem: properties of omega j}) we see that the contribution from those $k$ with $(k,p) > 1$ is $O_A(X^{-A})$, so we may safely ignore this condition. We may also ignore the condition $(m,p)=1$, since $m \leq M < p$. We insert the definition \eqref{eq: defn of omega j} of $\omega_1(\xi)$ and interchange to deduce that for any $c > 0$ we have
\begin{align*}
&\sum_{\substack{k=1 \\ (k,2)=1 }}^{\infty} \frac{1}{k} \omega_1\left( mk^2\sqrt{\frac{\pi}{p}}\right)\\
& \ = \  \frac{1}{2\pi i} \int_{(c)} \frac{\Gamma(\frac{s}{2}+\frac{1}{4})}{\Gamma(\frac{1}{4})} \left( 1-\frac{1}{2^{\frac{1}{2}-s}}\right)  \left( 1-\frac{1}{2^{1+2s}}\right) \zeta(1+2s) \left( \frac{p}{\pi}\right)^{s/2}m^{-s}\,\frac{ds}{s}.
\end{align*}
We move the line of integration to Re$\,s=-\frac{1}{2}+\varepsilon$, leaving a residue at $s=0$. The new integral is $O_{\varepsilon}\left(p^{-\frac{1}{4}+\varepsilon} m^{\frac{1}{2}-\varepsilon}\right)$. Using $b_m\ll 1$, we see that the total contribution of this error term is $\ll \ X^{\frac{3}{4}+\varepsilon} M^{\frac{1}{2}}$. This is $O(X^{1-\varepsilon})$ by \eqref{eq:outline section, defn of M, length of mollifier}. Writing the residue at $s=0$ as an integral along a small circle around $0$, we deduce that
\begin{equation}\label{eq: first moment S 1 square}
\begin{split}
S_1^\square = O(X^{1-\varepsilon})&+ \frac{2}{\left( 1-\frac{1}{\sqrt{2}}\right)^2}\sum_{p\equiv 1 \, (\text{mod }8) } (\log p) \Phi\left(\frac{p}{X}\right) \sum_{\substack{m\leq M\\ (m,2)=1}} \frac{b_m}{m} \\
&\times \frac{1}{2\pi i} \oint_{|s|=\frac{1}{2\log X}} \frac{\Gamma(\frac{s}{2}+\frac{1}{4})}{\Gamma(\frac{1}{4})} \left( 1-\frac{1}{2^{\frac{1}{2}-s}}\right)  \left( 1-\frac{1}{2^{1+2s}}\right) \zeta(1+2s) \left( \frac{p}{\pi}\right)^{s/2}m^{-s}\,\frac{ds}{s}.
\end{split}
\end{equation}

We next use the definition $b_m=\mu(m)H\left( \frac{\log m}{\log M}\right)$ and the Fourier inversion formula (compare with \eqref{Fourierinversion1},\eqref{Fourierinversion2})
\begin{align}\label{eq: write H as Fourier integral}
H(t) \ = \ \int_{-\infty}^{\infty} h(z) e^{-t(1+iz)}\,dz,
\end{align}
where
\begin{align}\label{eq: write h as Fourier transform}
h(z) \ = \ \int_{-\infty}^{\infty} e^t H(t) e^{izt}\,dt,
\end{align}
to write
\begin{align*}
& \sum_{\substack{m\leq M\\ (m,2)=1}} \frac{b_m}{m}  m^{-s} \ = \ \int_{-\infty}^{\infty} h(z) \sum_{\substack{m=1\\ (m,2)=1}}^{\infty} \frac{\mu(m)}{m^{1+s+\frac{1+iz}{\log M}}}\,dz \\
& \ = \ \int_{-\infty}^{\infty} h(z)\left(1-\frac{1}{2^{1+s+\frac{1+iz}{\log M}}} \right)^{-1}  \zeta^{-1}\left( 1+s+\frac{1+iz}{\log M}\right) \,dz.
\end{align*}
From repeated integration by parts we obtain
\begin{align}\label{eq: h has rapid decay}
h(z) \ll_j \frac{1}{(1+|z|)^j},
\end{align}
and therefore we may truncate this integral to the range $|z|\leq \sqrt{\log M}$. Thus,
\begin{align*}
\sum_{\substack{m\leq M\\ (m,2)=1}} \frac{b_m}{m}  m^{-s} \ = \ \int_{|z| \leq \sqrt{\log M}} &h(z)\left(1-\frac{1}{2^{1+s+\frac{1+iz}{\log M}}} \right)^{-1}  \zeta^{-1}\left( 1+s+\frac{1+iz}{\log M}\right) \,dz \ \ \\ 
&+ \ O_A\left(\frac{1}{(\log X)^A}\right).
\end{align*}
For $|s|=\frac{1}{2\log X}$ and $|z|\leq \sqrt{\log M}$, we may write $\left(1-\frac{1}{2^{1+s+\frac{1+iz}{\log M}}} \right)^{-1}  \zeta^{-1}\left( 1+s+\frac{1+iz}{\log M}\right)$ as a power series and arrive at
\begin{align*}
\sum_{\substack{m\leq M\\ (m,2)=1}} \frac{b_m}{m}  m^{-s} & = \ 2\int_{|z| \leq \sqrt{\log M}} h(z)  \left( s+\frac{1+iz}{\log M}\right) \,dz \ \ + \ O\left( \frac{1}{(\log X)^{2}}\right).
\end{align*}
We may extend the range of integration to the entire real line, with negligible error, because of \eqref{eq: h has rapid decay}. The definition of $H(t)$ implies that
\begin{align*}
H'(t) \ = \ -(1+iz)\int_{-\infty}^{\infty} h(z) e^{-t(1+iz)}\,dz.
\end{align*}
Therefore
\begin{align*}
\int_{-\infty}^{\infty} h(z)  \left( s+\frac{1+iz}{\log M}\right) \,dz \ = \ sH(0)-\frac{1}{\log M} H'(0),
\end{align*}
and hence
\begin{align}\label{eq: first moment m sum main term}
\sum_{\substack{m\leq M\\ (m,2)=1}} \frac{b_m}{m}  m^{-s} \ = \ 2sH(0)-\frac{2}{\log M} H'(0) \ \ + \ O\left( \frac{1}{(\log X)^2}\right).
\end{align}

We insert \eqref{eq: first moment m sum main term} into \eqref{eq: first moment S 1 square} to obtain
\begin{align*}
S_1^\square = &\frac{4}{\left( 1-\frac{1}{\sqrt{2}}\right)^2}\sum_{p\equiv 1 \, (\text{mod }8) } (\log p) \Phi\left(\frac{p}{X}\right) \ \frac{1}{2\pi i} \oint_{|s|=\frac{1}{2\log X}} \frac{\Gamma(\frac{s}{2}+\frac{1}{4})}{\Gamma(\frac{1}{4})} \left( 1-\frac{1}{2^{\frac{1}{2}-s}}\right)  \\
&\times  \left( 1-\frac{1}{2^{1+2s}}\right) \zeta(1+2s) \left( \frac{p}{\pi}\right)^{s/2} \left(sH(0)-\frac{1}{\log M} H'(0) \right) \,\frac{ds}{s} \ \ \ + \ O\left( \frac{X}{\log X}\right).
\end{align*}
We evaluate the integral using the formula
\begin{equation}\label{eq: residue as derivative}
\underset{s=0}{\mbox{Res}}\, g(s) = \frac{1}{(n-1)!} \frac{d^{n-1}}{ds^{n-1}} s^n g(s) \Bigg|_{s=0}
\end{equation}
for a pole of a function $g(s)$ at $s=0$ of order at most $n$. This yields
\begin{align*}
S_1^\square &= \ \frac{1}{\left( 1-\frac{1}{\sqrt{2}}\right)}\sum_{p\equiv 1 \, (\text{mod }8) } (\log p) \Phi\left(\frac{p}{X}\right) \ \left(H(0)-\frac{\log p}{2\log M} H'(0) \right) + O\left( \frac{X}{\log X}\right).
\end{align*}
By the support of $\Phi$ we have $\log p = \log X + O(1)$. We then use the prime number theorem in arithmetic progressions and partial summation to obtain
\begin{align}\label{eq: S1 square gives main term}
S_1^\square &= \frac{1}{\left( 1-\frac{1}{\sqrt{2}}\right)}\left(H(0)-\frac{\log X}{2\log M} H'(0) \right) \frac{X}{4}\widehat{\Phi}(0) + O\left( \frac{X}{\log X}\right).
\end{align}
Now \eqref{eq: S1 square gives main term} gives the main term for Proposition \ref{prop: asymptotic for S1}.

\subsection{Preparation of the off-diagonal}

We turn to bounding $S_1^{\neq}$. In order to complete the proof of Proposition \ref{prop: asymptotic for S1}, we prove
\begin{align}\label{eq: first moment desired bound for S1 neq}
S_1^{\neq} \ll \frac{X}{(\log X)^{1-\varepsilon}}.
\end{align}

We need to perform some technical massaging before $S_1^{\neq}$ is in a suitable form. Recall from \eqref{eq: splitting up first moment into diagonal and off diagonal} that
\begin{align*}
S_1^{\neq}=\frac{2}{\left( 1-\frac{1}{\sqrt{2}}\right)^2}\mathop{\sum_{\substack{m\leq M\\ m \text{ odd}}} \sum_{\substack{n=1 \\ n \text{ odd}}}^{\infty}}_{mn \neq \square} \frac{b_m}{\sqrt{mn}}   \sum_{p\equiv 1 \, (\text{mod }8) } (\log p) \Phi\left(\frac{p}{X}\right)\omega_1\left( n\sqrt{\frac{\pi}{p}}\right) \left(\frac{mn}{p}\right).
\end{align*}
We begin by uniquely writing $n = rk^2$, where $r$ is square-free and $k$ is an integer (this variable $k$ is unrelated to the variable $k$ appearing in the analysis for $S_1^\square$). The condition $mn\neq \square$ is equivalent to $m\neq r$, since both $m$ and $r$ are square-free. It follows that
\begin{align*}
S_1^{\neq} &= \frac{2}{\left( 1-\frac{1}{\sqrt{2}}\right)^2}\sum_{\substack{m\leq M\\ m \text{ odd}}} \frac{b_m}{\sqrt{m}} \sum_{\substack{r=1 \\ r \text{ odd} \\ r\neq m }}^{\infty} \sum_{\substack{k=1 \\ k \text{ odd} }}^{\infty} \frac{\mu^2(r)}{k\sqrt{r}}   \sum_{p\equiv 1 \, (\text{mod }8) } (\log p) \Phi\left(\frac{p}{X}\right)\omega_1\left( rk^2\sqrt{\frac{\pi}{p}}\right) \left(\frac{mrk^2}{p}\right)
\end{align*}
We next factor out the greatest common divisor, say $g$, of $m$ and $r$. We change variables $m \rightarrow gm, r \rightarrow gr$ and obtain
\begin{align*}
S_1^{\neq} = \frac{2}{\left( 1-\frac{1}{\sqrt{2}}\right)^2} &\sum_{\substack{g \text{ odd} }}\frac{\mu^2(g)}{g} \sum_{\substack{m\leq M/g\\ (m,2g)=1}} \frac{b_{mg}}{\sqrt{m}}  \sum_{\substack{r=1 \\ (r,2g)=1 \\ (m,r)=1\\ mr>1 }}^{\infty} \frac{\mu^2(r)}{\sqrt{r}}\sum_{\substack{k=1 \\ k \text{ odd} }}^{\infty} \frac{1}{k} \\  
&\times\sum_{p\equiv 1 \, (\text{mod }8) } (\log p) \Phi\left(\frac{p}{X}\right)\omega_1\left( grk^2\sqrt{\frac{\pi}{p}}\right) \left(\frac{mrg^2k^2}{p}\right).
\end{align*}
Observe that the support of $b_{gm}$ forces $g \leq M < X^{\frac{1}{2}}$, but we prefer not to indicate this explicitly.

Clearly we have $\left( \frac{g^2k^2}{p}\right)=1$ for $p\nmid gk$ and $=0$ otherwise. Since $g \leq M<p$ the condition $p \nmid g$ is automatically satisfied. By Lemma \ref{lem: properties of omega j} we may truncate the sum over $k$ to $k \leq X^{\frac{1}{4}+\varepsilon}$ at the cost of an error $O(X^{-1})$, say. We may similarly truncate the sum on $r$ to $r \leq X^{\frac{1}{2}+\varepsilon}$. With $k$ suitably reduced we may drop the condition $p \nmid k$, and then we use the rapid decay of $\omega_1$ again to extend the sum on $k$ to infinity. It follows that
\begin{equation}\label{eq: first moment S1 neq remove coprim conditions}
\begin{split}
S_1^{\neq} = \frac{2}{\left( 1-\frac{1}{\sqrt{2}}\right)^2} &\sum_{\substack{g \text{ odd} }}\frac{\mu^2(g)}{g} \sum_{\substack{m\leq M/g\\ (m,2g)=1}} \frac{b_{mg}}{\sqrt{m}}  \sum_{\substack{r\leq X^{1/2+\varepsilon} \\ (r,2g)=1 \\ (m,r)=1\\ mr>1 }} \frac{\mu^2(r)}{\sqrt{r}}\sum_{\substack{k=1 \\ k \text{ odd} }}^{\infty} \frac{1}{k} \\  
&\times\sum_{p\equiv 1 \, (\text{mod }8) } (\log p) \Phi\left(\frac{p}{X}\right)\omega_1\left( grk^2\sqrt{\frac{\pi}{p}}\right) \left(\frac{mr}{p}\right) + O(X^{-1}).
\end{split}
\end{equation}

We next detect the congruence condition $p \equiv 1 \pmod{8}$ with multiplicative characters modulo 8. Therefore
\begin{equation}\label{eq: first moment detect mod 8 with characters}
\begin{split}
\sum_{p\equiv 1 \, (\text{mod }8) } &(\log p) \Phi\left(\frac{p}{X}\right)\omega_1\left( grk^2\sqrt{\frac{\pi}{p}}\right) \left(\frac{mr}{p}\right) \\ 
&=  \frac{1}{4} \sum_{\gamma \in \{\pm 1, \pm 2\}} \sum_{p} (\log p) \Phi\left(\frac{p}{X}\right)\omega_1\left( grk^2\sqrt{\frac{\pi}{p}}\right) \left(\frac{\gamma mr}{p}\right).
\end{split}
\end{equation}
Since $m$ and $r$ are odd and square-free and $(m,r)=1$, it follows that $mr$ is odd and square-free. Hence, for each $\gamma\in \{1,-1,2,-2\}$, the integer $\gamma  mr$ is square-free. Therefore $\gamma mr\equiv 1$, $2$, or $3$ (mod $4$). If $\gamma mr \equiv 1$ (mod $4$), then $\left(\frac{\gamma mr}{\cdot }\right)$ is a real primitive character modulo $|\gamma mr|$, while if $\gamma mr\equiv 2$ or $3$ (mod $4$), then $\left(\frac{4\gamma mr}{\cdot }\right)$ is a real primitive character modulo $|4\gamma mr|$ (see \cite[Theorem 2.2.15]{Coh}). Moreover, for $p$ odd, $\left(\frac{4\gamma mr}{p }\right)=\left(\frac{\gamma mr}{p }\right)$. Therefore the sum in \eqref{eq: first moment detect mod 8 with characters} is equal to
\begin{align}\label{eq: first moment chi bmr}
\frac{1}{4} \sum_{\gamma \in \{\pm 1, \pm 2\}} \sum_{p} (\log p) \Phi\left(\frac{p}{X}\right)\omega_1\left( grk^2\sqrt{\frac{\pi}{p}}\right) \chi_{\gamma mr}(p),
\end{align}
where $\chi_{\gamma mr}(\cdot)=\left(\frac{\gamma mr}{\cdot }\right)$ if $\gamma mr \equiv 1$ (mod $4$), and $\chi_{\gamma mr}(\cdot)=\left(\frac{4\gamma mr}{\cdot }\right)$ if $\gamma mr\equiv 2$ or $3$ (mod $4$), so that $\chi_{\gamma mr}(\cdot)$ is a real primitive character for all the relevant $\gamma,m,r$. Also, since $mr>1$, we see that $\gamma mr$ is never $1$, so each $\chi_{\gamma mr}$ is nonprincipal.

We insert the definition of $\omega_1$ into \eqref{eq: first moment chi bmr} in order to facilitate a separation of variables. Recalling \eqref{eq: first moment S1 neq remove coprim conditions} and \eqref{eq: first moment detect mod 8 with characters}, we interchange the order of summation and integration to obtain
\begin{align*}
S_1^{\neq} &= O(1)+ \frac{2}{\left( 1-\frac{1}{\sqrt{2}}\right)^2} \sum_{\substack{g \text{ odd} }}\frac{\mu^2(g)}{g} \sum_{\substack{m\leq M/g\\ (m,2g)=1}} \frac{b_{mg}}{\sqrt{m}}  \sum_{\substack{r\leq X^{1/2+\varepsilon} \\ (r,2g)=1 \\ (m,r)=1\\ mr>1 }} \frac{\mu^2(r)}{\sqrt{r}} \frac{1}{4}\sum_{\gamma \in \{\pm 1, \pm 2\}}\sum_{\substack{k=1 \\ k \text{ odd} }}^{\infty} \frac{1}{k} \\
& \times \frac{1}{2\pi i} \int_{(c)} \frac{\Gamma(\frac{s}{2}+\frac{1}{4})}{\Gamma(\frac{1}{4})} \left( 1-\frac{1}{2^{\frac{1}{2}-s}}\right)\pi^{-s/2} \left( grk^2\right)^{-s}  \sum_{p } (\log p) \Phi\left(\frac{p}{X}\right) \chi_{\gamma mr}(p) p^{s/2} \, \frac{ds}{s}.
\end{align*}
We choose $c = \frac{1}{\log X}$, so that $p^{s/2}$ is bounded in absolute value. We can put the summation on $k$ inside of the integral, where it becomes a zeta factor, and we obtain
\begin{equation*}
\begin{split}
S_1^{\neq} &= O(1)+ \frac{2}{\left( 1-\frac{1}{\sqrt{2}}\right)^2} \sum_{\substack{g \text{ odd} }}\frac{\mu^2(g)}{g} \sum_{\substack{m\leq M/g\\ (m,2g)=1}} \frac{b_{mg}}{\sqrt{m}}  \sum_{\substack{r\leq X^{1/2+\varepsilon} \\ (r,2g)=1 \\ (m,r)=1\\ mr>1 }} \frac{\mu^2(r)}{\sqrt{r}} \frac{1}{4}\sum_{\gamma \in \{\pm 1, \pm 2\}} \frac{1}{2\pi i} \int_{(c)} \frac{\Gamma(\frac{s}{2}+\frac{1}{4})}{\Gamma(\frac{1}{4})} \\ 
&\times \left( 1-\frac{1}{2^{\frac{1}{2}-s}}\right)\left(1 - \frac{1}{2^{1+2s}}\right)\zeta(1+2s)\pi^{-s/2} \left( gr\right)^{-s}  \sum_{p } (\log p) \Phi\left(\frac{p}{X}\right) \chi_{\gamma mr}(p) p^{s/2} \, \frac{ds}{s}.
\end{split}
\end{equation*}

It is more convenient to replace the $\log p$ factor with the von Mangoldt function $\Lambda(n)$. By trivial estimation we have
\begin{align*}
\sum_{p } (\log p) \Phi\left(\frac{p}{X}\right) \chi_{\gamma mr}(p) p^{s/2} = \sum_n \Lambda(n) \Phi \left( \frac{n}{X}\right) \chi_{\gamma mr}(n) n^{s/2} + O(X^{1/2}).
\end{align*}
When we sum the error term over $m,g,r$ and integrate over $s$, the total contribution is $O(X^{1-\varepsilon})$, provided $\varepsilon = \varepsilon(\theta) >0$ is sufficiently small. By the rapid decay of the $\Gamma$ function in vertical strips we can truncate the integral to $|\text{Im}(s)| \leq (\log X)^2$, at the cost of a negligible error. We therefore obtain
\begin{equation}\label{eq: first moment S1 neq all done massaging}
\begin{split}
S_1^{\neq} &= O(X^{1-\varepsilon})+\frac{2}{\left( 1-\frac{1}{\sqrt{2}}\right)^2} \sum_{\substack{g \text{ odd} }}\frac{\mu^2(g)}{g} \sum_{\substack{m\leq M/g\\ (m,2g)=1}} \frac{b_{mg}}{\sqrt{m}}  \sum_{\substack{r\leq X^{1/2+\varepsilon} \\ (r,2g)=1 \\ (m,r)=1\\ mr>1 }} \frac{\mu^2(r)}{\sqrt{r}} \frac{1}{4}\sum_{\gamma \in \{\pm 1, \pm 2\}}  \\ 
&\times \frac{1}{2\pi i} \int_{\frac{1}{\log X} - i(\log X)^2}^{\frac{1}{\log X} + i(\log X)^2} \frac{\Gamma(\frac{s}{2}+\frac{1}{4})}{\Gamma(\frac{1}{4})} \left( 1-\frac{1}{2^{\frac{1}{2}-s}}\right)\left(1 - \frac{1}{2^{1+2s}}\right)\zeta(1+2s) \\
&\times \left(\frac{X}{\pi} \right)^{s/2}  \left( gr\right)^{-s} \sum_n \Lambda(n) \Phi_s \left( \frac{n}{X}\right) \chi_{\gamma mr}(n) \, \frac{ds}{s}.
\end{split}
\end{equation}

Having arrived at \eqref{eq: first moment S1 neq all done massaging}, we are finished with the preparatory technical manipulations. We proceed to show that $S_1^{\neq}$ is small. As discussed in Section \ref{sec: outline of pos prop thm}, we apply three different arguments, depending on the size of $mr$. We call these ranges Regimes I, II, and III, which correspond to small, medium, and large values of $mr$. In Regime I we have $1 < mr \ll \exp(\varpi \sqrt{\log x})$, where $\varpi > 0$ is a sufficiently small, fixed constant. Regime II corresponds to $\exp(\varpi \sqrt{\log x}) \ll mr \ll X^{\frac{1}{10}}$, and Regime III corresponds to $X^{\frac{1}{10}} \ll mr \ll M X^{\frac{1}{2}+\varepsilon}$. We then write
\begin{align}\label{eq: first moment decompose S1 neq into E1 and E2}
S_1^{\neq} = E_1 + E_2,
\end{align}
where $E_1$ contains those terms with $mr \ll \exp(\varpi \sqrt{\log x})$, and $E_2$ contains those terms with $mr\gg\exp(\varpi \sqrt{\log x})$. We claim the bounds
\begin{equation}\label{eq: first moment bounds for E1 and E2}
\begin{split}
E_1 &\ll \frac{X}{(\log X)^{1-\varepsilon}}, \\
E_2 &\ll X \exp(-c \varpi \sqrt{\log x}),
\end{split}
\end{equation}
where $c > 0$ is some absolute constant. Taking together \eqref{eq: first moment decompose S1 neq into E1 and E2} and \eqref{eq: first moment bounds for E1 and E2} clearly gives \eqref{eq: first moment desired bound for S1 neq}, and this yields Proposition \ref{prop: asymptotic for S1}. It therefore suffices to show \eqref{eq: first moment bounds for E1 and E2}.

\subsection{Regime I}

We first bound $E_1$, which is precisely the contribution of Regime I. By definition, we have
\begin{equation}\label{eq: first moment defn of E1}
\begin{split}
E_1 &:= \frac{2}{\left( 1-\frac{1}{\sqrt{2}}\right)^2} \sum_{\substack{g \text{ odd} }}\frac{\mu^2(g)}{g} \sum_{\substack{m\leq M/g\\ (m,2g)=1}} \frac{b_{mg}}{\sqrt{m}}  \sum_{\substack{r\leq X^{1/2+\varepsilon} \\ (r,2g)=1 \\ (m,r)=1\\ 1 < mr \ll \exp(\varpi \sqrt{\log x}) }} \frac{\mu^2(r)}{\sqrt{r}} \frac{1}{4}\sum_{\gamma \in \{\pm 1, \pm 2\}}  \\ 
&\times \frac{1}{2\pi i} \int_{\frac{1}{\log X} - i(\log X)^2}^{\frac{1}{\log X} + i(\log X)^2} \frac{\Gamma(\frac{s}{2}+\frac{1}{4})}{\Gamma(\frac{1}{4})} \left( 1-\frac{1}{2^{\frac{1}{2}-s}}\right)\left(1 - \frac{1}{2^{1+2s}}\right)\zeta(1+2s) \\
&\times \left(\frac{X}{\pi} \right)^{s/2}  \left( gr\right)^{-s} \sum_n \Lambda(n) \Phi_s \left( \frac{n}{X}\right) \chi_{\gamma mr}(n) \, \frac{ds}{s}.
\end{split}
\end{equation}
We transform the sum on $n$ with partial summation to obtain
\begin{align}\label{eq: first moment partial summation before pnt}
\sum_n \Lambda(n) \Phi_s \left( \frac{n}{X}\right) \chi_{\gamma mr}(n) &= -\int_0^{\infty} \frac{1}{X}\Phi_s'\left(\frac{w}{X}\right) \Bigg(\sum_{n\leq w}\Lambda(n) \chi_{bmr}(n)\Bigg) \,dw.
\end{align}
By \cite[equation (8) of Chapter 20]{Dav}, we have
\begin{align}\label{eq: first moment davenport pnt}
\sum_{n\leq w}\Lambda(n) \chi_{\gamma mr}(n) = -\frac{w^{\beta_1}}{\beta_1} +O\left(w \exp(-c_1 \sqrt{\log w}) \right),
\end{align}
where $c_1>0$ is some absolute constant, and the term $-w^{\beta_1}/\beta_1$ only appears if $L(s,\chi_{\gamma mr})$ has a real zero $\beta_1$ which satisfies $\beta_1 > 1- \frac{c_2}{\log |\gamma mr|}$ for some sufficiently small constant $c_2>0$. All the constants in \eqref{eq: first moment davenport pnt}, implied or otherwise, are effective.

The contribution from the error term in \eqref{eq: first moment davenport pnt} is easy to control. Observe that
\begin{align}\label{eq: first moment bound on f s '}
\int_0^{\infty } \frac{1}{X}\left| \Phi_s'\left( \frac{w}{X}\right)\right|\,dw \ = \ \int_0^{\infty} |\Phi_s'(u)|\,du \ \ll \ |s|+1,
\end{align}
uniformly in $s$ with $\text{Re}(s)$ bounded. Taking \eqref{eq: first moment defn of E1},\eqref{eq: first moment partial summation before pnt} and \eqref{eq: first moment bound on f s '} together, we see the error term of \eqref{eq: first moment davenport pnt} contributes
\begin{align}\label{eq: first moment regime I ordinary zeros}
\ll X \exp(c_3(\varpi-c_1)\sqrt{\log X})
\end{align}
to $E_1$, where $c_3 > 0$ is some absolute constant. The bound \eqref{eq: first moment regime I ordinary zeros} is more than adequate for \eqref{eq: first moment bounds for E1 and E2} provided we choose $\varpi > 0$ sufficiently small in terms of $c_1$.

The conductor of the primitive character $\chi_{\gamma mr}$ is $\ll \exp(\varpi \sqrt{\log X}) \leq \exp(2\varpi \sqrt{\log X})$. We apply Page's theorem \cite[equation (9) of Chapter 14]{Dav}, which implies that, for some fixed absolute constant $c_4>0$, there is at most one real primitive character $\chi_{\gamma mr}$ with modulus $\leq \exp(2\varpi \sqrt{\log X})$ for which the $L$-function $L(s,\chi_{\gamma mr})$ has a real zero satisfying
\begin{align}\label{eq: first moment lower bound for exceptional beta 1}
\beta_1 > 1 - \frac{c_4}{2\varpi \sqrt{\log X}}.
\end{align}
To estimate the contribution of the possible term $ -\frac{w^{\beta_1}}{\beta_1}$, we evaluate the integral
\begin{align*}
\int_0^{\infty} \frac{w^{\beta_1}}{\beta_1}\frac{1}{X}\Phi_s'\left(\frac{w}{X}\right) \,dw
\end{align*}
arising from \eqref{eq: first moment partial summation before pnt} and \eqref{eq: first moment davenport pnt}. We make the change of variable $\frac{w}{X}\mapsto u$ and integrate by parts to see that this integral equals
\begin{align*}
X^{\beta_1}\int_0^{\infty} \frac{u^{\beta_1}}{\beta_1}\Phi_s'(u) \,du \ = \ -X^{\beta_1}\int_0^{\infty}\Phi_s(u)u^{\beta_1-1}\,du \ = \ -X^{\beta_1} \Phi^{\dagger}\left( \frac{s}{2}+\beta_1\right).
\end{align*}

We assume that a real zero satisfying \eqref{eq: first moment lower bound for exceptional beta 1} does exist, for otherwise we already have an acceptable bound for $E_1$. Let $q^*$ denote the conductor of the exceptional character $\chi_{\gamma mr}$ for which the real zero $\beta_1$ satisfying \eqref{eq: first moment lower bound for exceptional beta 1} exists. Then we have
\begin{equation}\label{eq: first moment exception E 1}
\begin{split}
E_1 = -\frac{1}{2\pi i} &\frac{\sqrt{\gamma^*}}{2\left(1 - \frac{1}{\sqrt{2}}\right)^2} \frac{X^{\beta_1}}{\sqrt{|q^*|}} \int_{\frac{1}{\log X} -i (\log X)^2}^{\frac{1}{\log X} +i (\log X)^2} \frac{\Gamma(\frac{s}{2}+\frac{1}{4})}{\Gamma(\frac{1}{4})} \left( 1-\frac{1}{2^{\frac{1}{2}-s}}\right)\left(1 - \frac{1}{2^{1+2s}}\right) \\
&\times\left( \frac{X}{\pi}\right)^{s/2}\zeta(1+2s)\Phi^{\dagger} \left( \frac{s}{2}+\beta_1\right) \mathop{\sum\sum}_{\substack{1 < mr \ll \exp(\varpi \sqrt{\log X}) \\ (mr,2)=1 \\ (m,r)=1 \\ \gamma mr = q^*}}\frac{\mu^2(r)}{r^s} \sum_{\substack{(g,2mr)=1}} \frac{\mu^2(g)b_{gm}}{g^{1+s}}\frac{ds}{s} \\
&+O\left(X\exp(-c_5\sqrt{\log X})\right),
\end{split}
\end{equation}
where $c_5>0$ is some constant, and $\gamma^*$ is some bounded power of two.

We next write $b_{gm} = \mu(gm) H(\frac{\log gm}{\log M})$ and apply Fourier inversion as in \eqref{eq: write H as Fourier integral},\eqref{eq: write h as Fourier transform} to obtain
\begin{equation}\label{eq: first moment exceptional mollifier}
\begin{split}
\sum_{\substack{(g,2mr)=1 }}\frac{b_{mg}}{g^{1+s}} \ &= \ \mu(m) \int_{-\infty}^{\infty} \frac{1}{m^{\frac{1+iz}{\log M}}} h(z) \prod_{p|2mr} \left(1-\frac{1}{p^{1+s+\frac{1+iz}{\log M}}} \right)^{-1}  \zeta^{-1}\left( 1+s+\frac{1+iz}{\log M}\right) \,dz.
\end{split}
\end{equation}
By \eqref{eq: h has rapid decay} we can truncate the integral in \eqref{eq: first moment exceptional mollifier} to $|z| \leq \sqrt{\log M}$ at the cost of an error of size $O_B(d_2(mr) (\log X)^{-B})$. This error contributes to \eqref{eq: first moment exception E 1}
\begin{align*}
\ll_B \frac{X}{(\log X)^{B+O(1)}},
\end{align*}
which is acceptable. We therefore have
\begin{equation}\label{eq: first moment exceptional E 1 opened mollifier}
\begin{split}
E_1 &= -\frac{X^{\beta_1}}{2\left( 1-\frac{1}{\sqrt{2}}\right)^2} \,  \frac{\sqrt{\gamma^*}}{\sqrt{|q^*|}}  \mathop{\sum\sum}_{\substack{m\leq M, r\leq X^{\frac{1}{2}+\varepsilon} \\ (mr,2)=1 \\ (m,r)=1\\ \gamma mr=q^* }} \mu(m)\mu^2(r) \ \frac{1}{2\pi i} \int_{\frac{1}{\log X} -i (\log X)^2}^{\frac{1}{\log X} +i (\log X)^2}\frac{\Gamma(\frac{s}{2}+\frac{1}{4})}{\Gamma(\frac{1}{4})} \left( 1-\frac{1}{2^{\frac{1}{2}-s}}\right) r^{-s}  \\
& \times   \left(1-\frac{1}{2^{1+2s}}\right)\zeta(1+2s) \left(\frac{X}{\pi}\right)^{s/2} \Phi^{\dagger} \left( \frac{s}{2}+\beta_1\right) \,\int_{|z| \leq \sqrt{\log M}} \frac{1}{m^{\frac{1+iz}{\log M}}} \ h(z) \\
&\times \prod_{p|2mr} \left(1-\frac{1}{p^{1+s+\frac{1+iz}{\log M}}} \right)^{-1}  \zeta^{-1}\left( 1+s+\frac{1+iz}{\log M}\right) \,dz \frac{ds}{s} + O\left(\frac{X}{\log X}\right).
\end{split}
\end{equation}
We handle the $s$-integral in \eqref{eq: first moment exceptional E 1 opened mollifier} by moving the line of integration to $\text{Re}(s) = -\frac{c_6}{\log \log X}$, where $c_6>0$ is small enough that $\zeta(1+s+\frac{1+iz}{\log M})$ has no zeros in the region $\text{Re}(s) \geq -\frac{c_6}{\log \log X}, \text{Im}(s) \leq (\log X)^2$. By moving the line of integration we pick up a contribution from the pole at $s=0$. We write this residue as an integral around a circle of small radius centered at the origin, and thereby deduce
\begin{equation}\label{eq: first moment exceptional E 1 moved contour}
\begin{split}
E_1 &= -\frac{X^{\beta_1}}{2\left( 1-\frac{1}{\sqrt{2}}\right)^2} \,  \frac{\sqrt{\gamma^*}}{\sqrt{|q^*|}}  \mathop{\sum\sum}_{\substack{m\leq M, r\leq X^{1/2+\varepsilon} \\ (mr,2)=1 \\ (m,r)=1\\ \gamma mr=q^* }} \mu(m)\mu^2(r) \ \frac{1}{2\pi i} \oint_{|s| = \frac{1}{\log X}}\frac{\Gamma(\frac{s}{2}+\frac{1}{4})}{\Gamma(\frac{1}{4})} \left( 1-\frac{1}{2^{\frac{1}{2}-s}}\right) r^{-s} \\
& \times   \left(1-\frac{1}{2^{1+2s}}\right)\zeta(1+2s) \left(\frac{X}{\pi}\right)^{s/2} \Phi^{\dagger} \left( \frac{s}{2}+\beta_1\right) \,\int_{|z| \leq \sqrt{\log M}} \ \frac{1}{m^{\frac{1+iz}{\log M}}} h(z) \\
&\times \prod_{p|2mr} \left(1-\frac{1}{p^{1+s+\frac{1+iz}{\log M}}} \right)^{-1}  \zeta^{-1}\left( 1+s+\frac{1+iz}{\log M}\right) \,dz \frac{ds}{s} + O\left(\frac{X}{\log X}\right). 
\end{split}
\end{equation}

We have the bound
\begin{align}\label{eq: first moment upper bound for beta 1}
\beta_1 < 1 - \frac{c_7}{\sqrt{|q^*|} (\log |q^*|)^2},
\end{align}
where $c_7 > 0$ is a fixed absolute constant (see \cite[equation (12) of Chapter 14]{Dav}). If $q^*$ satisfies $|q^*| \leq (\log X)^{2-\varepsilon}$ then by \eqref{eq: first moment upper bound for beta 1} we derive
\begin{align*}
X^{\beta_1} \ll X \exp(-c_7 (\log X)^{\varepsilon/3}).
\end{align*}
By estimating \eqref{eq: first moment exceptional E 1 moved contour} trivially we then obtain
\begin{align*}
E_1 &\ll X \exp(-c_7 (\log X)^{\varepsilon/4}),
\end{align*}
which is an acceptable bound. We may therefore assume that $q^*$ satisfies
\begin{align}\label{eq: first moment lower bound for q star}
|q^*| > (\log X)^{2-\varepsilon}.
\end{align}
For $|s| = \frac{1}{\log X}$ we have the bounds
\begin{align*}
\zeta(1+2s) \ll \log X, \ \ \ \ \ \ \zeta^{-1}\left( 1+s+\frac{1+iz}{\log M}\right) \ll \frac{1+|z|}{\log X}.
\end{align*}
Using these bounds and \eqref{eq: first moment lower bound for q star} we deduce by trivial estimation that
\begin{align*}
\eqref{eq: first moment exceptional E 1 moved contour} \ll \frac{X}{|q^*|^{1/2-o(1)}} \ll \frac{X}{(\log X)^{1-\varepsilon}}.
\end{align*}
This completes the proof of the bound for $E_1$ in \eqref{eq: first moment bounds for E1 and E2}.

\subsection{Regime II}

It remains to prove the bound for $E_2$ in \eqref{eq: first moment bounds for E1 and E2}. From \eqref{eq: first moment S1 neq all done massaging} and \eqref{eq: first moment decompose S1 neq into E1 and E2} we see that $E_2$ is the contribution from those $m$ and $r$ in Regimes II and III. The estimates in regimes II and III are less delicate than those in regime I, and consequently the arguments are easier.

In \eqref{eq: first moment S1 neq all done massaging} we write $q = \gamma m r$. After breaking $q$ into dyadic segments we find
\begin{align*}
E_2 &\ll (\log X)^{O(1)} \sum_{\substack{Q=2^j \\ Q \gg \exp(\varpi \sqrt{\log X}) \\ Q \ll M X^{1/2+\varepsilon}}} \mathcal{E}(Q),
\end{align*}
where
\begin{align*}
\mathcal{E}(Q) := Q^{-\frac{1}{2}+\varepsilon} \sum_{\chi \in S(Q)} \left|\sum_n \Lambda(n) \Phi_{s_0} \left( \frac{n}{X}\right) \chi(n) \right|.
\end{align*}
Here $s_0$ is some complex number with $\text{Re}(s_0) = \frac{1}{\log X}$ and $|\text{Im}(s_0)| \leq (\log X)^2$. In order to prove \eqref{eq: first moment bounds for E1 and E2} it therefore suffices to show that
\begin{align}\label{eq: first moment bound for cal E Q}
\mathcal{E}(Q) \ll X \exp(-c_8 \varpi \sqrt{\log X})
\end{align}
for each $Q$ satisfying $\exp(\varpi \sqrt{\log X}) \ll Q\ll M X^{\frac{1}{2}+\varepsilon}$. In this subsection we treat the $Q$ belonging to Regime II, that is, those $Q$ which satisfy $Q \ll X^{\frac{1}{10}}$. In the next subsection we treat the $Q$ in Regime III, which satisfy $Q \gg X^{\frac{1}{10}}$.

In Regime II we employ zero-density estimates. We begin by writing $\Phi_{s_0}$ as the integral of its Mellin transform, yielding
\begin{align*}
\sum_n \Lambda (n) \, \Phi_{s_0}\left(\frac{n}{X}\right) \chi(n) &= \frac{1}{2\pi i} \int_{(2)} X^w \Phi^{\dagger}\left(w+\frac{s_0}{2}\right) \left(-\frac{L'}{L}\left(w,\chi \right)\right) dw.
\end{align*}
Observe that from repeated integration by parts we have
\begin{align}\label{eq: f dagger rapid decay}
\left|\Phi^\dagger(\sigma + it + \frac{s_0}{2})\right| \ll_{\sigma,j} (\log X)^j \left(1+\left|t-\frac{\text{Im}(s_0)}{2}\right|\right)^{-j}
\end{align}
for every non-negative integer $j$.

We shift the line of integration to $\text{Re}(w) = -\frac{1}{2}$, picking up residues from all of the zeros in the critical strip. On the line $\text{Re}(w) = -\frac{1}{2}$ we have the bound
\begin{align*}
\left|\frac{L'}{L}(w,\chi) \right| \ll \log(q|w|),
\end{align*}
and this yields
\begin{align*}
\sum_n \Lambda (n) \, \Phi_{s_0}\left(\frac{n}{X}\right) \chi(n) &= \sum_{\substack{L(\rho,\chi) = 0 \\ 0 \leq \beta \leq 1}} X^\rho \Phi^\dagger\left(\rho + \frac{s_0}{2}\right) + O \left(\frac{(\log X)^{O(1)}}{X^{1/2}} \right).
\end{align*}
We have written here $\rho = \beta + i\gamma$. The error term is, of course, completely acceptable for \eqref{eq: first moment bound for cal E Q} when summed over $q \ll Q$.

By \eqref{eq: f dagger rapid decay}, the contribution to $\mathcal{E}(Q)$ from those $\rho$ with $|\gamma| > Q^{1/2}$ is
\begin{align*}
\ll X Q^{-100},
\end{align*}
say, and this gives an acceptable bound. We have therefore obtained
\begin{align}\label{eq: regime II, reduction to zero counting}
\mathcal{E}(Q) \ll X\exp(-\varpi \sqrt{\log X}) + Q^{-\frac{1}{2}+\varepsilon}\sum_{\substack{\chi \in S(Q)}} \ \sum_{\substack{L(\rho,\chi)=0 \\ 0 \leq \beta \leq 1 \\ |\gamma| \leq Q^{1/2}}} X^\beta.
\end{align}

In order to bound the right side of \eqref{eq: regime II, reduction to zero counting}, we first need to introduce some notation. For a primitive Dirichlet character $\chi$ modulo $q$, let $N(T,\chi)$ denote the number of zeros of $L(s,\chi)$ in the rectangle
\begin{align*}
0 \leq \beta \leq 1, \ \ \ \ \ \ \  |\gamma | \leq T.
\end{align*}
For $T \geq 2$, say, we have \cite[Chapter 16]{Dav}
\begin{align}\label{eq: bound on N T chi}
N(T,\chi) \ll T \log(qT).
\end{align}
For $\frac{1}{2}\leq \alpha \leq 1$, define $N(\alpha,T,\chi)$ to be the number of zeros $\rho = \beta +i\gamma$ of $L(s,\chi)$ in the rectangle
\begin{align*}
\alpha \leq \beta \leq 1, \ \ \ \ \ \ \ |\gamma| \leq T,
\end{align*}
and define
\begin{align*}
N(\alpha,Q,T) &= \sum_{q \leq Q} \ \sideset{}{^*}\sum_{\substack{\chi (\text{mod }q)}} N(\alpha,T,\chi).
\end{align*}
In $N(\alpha,Q,T)$ the summation on $\chi$ is over primitive characters. We employ Jutila's zero-density estimate \cite[(1.7)]{Jut77}
\begin{align}\label{eq: Jutila zero density}
N(\alpha,Q,T) \ll (QT)^{4(1-\alpha)+\varepsilon},
\end{align}
which holds for $\alpha \geq \frac{4}{5}$.

In \eqref{eq: regime II, reduction to zero counting}, we separate the zeros $\rho$ according to whether $\beta < \frac{4}{5}$ or $\beta \geq \frac{4}{5}$. Using \eqref{eq: bound on N T chi} we deduce
\begin{align}\label{eq: regime II bound for zeros less than 4 5}
Q^{-\frac{1}{2}+\varepsilon}\sum_{\substack{\chi \in S(Q)}} \ \sum_{\substack{L(\rho,\chi)=0 \\ 0 \leq \beta < 4/5 \\ |\gamma| \leq Q^{1/2}}} X^\beta \ll X^{\frac{4}{5}}Q^{1+\varepsilon}.
\end{align}

For those zeros with $\beta \geq \frac{4}{5}$ we write
\begin{align*}
X^\beta = X^{4/5} + (\log X)\int_{4/5}^\beta X^\alpha d\alpha.
\end{align*}
We then embed $S(Q)$ into the set of all primitive characters with conductors $\leq Q$. Applying \eqref{eq: regime II bound for zeros less than 4 5} and \eqref{eq: Jutila zero density}, we obtain
\begin{align*}
\sum_{\substack{\chi \in S(Q)}} \ \sum_{\substack{L(\rho,\chi_q)=0 \\ 4/5 \leq \beta \leq 1 \\ |\gamma| \leq Q^{\frac{1}{2}}}} X^\beta &\ll X^{\frac{4}{5}} Q^{1+\varepsilon} + (\log X)\int_{4/5}^1 X^\alpha N(\alpha,Q,Q^{\frac{1}{2}}) d\alpha \\
&\ll X^{\frac{4}{5}} Q^{1+\varepsilon} + Q^\varepsilon \int_{4/5}^1 X^\alpha Q^{6(1-\alpha)}d\alpha.
\end{align*}
Since $Q \ll X^{\frac{1}{10}}$ the integrand of this latter integral is maximized when $\alpha = 1$. It follows that
\begin{align}\label{eq: regime II bound for zeros close to one}
Q^{-\frac{1}{2}+\varepsilon}\sum_{\substack{\chi \in S(Q)}} \ \sum_{\substack{L(\rho,\chi_q)=0 \\ 4/5 \leq \beta \leq 1 \\ |\gamma| \leq Q^{\frac{1}{2}}}} X^\beta \ll X^{\frac{4}{5}}Q^{1+\varepsilon} + X Q^{-\frac{1}{2}+\varepsilon} \ll X Q^{-\frac{1}{2}+\varepsilon}.
\end{align}
Combining \eqref{eq: regime II bound for zeros close to one} and \eqref{eq: regime II bound for zeros less than 4 5} yields
\begin{align*}
\mathcal{E}(Q) \ll XQ^{-\frac{1}{2}+\varepsilon},
\end{align*}
and this suffices for \eqref{eq: first moment bound for cal E Q}.

\subsection{Regime III}

In Regime III we have $X^{\frac{1}{10}} \ll Q \ll MX^{\frac{1}{2}+\varepsilon} = X^{\frac{1}{2}+\theta + \varepsilon}$ (recall \eqref{eq:outline section, defn of M, length of mollifier}). Here we depart from the philosophy of the previous two regimes, in that we do not bound $\mathcal{E}(Q)$ by considerations of zeros of $L$-functions. Rather, we exploit the combinatorial structure of the von Mangoldt function and Lemma \ref{lem: estimates for character sums}.

We observe that in Regime III one may still proceed with zero-density estimates by appealing to Heath-Brown's zero-density estimate for $L$-functions of quadratic characters \cite[Theorem 3]{HeaB95}. We present our method for the sake of variety, and because it might prove useful in other contexts.

Let us move to our treatment of $\mathcal{E}(Q)$ for these large $Q$. Given an arithmetic function $f:\mathbb{N} \rightarrow \mathbb{C}$ and a real number $W > 1$, let $f_{\leq W}(n)$ denote the arithmetic function 
\begin{align*}
f_{\leq W}(n) = 
\begin{cases}
f(n), \ &n \leq W, \\
0, &n > W.
\end{cases}
\end{align*}
We write $f_{>W}(n) = f(n) - f_{\leq W}(n)$. 

We write $\star$ for Dirichlet convolution. Our starting place is Vaughan's identity \cite[Proposition 13.4]{IK}. Given a parameter $V > 1$, we have
\begin{align}\label{eq: Vaughans identity}
\Lambda(n) &= \Lambda_{\leq V}(n) + (\mu_{\leq V} \star \log )(n) - (\mu_{\leq V} \star \Lambda_{\leq V} \star 1)(n) + (\mu_{> V} \star \Lambda_{> V} \star 1) (n).
\end{align}
We apply \eqref{eq: Vaughans identity} for $n \asymp X$, and we set $V := X^{\frac{1}{3}(\frac{1}{2} - \theta)}$. This reduces the estimation of $\mathcal{E}(Q)$ to the estimation of three different sums, say $\mathcal{E}_i(Q)$, for $i \in \{1,2,3\}$. Observe that there are four terms on the right side of \eqref{eq: Vaughans identity}, but $\Lambda_{\leq V}(n)$ is identically zero for $n \asymp X$.

We have
\begin{align*}
\mathcal{E}_1(Q) &:= Q^{-\frac{1}{2}+\varepsilon}\sum_{\substack{\chi \in S(Q)}} \left|\sum_n (\mu_{\leq V} \star \log)(n) \Phi_{s_0} \left(\frac{n}{X} \right)\chi(n) \right| \\
&\ll Q^{-\frac{1}{2}+\varepsilon}\sum_{\substack{\chi \in S(Q)}} \sum_{v \leq V} \mu^2(v) \left|\sum_m (\log m)\Phi_{s_0} \left(\frac{mv}{X} \right)\chi(m) \right|.
\end{align*}
By partial summation and the P\'olya-Vinogradov inequality, we find that
\begin{align}\label{eq: regime III bound for cal E 1}
\mathcal{E}_1(Q) &\ll Q^{1+\varepsilon}V \ll X^{\frac{1}{2} + \theta + \frac{1}{3}(\frac{1}{2}-\theta)+\varepsilon} \ll X^{1-\varepsilon},
\end{align}
the last inequality holding for $\varepsilon =\varepsilon(\theta) > 0$ sufficiently small.

The estimation of $\mathcal{E}_2(Q)$ is entirely similar, and we obtain
\begin{equation}\label{eq: regime III bound for cal E 2}
\begin{split}
\mathcal{E}_2(Q) &:= Q^{-\frac{1}{2}+\varepsilon}\sum_{\substack{\chi \in S(Q)}} \Bigg|\sum_n (\mu_{\leq V} \star \Lambda_{\leq V} \star 1)(n) \Phi_{s_0} \left(\frac{n}{X} \right)\chi(n) \Bigg| \\
&\ll Q^{1+\varepsilon}V^2 \ll X^{\frac{1}{2} + \theta + \frac{2}{3}(\frac{1}{2}-\theta)+\varepsilon} \ll X^{1-\varepsilon}.
\end{split}
\end{equation}

The last sum to estimate is $\mathcal{E}_3(Q)$:
\begin{align*}
\mathcal{E}_3(Q) &:= Q^{-\frac{1}{2}+\varepsilon}\sum_{\substack{\chi \in S(Q)}} \left|\sum_n (\mu_{> V} \star \Lambda_{> V} \star 1)(n) \Phi_{s_0} \left(\frac{n}{X} \right)\chi(n) \right| \\
&= Q^{-\frac{1}{2}+\varepsilon}\sum_{\substack{\chi \in S(Q)}} \Bigg|\mathop{\sum \sum}_{k,\ell} \alpha(k) \beta(\ell)\Phi_{s_0} \left( \frac{k\ell}{X}\right) \chi(k\ell) \Bigg|,
\end{align*}
where $\alpha(k) = \mu_{>V}(k)$ and $\beta(\ell) = (\Lambda_{>V} \star 1)(\ell)$. Observe that both $\alpha(\cdot)$ and $\beta(\cdot)$ are supported on integers  $m$ satisfying
\begin{align*}
V \ll m \ll XV^{-1}.
\end{align*}
We further observe that $|\alpha(k)| \leq 1, |\beta(\ell)| \leq \log(\ell)$. We perform dyadic decompositions on the ranges of $k$ and $\ell$, so that $k \asymp K, \ell \asymp L$, with
\begin{align*}
V \ll K \ll X V^{-1}, \ \ \ \ \ V \ll L \ll X V^{-1},
\end{align*}
and $KL \asymp X$. 

We next separate the variables by Mellin inversion on $\Phi_{s_0}$:
\begin{align*}
\mathcal{E}_1(Q) &\ll (\log X)^{O(1)} \sup_{K,L} \int_{(0)} \left|\Phi_{s_0}^{\dagger}\left( w + \frac{s_0}{2}\right) \right| Q^{-\frac{1}{2}+\varepsilon}\sum_{\substack{\chi \in S(Q)}}\Bigg|\mathop{\sum \sum}_{\substack{k \asymp K \\ \ell \asymp L}} \alpha(k) \beta(\ell)(k\ell)^{-w} \chi(k\ell) \Bigg| \ dw.
\end{align*}
The integral of $|\Phi_{s_0}^\dagger|$ has size $\ll (\log X)^{O(1)}$, so we obtain
\begin{align*}
\mathcal{E}_3(Q) \ll \sup_{K,L} \ Q^{-\frac{1}{2}+\varepsilon}\sum_{\substack{\chi \in S(Q)}}\Bigg|\mathop{\sum \sum}_{\substack{k \asymp K \\ \ell \asymp L}} \tilde{\alpha}(k) \tilde{\beta}(\ell) \chi(k\ell) \Bigg| ,
\end{align*}
where $\tilde{\alpha},\tilde{\beta}$ are complex sequences with $|\tilde{\alpha}(k)| = |\alpha(k)|, |\tilde{\beta}(\ell)| = |\beta(\ell)|$ for all $k,\ell$.

By multiplicativity and Cauchy-Schwarz we obtain
\begin{align*}
\mathcal{E}_3(Q) &\ll \sup_{K,L} \ Q^{-\frac{1}{2}+\varepsilon} \Bigg(\sum_{\substack{\chi \in S(Q)}} \Bigg|\sum_{k \asymp K} \tilde{\alpha}(k) \chi(k) \Bigg|^2 \Bigg)^{\frac{1}{2}}\Bigg(\sum_{\substack{\chi \in S(Q)}} \Bigg|\sum_{\ell \asymp L} \tilde{\beta}(\ell) \chi(\ell) \Bigg|^2 \Bigg)^{\frac{1}{2}}.
\end{align*}
Applying Lemma \ref{lem: estimates for character sums} yields
\begin{equation}\label{eq: regime III bound for cal E 3}
\begin{split}
\mathcal{E}_3(Q) &\ll \sup_{K,L} \ \frac{X^\varepsilon}{Q^{\frac{1}{2}}} \left((Q+K)K \right)^{\frac{1}{2}} \left((Q+L)L \right)^{\frac{1}{2}} \\
&\ll \sup_{K,L} \ X^\varepsilon \left(X^{\frac{1}{2}}Q^{\frac{1}{2}} + \frac{KL}{K^{\frac{1}{2}}} + \frac{KL}{L^{\frac{1}{2}}} + \frac{KL}{Q^{\frac{1}{2}}} \right) \\
&\ll X^\varepsilon \left(X^{\frac{3}{4} + \frac{\theta}{2}} + \frac{X}{V^{\frac{1}{2}}} + \frac{X}{Q^{\frac{1}{2}}} \right) \ll X^{1-\varepsilon}.
\end{split}
\end{equation}
The last inequality follows since $V = X^{\frac{1}{3}(\frac{1}{2}-\theta)}$ and $Q \gg X^{\frac{1}{10}}$. Then \eqref{eq: regime III bound for cal E 1}, \eqref{eq: regime III bound for cal E 2}, and \eqref{eq: regime III bound for cal E 3} imply
\begin{align*}
\mathcal{E}(Q) \ll X^{1-\varepsilon},
\end{align*}
and this suffices for \eqref{eq: first moment bound for cal E Q}.

\subsection{D\'enouement}

We can extract from our proof of Proposition \ref{prop: asymptotic for S1} the following result on character sums over primes, which we shall have occasion to use later.
\begin{lem}\label{lem:gen char sum over primes}
Let $X$ be a large real number, and let $\delta > 0$ be small and fixed. Let $s_0$ be a complex number with $|\textup{Re}(s_0)| \leq \frac{A_1}{\log X}$ and $|\textup{Im}(s_0)| \leq (\log X)^{A_2}$, for some positive real numbers $A_1$ and $A_2$. Given any positive real numbers $A_3,A_4$, and $B$, we have
\begin{align*}
\sum_{\substack{q \leq X^{1-\delta} \\ q \textup{ odd} \\ q \neq \square}} \frac{\tau(q)^{A_1} (\log q)^{A_2}}{\sqrt{q}} \Bigg|\sum_{p \equiv 1 \, (\textup{mod }8)} (\log p)\Phi_{s_0} \left( \frac{p}{X}\right) \left(\frac{q}{p} \right) \Bigg| \ll_{A_1,A_2,A_3,A_4,B,\delta} \ \frac{X}{(\log X)^B}.
\end{align*}
The implied constant is ineffective.
\end{lem}
\begin{proof}
Follow the proof of \eqref{eq: first moment bounds for E1 and E2}, but instead use the lower bound $q^* > c(D)(\log X)^D$, which holds for arbitrary $D>0$. The constant $c(D)$ is ineffective if $D \geq 2$.
\end{proof}

Lemma \ref{lem:gen char sum over primes} is quite strong since it corresponds, roughly, to square root cancellation on average in the sums over $p$. Thus, one would not expect to be able to prove an analogue of Lemma \ref{lem:gen char sum over primes} with the upper bound for $q$ replaced by $X^{1+\varepsilon}$ for any $\varepsilon > 0$.

\section{The mollified second moment}\label{sec:mollified second moment}

In this section we derive an upper bound of the correct order of magnitude for the sum $S_2$ defined in \eqref{eq: defn moment sums S1 and S2}. Our main result for this section is the following (recall \eqref{eq:outline section, defn of M, length of mollifier} and \eqref{eq:sieve section, defn of R}).
\begin{prop}\label{prop:upper bound for S2}
Let $\delta > 0$ be small and fixed, and let $\theta,\vartheta$ satisfy $\theta + 2\vartheta < \frac{1}{2}$. If $X \geq X_0(\delta,\theta,\vartheta)$, then
\begin{align*}
S_2 \leq \frac{1+\delta}{2(1- \frac{1}{\sqrt{2}})^2} \frac{\mathfrak{I}}{\vartheta} \frac{X}{4},
\end{align*}
where
\begin{align*}
\mathfrak{I} = -2\int_0^1 H(x)H'(x) dx &+ \frac{1}{\theta}\int_0^1 H(x) H''(x) dx + \frac{1}{\theta}\int_0^1 H'(x)^2 dx \\ 
&- \frac{1}{2\theta^2}\int_0^1 H'(x)H''(x)dx+\frac{1}{24\theta^3}\int_0^1 H''(x)^2 dx.
\end{align*}
\end{prop}

The proof of Proposition \ref{prop:upper bound for S2} follows the ideas outlined in Section~\ref{sec: outline of pos prop thm}. First, we note that $\log p\leq \log X$ in \eqref{eq: defn moment sums S1 and S2} because $\Phi$ is supported on $[\frac{1}{2},1]$. By positivity we may apply the upper bound sieve condition \eqref{sieveinequality} to write
\begin{equation*}
S_2\leq (\log X) S^+,
\end{equation*}
where $S^+$ is defined by
\begin{equation}\label{S+def}
S^+ = \sum_{\substack{ n\equiv 1 \, (\text{mod } 8)}} \mu^2(n) \Bigg(\sum_{\substack{d \mid n \\ d \leq D}} \lambda_d \Bigg) \Phi\left( \frac{n}{X}\right) L(\tfrac{1}{2},\chi_n)^2 M(n)^2.
\end{equation}
Note that $d$ is odd since $d \mid n$ and $n\equiv 1 \pmod 8$. Also, $\lambda_d\neq 0$ only for square-free $d$ by the definition \eqref{lambda}, and so $\lambda_d=\mu^2(d)\lambda_d$. We use Lemma \ref{lem: approx func eq} to write $L(\tfrac{1}{2},\chi_n)^2 = \mathcal{D}_2(n)$, then insert \eqref{mu2approx} into \eqref{S+def} to write
\begin{equation}\label{S+SMSR}
S^+ = S^+_N + S^+_R,
\end{equation}
where
\begin{equation}\label{SMdef}
S^+_N =  \sum_{\substack{ n\equiv 1 \, (\text{mod } 8) }} N_Y(n) \Bigg(\sum_{\substack{d \mid n \\ d \leq D}}\mu^2(d) \lambda_d \Bigg) \Phi\left( \frac{n}{X}\right) \mathcal{D}_2(n) M(n)^2
\end{equation}
and
\begin{equation*}
S^+_R = \sum_{\substack{ n\equiv 1 \,  (\text{mod } 8) }} R_Y(n) \Bigg(\sum_{\substack{d \mid n \\ d \leq D}}\mu^2(d) \lambda_d \Bigg) \Phi\left( \frac{n}{X}\right) \mathcal{D}_2(n) M(n)^2
\end{equation*}
We first obtain a bound on $S_R^+$. The remainder of this section will then be devoted to an analysis of $S_N^+$.

\subsection{The contribution of $S_R^+$}\label{subsec:contrib of R_Y}

In this subsection we show
\begin{align}\label{eq:bound on S_R^+}
S_R^+ \ll X^\varepsilon \left(\frac{X}{Y} + X^{1/2} M \right).
\end{align}
The arguments here are almost identical to those in \cite[Section 3]{Sou00}. Observe that $R_Y(n)=0$ unless $n = \ell^2 h$ with $\ell > Y$ and $h$ square-free. If $n\equiv 1 \pmod{8}$ then $\ell$ and $h$ are odd and $h \equiv 1 \pmod{8}$. By the divisor bound we have
\begin{align*}
|R_Y(n)| \ll n^\varepsilon, \ \ \ \ \ \ \  \ \Bigg| \sum_{\substack{d \mid n \\ d \leq D}}\mu^2(d) \lambda_d \Bigg| \ll n^\varepsilon,
\end{align*}
and therefore
\begin{align*}
S_R^+ \ll X^\varepsilon \sum_{\substack{Y < \ell \leq \sqrt{X} \\ 2 \nmid \ell}} \ \sum_{\substack{X/2\ell^2 < h \leq X/\ell^2 \\ h \equiv 1 \, (\text{mod }8)}} \mu^2(h) |M(\ell^2h)^2\mathcal{D}_2(\ell^2h)|.
\end{align*}
There is a mild complication compared to \cite{Sou00} in that it is possible to have $h=1$, in which case the character $\chi_h$ is principal.

We apply Cauchy-Schwarz and obtain
\begin{align}\label{eq:up bound for S R plus after cauchy}
S_R^+ \ll X^\varepsilon \sum_{\substack{Y < \ell \leq \sqrt{X} \\ 2 \nmid \ell}} \Bigg(\sum_{\substack{X/2\ell^2 < h \leq X/\ell^2 \\ h \equiv 1 \, (\text{mod }8)}} \mu^2(h) |M(\ell^2h)^2|^2 \Bigg)^{1/2}\Bigg(\sum_{\substack{X/2\ell^2 < h \leq X/\ell^2 \\ h \equiv 1 \, (\text{mod }8)}} \mu^2(h) |\mathcal{D}_2(\ell^2h)|^2 \Bigg)^{1/2}.
\end{align}
We have
\begin{align*}
M(\ell^2h)^2 = \sum_{\substack{m \leq M^2 \\ (m,2\ell)=1}} \frac{\alpha(m)}{\sqrt{m}}\left(\frac{h}{m} \right)
\end{align*}
for some coefficients $\alpha(m)$ satisfying $|\alpha(m)| \ll m^\varepsilon$. For $h=1$ we use the trivial bound $M(\ell^2)^4 \ll M^2 X^\varepsilon$. For $h > 1$ we use Lemma \ref{lem: estimates for character sums}. We therefore have
\begin{align}\label{eq:up bound S R + mollifier squared}
\sum_{\substack{X/2\ell^2 < h \leq X/\ell^2 \\ h \equiv 1 \, (\text{mod }8)}} \mu^2(h) |M(\ell^2h)^2|^2 \ll X^\varepsilon \left(\frac{X}{\ell^2} + M^2 \right).
\end{align}

Now observe that, for any $c > \frac{1}{2}$,
\begin{align*}
\mathcal{D}_2(\ell^2h) = \frac{2}{(1 - \frac{1}{\sqrt{2}})^4} \frac{1}{2\pi i}\int_{(c)}\frac{\Gamma^2 \left( \frac{s}{2} + \frac{1}{4}\right)}{\Gamma^2 \left( \frac{1}{4}\right)} \left(1 - \frac{1}{2^{1/2-s}}\right)^2 \left(\frac{\ell^2 h}{\pi} \right)^s L^2 \left( \frac{1}{2}+s,\chi_h\right) \mathcal{E}(s,2\ell) \frac{ds}{s},
\end{align*}
where
\begin{align*}
\mathcal{E}(s,k) = \prod_{p \mid k} \left(1 - \frac{\chi_h(p)}{p^{1/2+s}}\right)^s.
\end{align*}
If $h = 1$ then $L^2 \left( \frac{1}{2}+s,\chi_h\right)=\zeta^2(\frac{1}{2}+s)$. In any case, we move the line of integration to $c = \frac{1}{\log X}$, and we do not pick up contributions from any poles. When $h > 1$ this is obvious, and when $h = 1$ the double pole of $\zeta^2(\frac{1}{2}+s)$ is canceled out by the double zero of $(1 - 2^{-(1/2-s)})^2$. By trivial estimation we have then $|\mathcal{D}_2(\ell^2)| \ll X^\varepsilon$. For $h >1$ we apply Cauchy-Schwarz to obtain
\begin{align*}
|\mathcal{D}_2(\ell^2h)|^2 \ll X^\varepsilon\int_{(\frac{1}{\log X})} \left|\Gamma \left( \frac{s}{2} + \frac{1}{4}\right)\right|^2 \left|L \left( \frac{1}{2}+s,\chi_h\right) \right|^4 |ds|.
\end{align*}
Summing over $h$ and using Lemma \ref{lem: moment estimates}, we obtain
\begin{align}\label{eq:S_R contrib from fourth moment}
\sum_{\substack{X/2\ell^2 < h \leq X/\ell^2 \\ h \equiv 1 \, (\text{mod }8)}} \mu^2(h) |\mathcal{D}_2(\ell^2h)|^2 \ll \frac{X^{1+\varepsilon}}{\ell^2}.
\end{align}
Combining \eqref{eq:up bound for S R plus after cauchy}, \eqref{eq:up bound S R + mollifier squared}, and \eqref{eq:S_R contrib from fourth moment} yields \eqref{eq:bound on S_R^+}.

\subsection{Poisson summation}

We begin our evaluation of $S_N^+$ by inserting into \eqref{SMdef} the definition \eqref{eq: defn of mollifier} of the mollifier $M(n)$. We then use the definition of $\mathcal{D}_2$ (see Lemma \ref{lem: approx func eq}) to write
\begin{equation}\label{SM1}
\begin{split}
S^+_N = \frac{8}{(\sqrt{2}-1)^4}\sum_{\substack{ d\leq D \\ d \text{ odd} }} \mu^2(d)\lambda_d
& \mathop{\sum\sum}_{\substack{m_1,m_2\leq M \\ m_1,m_2 \text{ odd}}} \frac{b_{m_1}b_{m_2}}{\sqrt{m_1m_2}} \sum_{\substack{ n\equiv 1 \, (\text{mod } 8) \\ d|n }} N_Y(n) \Phi\left( \frac{n}{X}\right) \\
& \times \sum_{\substack{ \nu=1 \\ \nu \text{  odd}}}^{\infty} \frac{d_2(\nu) }{\sqrt{\nu}} \omega_2\left( \frac{\nu \pi}{n}\right) \left( \frac{n}{m_1m_2 \nu}\right).
\end{split}
\end{equation}

We next apply Poisson summation to evaluate the $n$-sum. Denote the $n$-sum in \eqref{SM1} by $Z$, i.e. define $Z$ by
\begin{equation}\label{Zdef}
Z = Z(d,\nu,m_1m_2;X,Y) = \sum_{\substack{ n\equiv 1 \, (\text{mod } 8) \\ d|n }} N_Y(n) \Phi\left( \frac{n}{X}\right) \omega_2\left( \frac{\nu \pi}{n}\right) \left( \frac{n}{m_1m_2 \nu}\right) .
\end{equation}
We insert the definition \eqref{MYRY} of $N_Y(n)$ and interchange the order of summation to write $Z$ as
\begin{equation}\label{Z1}
Z = \sum_{\substack{ \alpha\leq Y \\ \alpha \text{ odd}}} \mu(\alpha) \sum_{\substack{ n\equiv1 \, (\text{mod } 8) \\ [\alpha^2,d]|n }} F_{\nu}\left( \frac{n}{X}\right)  \left( \frac{n}{m_1m_2 \nu}\right),
\end{equation}
where $F_{\nu}(t)$ is defined by
\begin{equation}\label{Fdef}
F_{\nu}(t) = \Phi(t) \omega_2\left( \frac{\nu \pi}{tX}\right).
\end{equation}
If $\alpha$ and $d$ are square-free, then $[\alpha^2,d]=\alpha^2 d_1$, where
\begin{equation}\label{eq: defn of d 1}
d_1 = \frac{d}{(d,\alpha)}.
\end{equation}
We may thus relabel $n$ as $\alpha^2d_1m$ in \eqref{Z1}, and then split the resulting sum on $m$ according to the congruence class of $m \pmod {m_1m_2\nu}$. We deduce from \eqref{Z1} that
\begin{equation*}
Z = \sum_{\substack{ \alpha\leq Y \\ (\alpha,2m_1m_2\nu)=1}} \mu(\alpha) \left( \frac{d_1}{m_1m_2\nu}\right) \sum_{b \ (\text{mod } m_1m_2\nu)} \left( \frac{b}{m_1m_2\nu}\right)  \sum_{\substack{ m\equiv \overline{\alpha^2 d_1} \ (\text{mod } 8) \\ m\equiv b \ (\text{mod } m_1m_2\nu) }}F_{\nu}\left( \frac{\alpha^2 d_1 m}{X}\right).
\end{equation*}
By the Chinese Remainder Theorem, we may write the congruence conditions on $m$ as a single condition $m\equiv \gamma \pmod {8m_1m_2\nu}$ for some integer $\gamma$ depending on $\alpha,d,b$. Thus, we may relabel $m$ as $8jm_1m_2\nu+\gamma$, where $j$ ranges over all integers, and arrive at
\begin{equation}\label{Z2}
Z = \sum_{\substack{ \alpha\leq Y \\ (\alpha,2m_1m_2\nu)=1}} \mu(\alpha) \left( \frac{d_1}{m_1m_2\nu}\right) \sum_{b \ (\text{mod } m_1m_2\nu)} \left( \frac{b}{m_1m_2\nu}\right)     \sum_{j\in\mathbb{Z}}F_{\nu}\left( \frac{\alpha^2d_1(8jm_1m_2\nu+\gamma)}{X}\right).
\end{equation}
We apply Poisson summation to the $j$-sum to write
\begin{equation*}
\sum_{j\in\mathbb{Z}}F_{\nu}\left( \frac{\alpha^2d_1(8jm_1m_2\nu+\gamma)}{X}\right) = \frac{X}{8\alpha^2d_1m_1m_2\nu} \sum_{k \in\mathbb{Z}}e\left(\frac{k\gamma}{8m_1m_2\nu}\right)\hat{F}_{\nu}\left( \frac{kX}{8\alpha^2d_1m_1m_2\nu}\right).
\end{equation*}
We insert this into \eqref{Z2}, apply the reciprocity relation
$$
e\left(\frac{k\gamma}{8m_1m_2\nu}\right) \ = \ e\left(\frac{k\overline{8}b}{m_1m_2\nu}\right)e\left(\frac{k\overline{\alpha^2d_1m_1m_2\nu}}{8}\right),
$$
and then evaluate the $b$-sum using the definition \eqref{Gaussdef} of the Gauss sum. Therefore
\begin{equation*}
\begin{split}
Z = & \frac{X}{8m_1m_2\nu} \sum_{\substack{ \alpha\leq Y \\ (\alpha,2m_1m_2\nu)=1}} \frac{\mu(\alpha)}{\alpha^2d_1} \left( \frac{2d_1}{m_1m_2\nu}\right) \\
& \ \ \times \sum_{k \in\mathbb{Z}}e\left(\frac{k\overline{\alpha^2d_1m_1m_2\nu}}{8}\right)\hat{F}_{\nu}\left( \frac{kX}{8\alpha^2d_1m_1m_2\nu}\right) \tau_k(m_1m_2\nu).
\end{split}
\end{equation*}
Recalling \eqref{SM1} and \eqref{Zdef}, we arrive at
\begin{equation}\label{SM2}
\begin{split}
S_N^+ = & \frac{X}{(\sqrt{2}-1)^4}\sum_{\substack{ d\leq D \\ d \text{ odd} }} \mu^2(d)\lambda_d \mathop{\sum\sum}_{\substack{m_1,m_2\leq M \\ (m_1m_2,2d)=1}} \frac{b_{m_1}b_{m_2}}{(m_1m_2)^{3/2}}   \sum_{\substack{ \nu=1 \\ (\nu,2d)=1}}^{\infty} \frac{d_2(\nu) }{\nu^{3/2}} \sum_{\substack{ \alpha\leq Y \\ (\alpha,2m_1m_2\nu)=1}}  \frac{\mu(\alpha)}{\alpha^2d_1} \\
& \ \ \times  \left( \frac{2d_1}{m_1m_2\nu}\right)    \sum_{k \in\mathbb{Z} }e\left(\frac{k\overline{\alpha^2d_1m_1m_2\nu}}{8}\right)\hat{F}_{\nu}\left( \frac{kX}{8\alpha^2d_1m_1m_2\nu}\right) {\tau}_k(m_1m_2\nu).
\end{split}
\end{equation}
Note that we may impose the condition $(m_1m_2\nu,d)=1$ because otherwise $( \frac{2d_1}{m_1m_2\nu})=0$. We write \eqref{SM2} as
\begin{equation}\label{SM3}
S_N^+ = \mathcal{T}_0 + \mathcal{B},
\end{equation}
where $\mathcal{T}_0$ is the contribution from $k=0$ in \eqref{SM2}, while $\mathcal{B}$ is the contribution from $k\neq 0$ in \eqref{SM2}. We evaluate $\mathcal{T}_0$ in the next subsection, and $\mathcal{B}$ in later subsections.

\subsection{The contribution from $k=0$}\label{subsec:k = 0}

By \eqref{Gaussdef}, $\tau_0(n)=\varphi(n)$ if $n$ is a perfect square, and $\tau_0(n)=0$ otherwise. Hence the term $\mathcal{T}_0$ in \eqref{SM2} is
\begin{equation}\label{eq: T0 start}
\begin{split}
\mathcal{T}_0 =  \frac{X}{(\sqrt{2}-1)^4}\sum_{\substack{ d\leq D \\ d \text{ odd} }} \mu^2(d)\lambda_d \mathop{\sum\sum}_{\substack{m_1,m_2\leq M \\ (m_1m_2,2d)=1}} \frac{b_{m_1}b_{m_2}}{(m_1m_2)^{3/2}}   &\sum_{\substack{ \nu=1 \\ (\nu,2d)=1 \\ m_1m_2\nu=\square }}^{\infty} \frac{d_2(\nu) }{\nu^{3/2}} \sum_{\substack{ \alpha\leq Y \\ (\alpha,2m_1m_2\nu)=1}}  \frac{\mu(\alpha)}{\alpha^2d_1} \\
&\times    \hat{F}_{\nu}(0) \varphi(m_1m_2\nu).
\end{split}
\end{equation}

We first extend the sum over $\alpha$ to infinity. Since $\varphi(n)\leq n$, the error introduced in doing so is
\begin{equation}\label{eq: error in extending T0 alpha sum}
\ll X\sum_{d\leq D} |\lambda_d| \mathop{\sum\sum}_{m_1,m_2\leq M} \frac{|b_{m_1}b_{m_2}|}{\sqrt{m_1m_2}} \sum_{\substack{\nu=1\\ m_1m_2\nu=\square}}^{\infty}\frac{d_2(\nu)}{\sqrt{\nu}} \sum_{\alpha>Y} \frac{1}{\alpha^2d_1} |\hat{F}_{\nu}(0)|.
\end{equation}
By Lemma~\ref{lem: properties of omega j}, $\hat{F}_{\nu}(0)\ll 1$ uniformly for all $\nu>0$, and $\hat{F}_{\nu}(0)\ll \exp(-\frac{\pi\nu}{8X})$ for $\nu>X^{1+\varepsilon}$. Moreover, \eqref{lambda} implies that $|\lambda_d|\ll d^{\varepsilon}$, while $|b_m|\ll 1$ by \eqref{eq: defn of mollifier coeffs bm}. It follows from these bounds that \eqref{eq: error in extending T0 alpha sum} is
\begin{equation}\label{eq: error in extending T0 alpha sum 2}
\ll X^{1+\varepsilon}\sum_{d\leq D}  \mathop{\sum\sum}_{m_1,m_2\leq M} \frac{1}{\sqrt{m_1m_2}} \sum_{\substack{\nu \leq X^{1+\varepsilon}\\ m_1m_2\nu=\square}}\frac{1}{\sqrt{\nu}} \sum_{\alpha>Y} \frac{1}{\alpha^2d_1} + \exp\left(-X^{\varepsilon}\right).
\end{equation}
Since $m_1m_2\nu$ is a perfect square, the sum over $m_1,m_2,\nu$ in \eqref{eq: error in extending T0 alpha sum 2} is $\ll X^{\varepsilon}$. Also, the definition \eqref{eq: defn of d 1} of $d_1$ implies that
$$
\sum_{\alpha>Y} \frac{1}{\alpha^{2} d_1} = \frac{1}{d} \sum_{j|d} \varphi(j) \sum_{\substack{\alpha>Y \\ j|\alpha}} \frac{1}{\alpha^{2}} \ll \frac{1}{d^{1-\varepsilon}Y}.
$$
Therefore \eqref{eq: error in extending T0 alpha sum 2} is $O(X^{1+\varepsilon}/Y)$. This bounds the error in extending the sum over $\alpha$ in \eqref{eq: T0 start} to infinity, and we arrive at
\begin{equation*}
\begin{split}
\mathcal{T}_0 =  \frac{X}{(\sqrt{2}-1)^4}\sum_{\substack{ d\leq D \\ d \text{ odd} }} \mu^2(d)\lambda_d \mathop{\sum\sum}_{\substack{m_1,m_2\leq M \\ (m_1m_2,2d)=1}} \frac{b_{m_1}b_{m_2}}{(m_1m_2)^{3/2}}   \sum_{\substack{ \nu=1 \\ (\nu,2d)=1 \\ m_1m_2\nu=\square }}^{\infty} \frac{d_2(\nu) }{\nu^{3/2}} \sum_{\substack{ \alpha=1 \\ (\alpha,2m_1m_2\nu)=1}}^{\infty}  \frac{\mu(\alpha)}{\alpha^2d_1} \\
\times    \hat{F}_{\nu}(0) \varphi(m_1m_2\nu) + O\left( \frac{X^{1+\varepsilon}}{Y}\right).
\end{split}
\end{equation*}
Writing the sum on $\alpha$ as an Euler product, we deduce that
\begin{equation}\label{eq: T0 before sieve}
\begin{split}
\mathcal{T}_0 =  \frac{4X}{3(\sqrt{2}-1)^4\zeta(2)}\sum_{\substack{ d\leq D \\ d \text{ odd} }} \frac{\mu^2(d)\lambda_d}{d} \prod_{p|d} \left(\frac{p}{p+1}\right)\mathop{\sum\sum}_{\substack{m_1,m_2\leq M \\ (m_1m_2,2d)=1}} \frac{b_{m_1}b_{m_2}}{\sqrt{m_1m_2}} \\
\times  \sum_{\substack{ \nu=1 \\ (\nu,2d)=1 \\ m_1m_2\nu=\square }}^{\infty} \frac{d_2(\nu) }{\sqrt{\nu}} \hat{F}_{\nu}(0) \prod_{p|m_1m_2\nu}\left( \frac{p}{p+1}\right)   + O\left( \frac{X^{1+\varepsilon}}{Y}\right).
\end{split}
\end{equation}
We next evaluate the sum over $d$. Lemma~\ref{sieve} implies
\begin{equation}\label{eq:k=0 subsection, sum on d}
\begin{split}
\sum_{\substack{ d\leq D \\ (d,2m_1m_2\nu)=1 }} \frac{\mu^2(d)\lambda_d}{d} \prod_{p|d} \left(\frac{p}{p+1}\right) = \frac{1+E_0(X)}{\log R}\prod_{\substack{p| 2m_1m_2\nu \\ p\leq z_0 }}\left( 1+\frac{1}{p}\right)\prod_{p\leq z_0}\left( \frac{p^2}{p^2-1}\right) \\
+ O\left( (\log R)^{-2018}\right).
\end{split}
\end{equation}
Recall that $E_0(X) \rightarrow 0$, and depends only on $X, G$, and $\vartheta$. Heretofore we just write $o(1)$ instead of $E_0(X)$.

We may omit the condition $p \leq z_0$ by trivial estimation and \eqref{eq:sieve section, defn of z0}. It follows from \eqref{eq:k=0 subsection, sum on d} and \eqref{eq: T0 before sieve} that
\begin{equation}\label{eq: T0 after sieve}
\begin{split}
\mathcal{T}_0 =  \frac{2X}{(\sqrt{2}-1)^4} \frac{1+o(1)}{\log R} \mathop{\sum\sum}_{\substack{m_1,m_2\leq M \\ (m_1m_2,2)=1}} \frac{b_{m_1}b_{m_2}}{\sqrt{m_1m_2}}     \sum_{\substack{ \nu=1 \\ (\nu,2)=1 \\ m_1m_2\nu=\square }}^{\infty} \frac{d_2(\nu) }{\sqrt{\nu}} \hat{F}_{\nu}(0) \\
 + O\left( \frac{X}{(\log R)^{2018}}+\frac{X^{1+\varepsilon}}{Y}\right).
\end{split}
\end{equation}

The next task is to carry out the summation over $m_1,m_2,$ and $\nu$. Let $\Upsilon_0$ be defined by
\begin{equation}\label{eq: defn of Upsilon 0}
\Upsilon_0 = \mathop{\sum\sum}_{\substack{m_1,m_2\leq M \\ (m_1m_2,2)=1}} \frac{b_{m_1}b_{m_2}}{\sqrt{m_1m_2}}     \sum_{\substack{ \nu=1 \\ (\nu,2)=1 \\ m_1m_2\nu=\square }}^{\infty} \frac{d_2(\nu) }{\sqrt{\nu}} \hat{F}_{\nu}(0).
\end{equation}
We insert into \eqref{eq: defn of Upsilon 0} the definition \eqref{eq: defn of mollifier coeffs bm} of $b_m$ and the definitions \eqref{Fdef} and \eqref{eq: defn of omega j} of $F_{\nu}$ and $\omega_2$, and then apply the Fourier inversion formula \eqref{eq: write H as Fourier integral}. After interchanging the order of summation, we arrive at
\begin{equation}\label{eq: Upsilon 0 after Fourier inversion}
\begin{split}
\Upsilon_0 = \frac{1}{2\pi i} \int_{(c)} \frac{\Gamma(\frac{s}{2}+\frac{1}{4})^2}{\Gamma(\frac{1}{4})^2}\left(1-\frac{1}{2^{\frac{1}{2}-s}}\right)^2\left( \frac{X}{ \pi} \right)^{s} \check{\Phi}(s) \int_{-\infty}^{\infty}\int_{-\infty}^{\infty}h(z_1)h(z_2)  \\
\times \mathop{\sum\sum\sum}_{\substack{ (m_1m_2\nu,2)=1 \\ m_1m_2\nu=\square}} \frac{\mu(m_1)\mu(m_2)d_2(\nu)}{(m_1m_2\nu)^{\frac{1}{2}} m_1^{\frac{1+iz_1}{\log M}}m_2^{\frac{1+iz_2}{\log M}}\nu^s}  \,dz_1dz_2\frac{ds}{s},
\end{split}
\end{equation}
where we take $c=\frac{1}{\log X}$ to facilitate later estimations. We may write the sum on $m_1,m_2,\nu$ as an Euler product
\begin{equation*}
\mathop{\sum\sum\sum}_{\substack{ (m_1m_2\nu,2)=1 \\ m_1m_2\nu=\square}} \frac{\mu(m_1)\mu(m_2)d_2(\nu)}{(m_1m_2\nu)^{\frac{1}{2}} m_1^{\frac{1+iz_1}{\log M}}m_2^{\frac{1+iz_2}{\log M}}\nu^s} =\prod_{p>2} \mathop{\sum_{m_1=0}^{1}\sum_{m_2=0}^{1}\sum_{\nu=0}^{\infty}}_{  m_1+m_2+\nu \text{ even}} \frac{(-1)^{m_1+m_2} (\nu+1)}{p^{\frac{m_1+m_2+\nu}{2} + m_1 \left(\frac{1+iz_1}{\log M} \right) +m_2\left(\frac{1+iz_2}{\log M} \right)+\nu s }}.
\end{equation*}
This can also be written as
\begin{equation}\label{eq: defn of Q for T0}
\zeta^3(1+2s)\zeta\Big(1+\tfrac{2+iz_1+iz_2}{\log M} \Big) \zeta^{-2} \Big( 1+ \tfrac{1+iz_1}{\log M} +s\Big) \zeta^{-2}\Big( 1+\tfrac{1+iz_2}{\log M}+s \Big)Q\left(\tfrac{1+iz_1}{\log M},\tfrac{1+iz_2}{\log M},s \right),
\end{equation}
where $Q(w_1,w_2,s)$ is an Euler product that is uniformly bounded and holomorphic when each of $\text{Re}(w_1)$, $\text{Re}(w_2)$, and $\text{Re}(s)$ is $\geq -\varepsilon$. From this definition of $Q$ and a calculation, we see that
\begin{equation}\label{eq: Q at 000}
Q(0,0,0)=1,
\end{equation}
a fact we use shortly. We insert the expression \eqref{eq: defn of Q for T0} for the $m_1,m_2,\nu$-sum into \eqref{eq: Upsilon 0 after Fourier inversion} and arrive at
\begin{equation*}
\begin{split}
& \Upsilon_0 = \frac{1}{2\pi i} \int_{(c)} \frac{\Gamma(\frac{s}{2}+\frac{1}{4})^2}{\Gamma(\frac{1}{4})^2}\left(1-\frac{1}{2^{\frac{1}{2}-s}}\right)^2\left( \frac{X}{ \pi} \right)^{s} \check{\Phi}(s)\zeta^3(1+2s) \int_{-\infty}^{\infty}\int_{-\infty}^{\infty}h(z_1)h(z_2) \\
& \times \zeta\Big(1+\tfrac{2+iz_1+iz_2}{\log M} \Big) \zeta^{-2} \Big( 1+ \tfrac{1+iz_1}{\log M} +s\Big) \zeta^{-2}\Big( 1+\tfrac{1+iz_2}{\log M}+s \Big)Q\left(\tfrac{1+iz_1}{\log M},\tfrac{1+iz_2}{\log M},s \right)  \,dz_1dz_2\frac{ds}{s}.
\end{split}
\end{equation*}
By \eqref{eq: h has rapid decay} and the rapid decay of the gamma function, we may truncate the integrals to the region $|z_1|,|z_2|\leq \sqrt{\log M}$ and $|\text{Im}(s)|\leq (\log X)^2$, introducing a negligible error. We then deform the path of integration of the $s$-integral to the path made up of the line segment $L_1$ from $\frac{1}{\log X}-i(\log X)^2$ to $-\frac{c'}{\log\log X}-i(\log X)^2$, followed by the line segment $L_2$ from $-\frac{c'}{\log\log X}-i(\log X)^2$ to $-\frac{c'}{\log\log X}+i(\log X)^2$, and then by the line segment $L_3$ from $-\frac{c'}{\log\log X}+i(\log X)^2$ to $\frac{1}{\log X}+i(\log X)^2$, where $c'$ is a constant chosen so that
\begin{equation}\label{eq: zeta bound in zero free region}
\zeta(1+z)\ll \log |\text{Im}(z)| \ \ \ \ \text{and} \ \ \ \ \frac{1}{\zeta(1+z)} \ll \log |\text{Im}(z)|
\end{equation}
for Re$(z)\geq -c'/\log|\text{Im}(z)|$ and $|\text{Im}(z)|\geq 1$ (see, for example, Theorem~3.5 and (3.11.8) of Titchmarsh~\cite{T}). This leaves a residue from the pole at $s=0$. The contributions of the integrals over $L_1$ and $L_3$ are negligible because of the rapid decay of the $\Gamma$ function, while the contribution of the integral over $L_2$ is negligible because $X^s\ll \exp\left( -c'\frac{\log X}{\log\log X}\right)$ for $s$ on $L_2$. Hence the main contribution arises from the residue of the pole at $s=0$. Writing this residue as an integral along a circle centered at $0$, we arrive at
\begin{equation*}
\begin{split}
\Upsilon_0 = \frac{1}{2\pi i} \oint_{|s|=\frac{1}{\log X}} \frac{\Gamma(\frac{s}{2}+\frac{1}{4})^2}{\Gamma(\frac{1}{4})^2}\left(1-\frac{1}{2^{\frac{1}{2}-s}}\right)^2\left( \frac{X}{ \pi} \right)^{s} \check{\Phi}(s)\zeta^3(1+2s)\ & \\
\times \mathop{\int \int}_{|z_i| \leq \sqrt{\log M}}h(z_1)h(z_2)\zeta\Big(1+\tfrac{2+iz_1+iz_2}{\log M} \Big) \zeta^{-2} \Big( 1+ \tfrac{1+iz_1}{\log M} +s\Big)
& \zeta^{-2}\Big( 1+\tfrac{1+iz_2}{\log M}+s \Big) \\
\times Q\left(\tfrac{1+iz_1}{\log M},\tfrac{1+iz_2}{\log M},s \right)  \,dz_1dz_2\frac{ds}{s}
& + O\left( \frac{1}{(\log X)^{2018}}\right).
\end{split}
\end{equation*}
We may expand the zeta-functions and the function $Q$ into Laurent series. The main contribution arises from the first terms of the Laurent expansions, and so we deduce using \eqref{eq: Q at 000} that
\begin{equation*}
\begin{split}
\Upsilon_0 = \frac{1}{16\pi i} \oint_{|s|=\frac{1}{\log X}} \frac{\Gamma(\frac{s}{2}+\frac{1}{4})^2}{\Gamma(\frac{1}{4})^2}\left(1-\frac{1}{2^{\frac{1}{2}-s}}\right)^2\left( \frac{X}{ \pi} \right)^{s} \check{\Phi}(s) \mathop{\int \int}_{|z_i| \leq \sqrt{\log M}}h(z_1)h(z_2)\\
\times \left(\frac{\log M}{2+iz_1+iz_2} \right) \left( \frac{1+iz_1}{\log M} +s\right)^2 \left(\frac{1+iz_2}{\log M}+s \right)^2\,dz_1dz_2\frac{ds}{s^4} + O\left( \frac{1}{(\log X)^{1-\varepsilon}}\right).
\end{split}
\end{equation*}
By \eqref{eq: h has rapid decay}, we may extend the integrals over $z_1,z_2$ to $\mathbb{R}^2$, introducing a negligible error. We then apply the formula
\begin{equation}\label{eq: integrals of derivatives of H}
\begin{split}
& \int_{-\infty}^{\infty}\int_{-\infty}^{\infty}  h(z_1)h(z_2)\frac{(1+iz_1)^j(1+iz_2)^k}{2+iz_1+iz_2}\, dz_1dz_2 \\
&= \int_0^{\infty} \int_{-\infty}^{\infty}\int_{-\infty}^{\infty}  h(z_1)h(z_2) (1+iz_1)^j(1+iz_2)^k e^{-t(1+iz_1)-t(1+iz_2)}\, dz_1dz_2 dt \\
& = (-1)^{j+k}\int_0^{\infty} H^{(j)}(t) H^{(k)}(t)\,dt
\end{split}
\end{equation}
to obtain
\begin{equation*}
\begin{split}
\Upsilon_0 = \frac{1}{16\pi i} \oint_{|s|=\frac{1}{\log X}} \frac{\Gamma(\frac{s}{2}+\frac{1}{4})^2}{\Gamma(\frac{1}{4})^2}\left(1-\frac{1}{2^{\frac{1}{2}-s}}\right)^2\left( \frac{X}{ \pi} \right)^{s} \check{\Phi}(s)\Bigg\{ \frac{1}{(\log M)^3} \int_0^1 H''(t)^2\,dt \\
 -  \frac{4s}{(\log M)^2} \int_0^1 H'(t) H''(t)\,dt + \frac{2s^2}{\log M} \int_0^1 H(t) H''(t)\,dt + \frac{4s^2}{\log M} \int_0^1 H'(t)^2 \,dt \\
- 4s^3\int_0^1 H(t)H'(t) \,dt + s^4\log M \int_0^1 H(t)^2\,dt \Bigg\}\,\frac{ds}{s^4}  + O\left( \frac{1}{(\log X)^{1-\varepsilon}}\right).
\end{split}
\end{equation*}
We evaluate the $s$-integral as a residue using \eqref{eq: residue as derivative}. The result is
\begin{equation*}
\begin{split}
& \Upsilon_0 = \frac{\check{\Phi}(0)}{8}\left(1-\frac{1}{\sqrt{2}}\right)^2 \Bigg\{ \frac{1}{6}\left(\frac{\log X}{\log M}\right)^3 \int_0^1 H''(t)^2\,dt  -  2\left(\frac{\log X}{\log M}\right)^2 \int_0^1 H'(t) H''(t)\,dt \\
& + 2\frac{\log X}{\log M} \int_0^1 H(t) H''(t)\,dt + 4\frac{\log X}{\log M} \int_0^1 H'(t)^2 \,dt  - 4\int_0^1 H(t)H'(t) \,dt  \Bigg\}  + O\left( \frac{1}{(\log X)^{1-\varepsilon}}\right).
\end{split}
\end{equation*}
From this, \eqref{eq: T0 after sieve}, and the definition \eqref{eq: defn of Upsilon 0} of $\Upsilon_0$, we arrive at
\begin{equation}\label{eq: T0 after mollifier}
\begin{split}
\mathcal{T}_0 &=  \frac{X}{8\left(1-\frac{1}{\sqrt{2}}\right)^2} \frac{1+o(1)}{\log R} \Bigg\{ \frac{1}{24}\left(\frac{\log X}{\log M}\right)^3 \int_0^1 H''(t)^2\,dt  \\
&-  \frac{1}{2}\left(\frac{\log X}{\log M}\right)^2 \int_0^1 H'(t) H''(t)\,dt +
 \frac{\log X}{2\log M} \int_0^1 H(t) H''(t)\,dt + \frac{\log X}{\log M} \int_0^1 H'(t)^2 \,dt  \\
&- \int_0^1 
 H(t)H'(t) \,dt  \Bigg\} + O\left( \frac{X}{(\log X)^{1-\varepsilon}}+ \frac{X^{1+\varepsilon}}{Y}\right).
\end{split}
\end{equation}

\subsection{The contribution from $k\neq 0$: splitting into cases}

Having estimated the term $\mathcal{T}_0$ in \eqref{SM3}, we now begin our analysis of $\mathcal{B}$. The analysis of $\mathcal{B}$ is much more complicated than the analysis for $\mathcal{T}_0$.

The behavior of the additive character $e({k\overline{\alpha^2d_1m_1m_2\nu}}/{8})$ in \eqref{SM2} depends upon the residue class of $k$ modulo 8. We therefore distinguish the following cases for $k$: $k$ is odd, $k \equiv 2 \pmod 4$, $k \equiv 4 \pmod 8$, or $k\equiv 0 \pmod 8$. We split our analysis of the sum $\mathcal{B}$ in \eqref{SM3} according to these four cases. For the terms with odd $k$, we use the identity
\begin{equation*}
e\left(\frac{h}{8}\right) \ = \ \frac{\sqrt{2}}{2}\left( \frac{2}{h}\right) \ + \ \frac{\sqrt{2}}{2}\left( \frac{-2}{h}\right)i, \ \ \ \ h \text{ odd},
\end{equation*}
and treat separately the contributions of each term on the right-hand side. Moreover, for the terms with odd $k$ or $k\equiv 2 \pmod 4$, we use the second expression in \eqref{Gaussdef} for $\tau_k(n)$ and treat separately the contributions of the terms $\left( \frac{1+i}{2} \right) G_k(n)$ and $\left(\frac{-1}{n}\right)\left(\frac{1-i}{2}\right) G_k(n)$. We can treat these two contributions together as one combined sum for the terms with $k\equiv 0,4 \pmod 8$, because, for those $k$, the additive character $e({k\overline{\alpha^2d_1m_1m_2\nu}}/{8})$ is constant and the conditions $k\equiv 0,4 \pmod 8$ are invariant with respect to the substitution $k\mapsto -k$. Hence, in view of these considerations, \eqref{SM2}, and \eqref{SM3}, we write
\begin{align}\label{R0split}
\begin{split}
\mathcal{B} = \frac{X}{(\sqrt{2}-1)^4}\sum_{\substack{ d\leq D \\ d \text{ odd} }} \mu^2(d)\lambda_d &\mathop{\sum\sum}_{\substack{m_1,m_2\leq M \\ (m_1m_2,2d)=1}} \frac{b_{m_1}b_{m_2}}{(m_1m_2)^{3/2}} \sum_{\substack{ \nu=1 \\ (\nu,2d)=1}}^{\infty} \frac{d_2(\nu) }{\nu^{3/2}} \ \sum_{\substack{ \alpha\leq Y \\ (\alpha,2m_1m_2\nu)=1}} \frac{\mu(\alpha)}{\alpha^2d_1} \\
&\times (\mathcal{Q}_1 + \mathcal{Q}_2 + \mathcal{Q}_3 + \mathcal{Q}_4 + \mathcal{U}_1 + \mathcal{U}_2 + \mathcal{V} + \mathcal{W}),
\end{split}
\end{align}
where
\begin{equation}\label{Q1}
\mathcal{Q}_1 = \left(\frac{1+i}{2}\right)\frac{\sqrt{2}}{2} \left( \frac{2d_1}{m_1m_2\nu}\right)    \sum_{\substack{k \in\mathbb{Z} \\ k\text{ odd}}}\left(\frac{2}{kd_1m_1m_2\nu}\right)\hat{F}_{\nu}\left( \frac{kX}{8\alpha^2d_1m_1m_2\nu}\right) G_k(m_1m_2\nu),
\end{equation}
\begin{equation}\label{Q2}
\mathcal{Q}_2 = \left(\frac{1-i}{2}\right)\frac{\sqrt{2}}{2}\left( \frac{-2d_1}{m_1m_2\nu}\right)    \sum_{\substack{k \in\mathbb{Z} \\ k\text{ odd}}}\left(\frac{2}{kd_1m_1m_2\nu}\right)\hat{F}_{\nu}\left( \frac{kX}{8\alpha^2d_1m_1m_2\nu}\right) G_k(m_1m_2\nu),
\end{equation}
\begin{equation}\label{Q3}
\mathcal{Q}_3 = \left(\frac{1+i}{2}\right)i\frac{\sqrt{2}}{2}\left( \frac{2d_1}{m_1m_2\nu}\right)    \sum_{\substack{k \in\mathbb{Z} \\ k\text{ odd}}}\left(\frac{-2}{kd_1m_1m_2\nu}\right)\hat{F}_{\nu}\left( \frac{kX}{8\alpha^2d_1m_1m_2\nu}\right) G_k(m_1m_2\nu),
\end{equation}
\begin{equation}\label{Q4}
\mathcal{Q}_4 = \left(\frac{1-i}{2}\right)i\frac{\sqrt{2}}{2}\left( \frac{-2d_1}{m_1m_2\nu}\right)    \sum_{\substack{k \in\mathbb{Z} \\ k\text{ odd}}}\left(\frac{-2}{kd_1m_1m_2\nu}\right)\hat{F}_{\nu}\left( \frac{kX}{8\alpha^2d_1m_1m_2\nu}\right) G_k(m_1m_2\nu),
\end{equation}
\begin{equation}\label{U1}
\mathcal{U}_1 = \left(\frac{1+i}{2}\right)\left( \frac{2d_1}{m_1m_2\nu}\right)    \sum_{\substack{k \in\mathbb{Z} \\ k \equiv 2 \, (\text{mod } 4)}}e\left(\frac{k\overline{\alpha^2d_1m_1m_2\nu}}{8}\right)\hat{F}_{\nu}\left( \frac{kX}{8\alpha^2d_1m_1m_2\nu}\right) G_k(m_1m_2\nu),
\end{equation}
\begin{equation}\label{U2}
\mathcal{U}_2 = \left(\frac{1-i}{2}\right)\left( \frac{-2d_1}{m_1m_2\nu}\right)    \sum_{\substack{k \in\mathbb{Z} \\ k \equiv 2 \, (\text{mod } 4)}}e\left(\frac{k\overline{\alpha^2d_1m_1m_2\nu}}{8}\right)\hat{F}_{\nu}\left( \frac{kX}{8\alpha^2d_1m_1m_2\nu}\right) G_k(m_1m_2\nu),
\end{equation}
\begin{equation}\label{V}
\mathcal{V} = \left( \frac{2d_1}{m_1m_2\nu}\right)    \sum_{\substack{k \in\mathbb{Z} \\ k \equiv 4 \, (\text{mod } 8)}}e\left(\frac{k\overline{\alpha^2d_1m_1m_2\nu}}{8}\right)\hat{F}_{\nu}\left( \frac{kX}{8\alpha^2d_1m_1m_2\nu}\right) {\tau}_k(m_1m_2\nu),
\end{equation}
and
\begin{equation}\label{W}
\mathcal{W} =\left( \frac{2d_1}{m_1m_2\nu}\right)    \sum_{\substack{k \in\mathbb{Z} \\ k \equiv 0 \, (\text{mod } 8) \\ k\neq 0}}e\left(\frac{k\overline{\alpha^2d_1m_1m_2\nu}}{8}\right)\hat{F}_{\nu}\left( \frac{kX}{8\alpha^2d_1m_1m_2\nu}\right) {\tau}_k(m_1m_2\nu).
\end{equation}

\subsection{Evaluation of the sum with $\mathcal{Q}_1$}

In this subsection, we evaluate the sum
\begin{equation}\label{eq: defn of Q1 star}
\mathcal{Q}_1^* := \sum_{\substack{ \nu=1 \\ (\nu,2d)=1}}^{\infty} \frac{d_2(\nu) }{\nu^{3/2}} \ \sum_{\substack{ \alpha\leq Y \\ (\alpha,2m_1m_2\nu)=1}} \frac{\mu(\alpha)}{\alpha^2d_1}\, \mathcal{Q}_1,
\end{equation}
with $\mathcal{Q}_1$ defined by \eqref{Q1}. We may cancel the two Jacobi symbols $(\frac{2}{m_1m_2\nu})$ in \eqref{Q1}, insert the resulting expression into \eqref{eq: defn of Q1 star}, and then apply the Mellin inversion formula to the $\nu$-sum to deduce that
\begin{equation}\label{Q12}
\begin{split}
& \mathcal{Q}_1^* = \left(\frac{1+i}{2}\right)\frac{\sqrt{2}}{2}\ \sum_{\substack{ \alpha\leq Y \\ (\alpha,2m_1m_2\nu)=1}} \frac{\mu(\alpha)}{\alpha^2d_1} \left( \frac{d_1}{m_1m_2}\right)  \sum_{\substack{k \in\mathbb{Z} \\ k\text{ odd}}}  \left(\frac{2}{k d_1}\right) \\
& \times  \frac{1}{2\pi i} \int_{(c)}  \int_0^{\infty} \hat{F}_{t}\left( \frac{kX}{8\alpha^2d_1m_1m_2t}\right) t^{w-1} \,dt \ \sum_{\substack{ \nu=1 \\ (\nu,2d)=1}}^{\infty} \frac{d_2(\nu) }{\nu^{3/2+w}} \left( \frac{d_1}{\nu}\right) G_k(m_1m_2\nu) \,dw
\end{split}
\end{equation}
for any $c>1$. The interchange in the order of summation is justified by absolute convergence. The next step is to write the $\nu$-sum as an Euler product, as follows.

\begin{lem}\label{lem: nu-sum as an Euler product}
Let $d_1$ be as defined by \eqref{eq: defn of d 1}. For each nonzero integer $k$, define $k_1$ and $k_2$ uniquely by the equation
\begin{equation}\label{eq: defn of k1 and k2}
4kd_1 = k_1 k_2^2,
\end{equation}
where $k_1$ is a fundamental discriminant and $k_2$ is a positive integer. If $\ell$ is a positive integer and Re$(s)>1$, then
$$
\sum_{\substack{ \nu=1 \\ (\nu,2\alpha d)=1}}^{\infty} \frac{d_2(\nu) }{\nu^{s}} \left( \frac{d_1}{\nu}\right) \frac{G_k(\ell \nu)}{\sqrt{\nu}} \ = \ L(s,\chi_{k_1})^2\prod_p\mathcal{G}_{0,p}(s;k,\ell,\alpha, d) \ =: \ L(s,\chi_{k_1})^2\mathcal{G}_0(s;k,\ell,\alpha, d),
$$
where $\chi_{k_1}(\cdot)=\left( \frac{k_1}{\cdot}\right)$ and $\mathcal{G}_{0,p}(s;k,\ell,\alpha, d)$ is defined by
$$
\mathcal{G}_{0,p}(s;k,\ell,\alpha, d) \ = \ \Bigg( 1-\frac{1}{p^{s}}\Bigg(\frac{k_1}{p}\Bigg)\Bigg)^2 \ \ \ \ \text{if }p|2\alpha d, \ \ \ \text{and}
$$
$$
\mathcal{G}_{0,p}(s;k,\ell,\alpha, d) \ = \ \Bigg( 1-\frac{1}{p^{s}}\Bigg(\frac{k_1}{p}\Bigg)\Bigg)^2\sum_{r=0}^{\infty} \frac{r+1 }{p^{rs}}\left( \frac{d_1}{p^r}\right)\frac{{G}_k(p^{r+\text{\upshape{ord}}_p(\ell)})}{p^{r/2}}\ \ \ \ \text{if }p\nmid 2\alpha d.
$$
The function $\mathcal{G}_0(s;k,\ell,\alpha, d)$ is holomorphic for Re$(s)>\frac{1}{2}$. Moreover, if $k_3$ and $k_4$ are defined by the equation
\begin{equation}\label{eq: defn of k3 and k4}
k = k_3 k_4^2,
\end{equation}
with $k_3$ square-free and $k_4$ a positive integer, then
$$
\mathcal{G}_0(s;k,\ell,\alpha, d) \ll_{\varepsilon} (\alpha d|k|\ell)^{\varepsilon} \ell^{1/2} (\ell,k_4^2)^{1/2}
$$
uniformly for Re$(s)\geq \frac{1}{2}+\varepsilon$.
\end{lem}
\begin{proof}
It follows from the definition of $\mathcal{G}_{0,p}(s;k,\ell,\alpha, d)$ and Lemma~\ref{lem: properties of Gkn} that
$$
\mathcal{G}_{0,p}(s;k,\ell,\alpha, d) = \left( 1-\frac{1}{p^s} \left( \frac{k_1}{p}\right)\right)^2 \left(1+\frac{2}{p^s} \left( \frac{d_1k}{p}\right)\right) = 1-\frac{3}{p^{2s}}+\frac{2}{p^{3s}}\left( \frac{k_1}{p}\right)
$$
for each $p\nmid 2\alpha d k\ell$, since $ \left( \frac{d_1k}{p}\right)= \left( \frac{k_1}{p}\right)$ for odd primes $p$, by \eqref{eq: defn of k1 and k2}. The rest of the proof is similar to that of \cite[Lemma 5.3]{Sou00}.
\end{proof}

We also need some analytic properties of the function $h(\xi,w)$ defined for Re$(w)>0$ by
\begin{equation*}
h(\xi,w) = \int_0^{\infty} \hat{F}_{t}\left( \frac{\xi}{t}\right) t^{w-1} \,dt.
\end{equation*}
These are embodied in the following lemma. As a bit of notation, for a real number $x$ we define
\begin{align*}
\text{sgn}(x) =
\begin{cases}
1, \ \ \ \ \ \ \ &x \geq 0, \\
-1, &x <0.
\end{cases}
\end{align*}

\begin{lem}\label{lem: properties of h(xi,w)}
Let $F_t$ be defined by \eqref{Fdef}. If $\xi\neq 0$ then
$$
h(\xi,w)=|\xi|^w \check{\Phi}(w) \int_0^{\infty} \omega_2\Bigg( \frac{|\xi|\pi}{Xz} \Bigg) (\cos(2\pi z) -i\text{\upshape{sgn}}(\xi)\sin (2\pi z) ) \, \frac{dz}{z^{w+1}}.
$$
The integral above may be expressed as
\begin{equation}\label{eq: integral in lemma for h(xi,w)}
\begin{split}
& \frac{1}{2\pi i} \int_{(c)} \frac{\Gamma\left( \frac{s}{2}+\frac{1}{4}\right)^2}{\Gamma\left(\frac{1}{4}\right)^2}\left(1-\frac{1}{2^{\frac{1}{2}-s}} \right)^2\frac{X^s}{(\pi|\xi|)^s}  (2\pi)^{-s+w} \Gamma(s-w) \\
& \times \left\{ \cos \left(\tfrac{\pi}{2}(s-w)\right)-i\mbox{\upshape{sgn}}(\xi) \sin \left(\tfrac{\pi}{2}(s-w)\right) \right\} \,\frac{ds}{s}
\end{split}
\end{equation}
for any $c$ with \upshape{Re}$(w)+1>c>\max\{0,\text{\upshape{Re}}(w)\}$. If $\xi\neq 0$, then $h(\xi,w)$ is an entire function of $w$. In the region $1\geq \text{\upshape{Re}}(w)> -1$, it satisfies the bound
$$
h(\xi,w) \ll (1+|w|)^{-\text{Re}(w)-\frac{1}{2}} \exp\Bigg( -\frac{1}{10}\frac{\sqrt{|\xi|}}{\sqrt{X(|w|+1)}}\Bigg) |\xi|^w |\check{\Phi}(w)|.
$$
\end{lem}
\begin{proof}
The proof is similar to that of \cite[Lemma 5.2]{Sou00}.
\end{proof}

Now, by these lemmas and the rapid decay of $\check{\Phi}(w)$ as $|\text{Im}(w)|\rightarrow \infty$ in a fixed vertical strip, we may move the line of integration of the $w$-integral in \eqref{Q12} to Re$(w)=-\frac{1}{2}+\varepsilon$. This leaves a residue from a pole at $w=0$ only when $\chi_{k_1}$ is a principal character, which holds if and only if $k_1=1$. By \eqref{eq: defn of k1 and k2}, $k_1=1$ if and only if $kd_1$ is a perfect square. Hence
\begin{equation}\label{Q14}
\mathcal{Q}_1^* = \mathcal{P}_1 + \mathcal{R}_1,
\end{equation}
where $\mathcal{P}_1$ is defined by
\begin{equation}\label{eq: defn of P1}
\begin{split}
\mathcal{P}_1 = \underset{w=0}{\mbox{Res}} \  \left(\frac{1+i}{2}\right)\frac{\sqrt{2}}{2}\sum_{\substack{ \alpha\leq Y \\ (\alpha,2m_1m_2)=1}} \frac{\mu(\alpha)}{\alpha^2d_1}   \left( \frac{d_1}{m_1m_2}\right)\sum_{\substack{k \in\mathbb{Z} \\ k\text{ odd} \\ kd_1=\square}} \  h\left( \frac{kX}{8\alpha^2 d_1 m_1m_2},w\right) & \\
\times   \zeta(1+w)^2\mathcal{G}_{0}(1+w;k,m_1m_2,\alpha, d) &
\end{split}
\end{equation}
and $\mathcal{R}_1$ is defined by
\begin{equation}\label{eq: defn of R1}
\begin{split}
& \mathcal{R}_1 = \left(\frac{1+i}{2}\right)\frac{\sqrt{2}}{2} \sum_{\substack{ \alpha\leq Y \\ (\alpha,2m_1m_2)=1}} \frac{\mu(\alpha)}{\alpha^2d_1}\left( \frac{d_1}{m_1m_2}\right)  \sum_{\substack{k \in\mathbb{Z} \\ k\text{ odd}}}\left(\frac{2}{k d_1}\right)  \\
& \times   \frac{1}{2\pi i} \int_{(-\frac{1}{2}+\varepsilon)} h\left( \frac{kX}{8\alpha^2 d_1 m_1m_2},w\right) L(1+w,\chi_{k_1})^2\mathcal{G}_{0}(1+w;k,m_1m_2,\alpha, d) \,dw .
\end{split}
\end{equation}
We bound $\mathcal{R}_1$ in Subsection \ref{subsec: error term}. To estimate $\mathcal{P}_1$, observe that $d_1$ is square-free by its definition \eqref{eq: defn of d 1} and the fact that $d$ is square-free. This implies that $kd_1$ is a perfect square if and only if $k$ equals $d_1$ times a perfect square. Hence, in \eqref{eq: defn of P1}, we may relabel $k$ as $d_1j^2$, where $j$ runs through all the odd positive integers. With this and Lemma~\ref{lem: properties of h(xi,w)}, we deduce from \eqref{eq: defn of P1} that
\begin{equation}\label{P12}
\begin{split}
\mathcal{P}_1 = \underset{w=0}{\mbox{Res}} \  \left(\frac{1+i}{2}\right)\frac{\sqrt{2}}{2}\sum_{\substack{ \alpha\leq Y \\ (\alpha,2m_1m_2)=1}} \frac{\mu(\alpha)}{\alpha^{2}d_1} \zeta(1+w)^2 \check{\Phi}(w)X^w  \frac{1}{2\pi i} \int_{(c)}  \frac{\Gamma\left( \frac{s}{2}+\frac{1}{4}\right)^2}{\Gamma\left(\frac{1}{4}\right)^2}  \left(1-\frac{1}{2^{\frac{1}{2}-s}} \right)^2  \\
\times \pi^{-s}  \Gamma_2(s-w)(8\alpha^2m_1m_2)^{s-w}  \sum_{\substack{j =1 \\ j\text{ odd} }}^{\infty} j^{-2s+2w}   \left( \frac{d_1}{m_1m_2}\right) \mathcal{G}_{0}(1+w;d_1j^2,m_1m_2,\alpha, d)\,\frac{ds}{s} ,
\end{split}
\end{equation}
where $\Gamma_2(u)$ is defined by
\begin{equation}\label{eq: defn of Gamma 2}
\Gamma_2(u)= (2\pi)^{-u} \Gamma(u) ( \cos \left(\tfrac{\pi}{2}u\right)-i \sin \left(\tfrac{\pi}{2}u\right) ),
\end{equation}
and where we take $c>\frac{1}{2}$ to guarantee the absolute convergence of the $j$-sum.

We next write the $j$-sum in \eqref{P12} as an Euler product. By (ii) of Lemma~\ref{lem: properties of Gkn}, if $j$ is a positive integer then
$$
\left( \frac{d_1}{p^{\beta}}\right) G_{d_1j^2}(p^{\beta}) = G_{j^2}(p^{\beta})
$$
for all $p\nmid 2\alpha d$ and $\beta\geq 1$. From this and the definition of $\mathcal{G}_0$ in Lemma~\ref{lem: nu-sum as an Euler product}, we see that
$$
\left( \frac{d_1}{m_1m_2}\right) \mathcal{G}_{0}(1+w;d_1j^2,m_1m_2,\alpha, d) = \mathcal{G}(1+w;j^2,m_1m_2,\alpha d),
$$
where $\mathcal{G}$ is defined by \cite[(5.8)]{Sou00}. Hence we may write the inner $j$-sum in \eqref{P12} as an Euler product
\begin{equation}\label{eq: write inner j sum as Euler product}
\begin{split}
\sum_{\substack{ j=1 \\ j \text{ odd} }}^{\infty}j^{-2s+2w} \mathcal{G}(1+w;j^2,m_1m_2,\alpha d)
& = \left(1-\frac{1}{2^{1+w}}\right)^2 \prod_{p>2} \sum_{b=0}^{\infty} p^{2b(w-s)} \mathcal{G}_p(1+w;p^{2b},m_1m_2,\alpha d)\\
& = \left(1-\frac{1}{4^{s-w}}\right)\prod_{p} \sum_{b=0}^{\infty} p^{2b(w-s)} \mathcal{G}_p(1+w;p^{2b},m_1m_2,\alpha d).
\end{split}
\end{equation}
This latter expression is \cite[p. 471]{Sou00}
$$
\left(1-\frac{1}{4^{s-w}}\right) (m_1m_2)^{1-s+w} \ell_1^{s-w-\frac{1}{2}} \zeta(2s-2w)\zeta(2s+1) \mathcal{H}_1(s-w,1+w;m_1m_2,\alpha d),
$$
where $\ell_1$ is the square-free integer defined by the equation
\begin{equation}\label{eq: defn of ell 1}
m_1m_2 = \ell_1 \ell_2^2, \ \ \ \ \mu^2(\ell_1)=1, \ \ell_2\in\mathbb{Z},
\end{equation}
and $\mathcal{H}_1$ is defined by an Euler product
$$
\mathcal{H}_1(s-w,1+w;m_1m_2,\alpha d) = \prod_p \mathcal{H}_{1,p}.
$$
The local factors $\mathcal{H}_{1,p}$ are
\begin{equation}\label{eq: defn of H 1}
\mathcal{H}_{1,p} =
\begin{cases}
\left( 1-\frac{1}{p^{1+w}}\right)^2\left( 1-\frac{1}{p^{1+2s}}\right) & \text{if } p|2\alpha d\\
\frac{\left( 1-\frac{1}{p^{1+w}}\right)^2}{ \left( 1- \frac{1}{p^{1+2s}}\right)}\left( 1+\frac{2}{p^{1+w}} -\frac{2}{p^{1+2s-w}}+\frac{1}{p^{1+2s}}-\frac{3}{p^{2+2s}}+\frac{1}{p^{3+4s}}\right) & \text{if } p\nmid 2\alpha d m_1m_2 \\
\frac{\left( 1-\frac{1}{p^{1+w}}\right)^2}{ \left( 1- \frac{1}{p^{1+2s}}\right)}\left( 1-\frac{1}{p^{2s-2w}} +\frac{2}{p^{2s-w}}-\frac{2}{p^{1+2s-w}}+\frac{1}{p^{1+2s}}-\frac{1}{p^{1+4s-2w}}\right) & \text{if } p| \ell_1 \\
\frac{\left( 1-\frac{1}{p^{1+w}}\right)^2}{ \left( 1- \frac{1}{p^{1+2s}}\right)}\left( 1-\frac{1}{p} +\frac{2}{p^{1+w}}-\frac{2}{p^{1+2s-w}}+\frac{1}{p^{1+2s}}-\frac{1}{p^{2+2s}}\right) & \text{if } p|m_1m_2, \ p\nmid \ell_1.
\end{cases}
\end{equation}
Inserting this expression for the $j$-sum in \eqref{P12} into \eqref{P12}, we find that
\begin{equation}\label{P13}
\mathcal{P}_1 = \left(\frac{1+i}{2}\right)\frac{\sqrt{2}}{2}\sum_{\substack{ \alpha\leq Y \\ (\alpha,2m_1m_2)=1}} \frac{\mu(\alpha)}{\alpha^{2}d_1} \ \mathcal{I},
\end{equation}
where
\begin{equation}\label{eq: defn of I for P1}
\begin{split}
\mathcal{I} = \underset{w=0}{\mbox{Res}} \ \zeta(1+w)^2 \check{\Phi}(w)X^w  \frac{1}{2\pi i} \int_{(c)}  \frac{\Gamma\left( \frac{s}{2}+\frac{1}{4}\right)^2}{\Gamma\left(\frac{1}{4}\right)^2}   \left(1-\frac{1}{2^{\frac{1}{2}-s}} \right)^2  \pi^{-s}  \Gamma_2(s-w)(8\alpha^2)^{s-w}  \\
\times \left(1-\frac{1}{4^{s-w}}\right) m_1m_2 \ell_1^{s-w-\frac{1}{2}}  \zeta(2s-2w)\zeta(2s+1) \mathcal{H}_1(s-w,1+w;m_1m_2,\alpha d)\,\frac{ds}{s}.
\end{split}
\end{equation}

The next step is to extend the $\alpha$-sum to infinity and show that the error introduced in doing so is small. To do this, we need to move the line of integration in \eqref{eq: defn of I for P1} closer to $0$ to guarantee the absolute convergence of the $\alpha$-sum. We first evaluate the residue to see that \eqref{eq: defn of I for P1} is the same as
\begin{equation}\label{I2}
\begin{split}
\mathcal{I} =  \frac{ \check{\Phi}(0) }{2\pi i} \int_{(c)}  \frac{\Gamma\left( \frac{s}{2}+\frac{1}{4}\right)^2}{\Gamma\left(\frac{1}{4}\right)^2}   \left(1-\frac{1}{2^{\frac{1}{2}-s}} \right)^2  \pi^{-s}  \Gamma_2(s)(8\alpha^2)^{s} \left(1-\frac{1}{4^{s}}\right) m_1m_2  \\
\times   \ell_1^{s-\frac{1}{2}} \zeta(2s)\zeta(2s+1) \mathcal{H}_1(s,1;m_1m_2,\alpha d) \Bigg\{2\gamma+\frac{(\check{\Phi})'(0)}{\check{\Phi}(0)} + \log\Bigg(\frac{X}{8\alpha^2\ell_1}\Bigg) \\
 -\frac{\Gamma_2'}{\Gamma_2}(s) +\frac{\log 4}{( 1-4^{s})}-2\frac{\zeta'}{\zeta}(2s)+\frac{\frac{\partial}{\partial w}\mathcal{H}_1(s-w,1+w;m_1m_2,\alpha d)}{\mathcal{H}_1(s-w,1+w;m_1m_2,\alpha d)}\Bigg|_{w=0}   \Bigg\}\,\frac{ds}{s}.
\end{split}
\end{equation}
Here $\gamma$ denotes the Euler-Mascheroni constant. The definition \eqref{eq: defn of H 1} of $\mathcal{H}_1(s-w,1+w;m_1m_2,\alpha d)$ implies that it is holomorphic for Re$(s)>0$ and $|w|<\max\{\frac{1}{2},2|s|\}$, and that it and its first partial derivatives at $w=0$ are bounded by $\ll (\alpha X)^{\varepsilon}$ for Re$(s)\geq \frac{1}{\log X}$. Thus, by the rapid decay of the gamma function, we may move the line of integration in \eqref{I2} to Re$(s)=\frac{1}{\log X}$. There is no residue because the poles of $\zeta(2s)$ and $\frac{\zeta'}{\zeta}(2s)$ at $s=\frac{1}{2}$ are canceled by the zero of the factor $(1-2^{s-\frac{1}{2}})^2$. Using well-known bounds for $\zeta(2s)$ and $\zeta'(2s)$ implied by the Phragm\'en-Lindel\"{o}f principle, we see that the new integral is now bounded by
$$
\ll m_1m_2\ell_1^{-\frac{1}{2}+\varepsilon} \alpha^{\varepsilon} X^{\varepsilon}\int_{\left( \frac{1}{\log X}\right)} \left|\Gamma\left( \tfrac{s}{2}+\tfrac{1}{4}\right)\right|^2 \max\{|\Gamma_2(s)|,|\Gamma_2'(s)|\}(1+|s|)^{\frac{1}{2}+\varepsilon} \,|ds|,
$$
which is $\ll m_1m_2\ell_1^{-\frac{1}{2}+\varepsilon}\alpha^{\varepsilon} X^{\varepsilon}$ by the rapid decay of the gamma function. Dividing this bound by $\alpha^2 d_1$ and summing the result over all $\alpha>Y$, we deduce that
\begin{equation}\label{P13.5}
\sum_{\substack{ \alpha>Y \\ (\alpha,2m_1m_2)=1}} \frac{\mu^2(\alpha)}{\alpha^{2}d_1} \ |\mathcal{I}| \ll \frac{m_1 m_2 \ell_1^{-\frac{1}{2}+\varepsilon} X^{\varepsilon}}{d^{1-\varepsilon}Y^{1-\varepsilon}}
\end{equation}
because, by \eqref{eq: defn of d 1}, if $\varphi(j)$ is the Euler totient function, then
$$
\sum_{\alpha>Y} \frac{1}{\alpha^{2-\varepsilon} d_1} = \frac{1}{d} \sum_{j|d} \varphi(j) \sum_{\substack{\alpha>Y \\ j|\alpha}} \frac{1}{\alpha^{2-\varepsilon}} \ll \frac{1}{d^{1-\varepsilon}Y^{1-\varepsilon}}.
$$
From \eqref{P13}, \eqref{P13.5}, and \eqref{eq: defn of I for P1} now with $c=\frac{1}{\log X}$, we arrive at
\begin{equation}\label{P14}
\begin{split}
\mathcal{P}_1 = \underset{w=0}{\mbox{Res}} \   \left(\frac{1+i}{2}\right)\frac{\sqrt{2}}{2} X^w\frac{1}{2\pi i} \int_{\left( \frac{1}{\log X}\right)}    \Gamma_2(s-w)8^{s-w} \left(1-\frac{1}{4^{s-w}}\right) \\
\times \mathcal{K}(s,w;m_1m_2, d)\,\frac{ds}{s} + O\Bigg( \frac{m_1 m_2 \ell_1^{-\frac{1}{2}+\varepsilon} X^{\varepsilon}}{d^{1-\varepsilon}Y^{1-\varepsilon}} \Bigg),
\end{split}
\end{equation}
with $\mathcal{K}(s,w;m_1m_2, d)$ defined by
\begin{equation}\label{eq: defn of K for 2nd moment analysis}
\begin{split}
\mathcal{K}(s,w;m_1m_2, d) =  \zeta(1+w)^2 \check{\Phi}(w) \frac{\Gamma\left( \frac{s}{2}+\frac{1}{4}\right)^2}{\Gamma\left(\frac{1}{4}\right)^2}   \left(1-\frac{1}{2^{\frac{1}{2}-s}} \right)^2 \pi^{-s} m_1m_2  \ell_1^{s-w-\frac{1}{2}} \\
\times   \zeta(2s-2w)\zeta(2s+1) \sum_{\substack{ \alpha=1 \\ (\alpha,2m_1m_2)=1}}^{\infty} \frac{\mu(\alpha)}{\alpha^{2-2s+2w}d_1} \mathcal{H}_1(s-w,1+w;m_1m_2, \alpha d),
\end{split}
\end{equation}
where, as before, $\ell_1$ is defined by \eqref{eq: defn of ell 1}, $d_1$ is defined by \eqref{eq: defn of d 1}, and $\mathcal{H}_1$ is defined as the product of \eqref{eq: defn of H 1} over all primes.

It is convenient for later calculations to write $\mathcal{P}_1$ in terms of a residue, as in \eqref{P14}, rather than in terms of logarithmic derivatives as in \eqref{I2}.

\subsection{Bounding the contribution of $\mathcal{R}_1$}\label{subsec: error term}

Having handled $\mathcal{P}_1$ in \eqref{Q14}, we next turn to $\mathcal{R}_1$, defined by \eqref{eq: defn of R1}. It will be convenient to denote
\begin{equation}\label{eq: defn of R}
\begin{split}
& \mathcal{R}(\ell,d) = \frac{1}{\ell}\left(\frac{1+i}{2}\right)\frac{\sqrt{2}}{2} \sum_{\substack{ \alpha\leq Y \\ (\alpha,2\ell)=1}} \frac{\mu(\alpha)}{\alpha^2d_1}\left( \frac{d_1}{\ell}\right)  \sum_{\substack{k \in\mathbb{Z} \\ k\text{ odd}}}\left(\frac{2}{k d_1}\right)  \\
& \times   \frac{1}{2\pi i} \int_{(-\frac{1}{2}+\varepsilon)} h\left( \frac{kX}{8\alpha^2 d_1 \ell},w\right) L(1+w,\chi_{k_1})^2\mathcal{G}_{0}(1+w;k,\ell,\alpha, d) \,dw,
\end{split}
\end{equation}
so that $\mathcal{R}_1=m_1m_2\mathcal{R}(m_1m_2,d)$. We will bound $|\mathcal{R}(\ell,d)|$ on average as $\ell$ and $d$ each range over a dyadic interval.

Let $\beta_{\ell,d} =\overline{\mathcal{R}(\ell,d)}/|\mathcal{R}(\ell,d)|$ if $\mathcal{R}(\ell,d)\neq 0$, and $\beta_{\ell,d}=1$ otherwise. Then $|\beta_{\ell,d}|=1$ and $|\mathcal{R}(\ell,d)|=\beta_{\ell,d} \mathcal{R}(\ell,d)$. We sum this over all $\ell,d $ with $J\leq\ell< 2J$ and $V\leq d<2V$, where $J,V\geq 1$. We then insert the definition \eqref{eq: defn of R} and bring the $d,\ell$-sum inside the integral to deduce that
\begin{equation}\label{eq: bound for R over dyadic intervals}
\sum_{\substack{d=V \\ (d,2)=1}}^{2V-1}\sum_{\substack{\ell=J \\ (\ell,2d)=1}}^{2J-1} |\mathcal{R}(\ell,d)| = \sum_{\substack{d=V \\ (d,2)=1}}^{2V-1} \sum_{\substack{\ell=J \\ (\ell,2d)=1}}^{2J-1} \beta_{\ell,d}\mathcal{R}(\ell,d) \ll \sum_{\substack{ \alpha\leq Y \\ (\alpha,2)=1}} \frac{1}{\alpha^2} \sum_{\substack{k \in\mathbb{Z} \\ k\text{ odd}}} \int_{(-\frac{1}{2}+\varepsilon)} U(\alpha,k,w)\,|dw|,
\end{equation}
where for brevity we denote
\begin{equation*}
U(\alpha,k,w) = \sum_{\substack{d=V \\ (d,2)=1}}^{2V-1} \frac{1}{d_1}|L(1+w,\chi_{k_1})|^2\Bigg|  \sum_{\substack{\ell=J \\ (\ell,2\alpha d)=1}}^{2J-1}  \frac{\beta_{\ell,d}}{\ell}\left( \frac{d_1}{\ell}\right) \mathcal{G}_0(1+w;k,\ell,\alpha,d) h\left( \frac{kX}{8\alpha^2d_1\ell },w\right)\Bigg|.
\end{equation*}
We split the $k$-sum into dyadic blocks $K\leq |k|<2K$, with $K\geq 1$, and apply Cauchy's inequality to write
\begin{equation}\label{eq: applying Cauchy to sum of U}
\begin{split}
& \sum_{\substack{ K\leq |k|< 2K \\ k\text{ odd}}}U(\alpha,k,w) \ll \Bigg(\sum_{\substack{d=V \\ (d,2)=1}}^{2V-1} \frac{1}{d_1} \sum_{\substack{ K\leq |k|<2K \\ k\text{ odd} }} k_2  | L(1+w,\chi_{k_1})|^4  \Bigg)^{\frac{1}{2}} \\
& \times \Bigg(\sum_{\substack{d=V \\ (d,2)=1}}^{2V-1} \frac{1}{d_1} \sum_{\substack{ K\leq |k|< 2K \\ k\text{ odd} }}\frac{1}{k_2} \Bigg| \sum_{\substack{\ell=J \\ (\ell,2\alpha d)=1}}^{2J-1} \frac{\beta_{\ell,d}}{\ell}\left( \frac{d_1}{\ell}\right) \mathcal{G}_0(1+w;k,\ell,\alpha,d) h\left( \frac{kX}{8\alpha^2d_1\ell },w\right) \Bigg|^2 \Bigg)^{\frac{1}{2}},
\end{split}
\end{equation}
where $k_2$ is defined by \eqref{eq: defn of k1 and k2}. To bound the first factor on the right-hand side of \eqref{eq: applying Cauchy to sum of U}, we split the $k$-sum according to the values of $k_1$ and $k_2$ and interchange the order of summation. Then we use the fact that $d_1\geq d/\alpha$ by \eqref{eq: defn of d 1} to deduce that
\begin{equation*}
\begin{split}
& \sum_{\substack{d=V \\ (d,2)=1}}^{2V-1} \frac{1}{d_1} \sum_{\substack{ K\leq |k|<2K \\ k\text{ odd} }} k_2  | L(1+w,\chi_{k_1})|^4 \leq  \frac{\alpha}{V} \sum_{\substack{0<|k_1|\ll KV }} | L(1+w,\chi_{k_1})|^4 \sum_{k_2\ll \sqrt{\frac{KV}{k_1}} }k_2 \sum_{\substack{d=V \\ (d,2)=1 \\ d_1|k_1k_2^2}}^{2V-1} 1 .
\end{split}
\end{equation*}
We estimate the inner sum using the divisor bound, and find that the above is
\begin{equation*}
\ll \alpha K^{1+\varepsilon} V^{\varepsilon} \sum_{\substack{0<|k_1|\ll KV }} \frac{1}{k_1}| L(1+w,\chi_{k_1})|^4  \ll \alpha K^{1+\varepsilon} V^{\varepsilon} (1+|w|)^{1+\varepsilon}
\end{equation*}
by Lemma~\ref{lem: moment estimates}. It follows from this and \eqref{eq: applying Cauchy to sum of U} that
\begin{equation}\label{eq:using moment estimates of Heath Brown}
\begin{split}
& \sum_{\substack{ K\leq |k|<2K \\ k\text{ odd}}}U(\alpha,k,w) \ll \Bigg(\alpha K^{1+\varepsilon} V^{\varepsilon} (1+|w|)^{1+\varepsilon}  \Bigg)^{\frac{1}{2}} \\
& \times \Bigg(\sum_{\substack{d=V \\ (d,2)=1}}^{2V-1} \frac{1}{d_1} \sum_{\substack{ K\leq |k|<2K \\ k\text{ odd} }}\frac{1}{k_2} \Bigg| \sum_{\substack{\ell=J \\ (\ell,2\alpha d)=1}}^{2J-1} \frac{\beta_{\ell,d}}{\ell}\left( \frac{d_1}{\ell}\right) \mathcal{G}_0(1+w;k,\ell,\alpha,d) h\left( \frac{kX}{8\alpha^2d_1\ell },w\right) \Bigg|^2 \Bigg)^{\frac{1}{2}}.
\end{split}
\end{equation}
The next task is to bound the second factor on the right-hand side. To this end we prove the following two lemmas.
\begin{lem}\label{lem: version of Lemma 5.4 of Sound}
Let $\alpha\leq Y$, $d$, $K$, and $J$ be positive integers, and suppose $w$ is a complex number with real part $-\frac{1}{2}+\varepsilon$. Then for any choice of complex numbers $\gamma_{\ell}$ with $|\gamma_{\ell}|\leq 1$,
$$
\sum_{\substack{ K\leq |k|<2K \\ k\textup{ odd} }}\frac{1}{k_2} \Bigg| \sum_{\substack{\ell=J \\ (\ell,2\alpha d)=1}}^{2J-1} \frac{\gamma_{\ell}}{\ell}  \mathcal{G}_0(1+w;k,\ell,\alpha,d) h\left( \frac{kX}{8\alpha^2d_1\ell },w\right) \Bigg|^2
$$
is bounded by
\begin{align*}
\ll_{\varepsilon} |\check{\Phi}(w)|^2   \frac{d_1\alpha^{2+\varepsilon}J^{2+\varepsilon}K^{\varepsilon}d^{\varepsilon} }{X^{1-\varepsilon}} \exp\Bigg( -\frac{1}{20}\frac{\sqrt{K}}{\alpha\sqrt{d_1 J(1+|w|)}}\Bigg).
\end{align*}
and also by
\begin{align*}
\ll_{\varepsilon} ((1+|w|)\alpha d JKX)^{\varepsilon}|\check{\Phi}(w)|^2  \frac{\alpha^2 d_1(JK+J^2)}{KX}.
\end{align*}
\end{lem}

\begin{lem}\label{lem: version of Lemma 5.5 of Sound}
Let $\delta_{\ell}\ll \ell^{\varepsilon}$ be any sequence of complex numbers and let \upshape{Re}$(w)=-\frac{1}{2}+\varepsilon$. Then
\begin{equation*}
\sum_{K\leq |k|< 2K}\frac{1}{k_2} \left| \sum_{\substack{\ell=J \\ (\ell,2\alpha d)=1}}^{2J-1} \frac{\delta_{\ell}}{\sqrt{\ell}} \mathcal{G}_0(1+w;k,\ell,\alpha, d)\right|^2 \ll_{\varepsilon} (\alpha d JK)^{\varepsilon} J(J+K).
\end{equation*}
\end{lem}

\begin{proof}[Proof of Lemma \ref{lem: version of Lemma 5.4 of Sound} assuming Lemma \ref{lem: version of Lemma 5.5 of Sound}]
To prove the first bound, we use the triangle inequality and apply the bounds for $\mathcal{G}_0$ from Lemma~\ref{lem: nu-sum as an Euler product} and $h(\xi,w)$ from Lemma~\ref{lem: properties of h(xi,w)} to deduce that the sum in question is
$$
\ll |\check{\Phi}(w)|^2\frac{d_1\alpha^{2+\varepsilon}J^{\varepsilon}K^{\varepsilon}d^{\varepsilon} }{X^{1-\varepsilon}} \exp\Bigg( -\frac{1}{20}\frac{\sqrt{K}}{\alpha\sqrt{d_1 J(1+|w|)}}\Bigg)\sum_{\substack{ K\leq |k|<2K \\ k\text{ odd} }}\frac{1}{|k|k_2} \Bigg(\sum_{\substack{\ell=J \\ (\ell,2\alpha d)=1}}^{2J-1} (\ell,k_4^2)^{\frac{1}{2}} \Bigg)^2.
$$
We then estimate the $k$-sum by splitting it according to the values of $k_1$ and $k_2$ and using $(\ell,k_4^2)\leq k_4^2\leq k_2^2$, which follows from \eqref{eq: defn of k1 and k2} and \eqref{eq: defn of k3 and k4}. This leads to the first bound of the lemma.

To prove the second bound, we apply Lemma~\ref{lem: properties of h(xi,w)} and write the integral \eqref{eq: integral in lemma for h(xi,w)} as 
\begin{align*}
\frac{1}{2\pi i} \int_{(c)} g(s,w; \text{sgn}(\xi))\left(\frac{X}{\pi|\xi|}\right)^s\,ds
\end{align*}
with $c=\varepsilon$. We then bring the $\ell$-sum inside the integral and use the triangle inequality to deduce that
\begin{align*}
& \Bigg| \sum_{\substack{\ell=J \\ (\ell,2\alpha d)=1}}^{2J-1} \frac{\gamma_{\ell}}{\ell}  \mathcal{G}_0(1+w;k,\ell,\alpha,d) h\left( \frac{kX}{8\alpha^2d_1\ell },w\right) \Bigg| \\
& \ll |\check{\Phi}(w)|\left(\frac{\alpha^{1+\varepsilon} d_1^{\frac{1}{2}+\varepsilon} }{|k|^{\frac{1}{2}-\varepsilon}X^{\frac{1}{2}-\varepsilon}}\right)\int_{(\varepsilon)}\Bigg| g(s,w;\text{sgn}(k))\sum_{\substack{\ell=J \\ (\ell,2\alpha d)=1}}^{2J-1}  \frac{\gamma_{\ell}}{\ell^{1+w-s}} \mathcal{G}_0(1+w;k,\ell,\alpha,d)   \Bigg|\,|ds|.
\end{align*}
Thus, since $g(s,w;\text{sgn}(k))\ll_{\varepsilon}(1+|w|)^{\varepsilon}\exp(-(\frac{\pi}{2}-\varepsilon) |\text{Im}(s)|)$ by Stirling's formula, it follows from Cauchy's inequality that
\begin{equation*}
\begin{split}
\Bigg| \sum_{\substack{\ell=J \\ (\ell,2\alpha d)=1}}^{2J-1} \frac{\gamma_{\ell}}{\ell}  \mathcal{G}_0(1+w;k,\ell,\alpha,d) h\left( \frac{kX}{8\alpha^2d_1\ell },w\right) \Bigg|^2 \ll  (1+|w|)^{\varepsilon}|\check{\Phi}(w)|^2 \left(\frac{\alpha^{2+\varepsilon} d_1^{1+\varepsilon} }{|k|^{1-\varepsilon}X^{1-\varepsilon}}\right)  \\
\times \int_{(\varepsilon)}\exp(-(\tfrac{\pi}{2}-\varepsilon)|\text{Im}(s)|)\Bigg| \sum_{\substack{\ell=J \\ (\ell,2\alpha d)=1}}^{2J-1}  \frac{\gamma_{\ell}}{\ell^{1+w-s}} \mathcal{G}_0(1+w;k,\ell,\alpha,d)   \Bigg|^2\,|ds|.
\end{split}
\end{equation*}
The second bound of the lemma follows from this and Lemma~\ref{lem: version of Lemma 5.5 of Sound}.
\end{proof}

\begin{proof}[Proof of Lemma \ref{lem: version of Lemma 5.5 of Sound}]
For any integer $k=\pm \prod_{i, \ a_i\geq 1} p_i^{a_i}$, let $a(k)$ and $b(k)$ be defined by
\begin{equation}\label{eq: defn of a(k) and b(k)}
a(k)=\prod_i p_i^{a_i+1} \ \ \ \ \text{and} \ \ \ \ b(k)=\prod_{a_i=1} p_i \prod_{a_i\geq 2}p_i^{a_i-1}.
\end{equation}
From the definition of $\mathcal{G}_0$ in Lemma~\ref{lem: nu-sum as an Euler product}, we see that $\mathcal{G}_0(1+w;k,\ell,\alpha, d)=0$ unless $\ell$ can be written as $gm$ with $g|a(k)$ and $m$ square-free and relatively prime to $k$. With this expression for $\ell$, it follows from Lemma~\ref{lem: properties of Gkn} that if $(\ell,2\alpha d)=1$ then
\begin{equation}\label{eq: G0 in terms of g and m}
\mathcal{G}_0(1+w;k,\ell,\alpha, d)=\sqrt{m}\left( \frac{k}{m}\right) \prod_{p|m} \left( 1+ \frac{2}{p^{1+w}}\left( \frac{k_1}{p}\right)\right)^{-1} \mathcal{G}_{0}(1+w;k,g,\alpha, d).
\end{equation}
From this and Cauchy's inequality, we arrive at
\begin{equation}\label{eq: split G0 sum into 3|m and 3 not divides m}
\sum_{K\leq |k|< 2K}\frac{1}{k_2} \left| \sum_{\substack{\ell=J\\ (\ell,2\alpha d)=1}}^{2J-1} \frac{\delta_{\ell}}{\sqrt{\ell}} \mathcal{G}_0(1+w;k,\ell,\alpha, d)\right|^2 \ll_{\varepsilon} K^{\varepsilon} \sum_{K\leq |k|< 2K}\frac{1}{k_2}  \sum_{\substack{g|a(k) \\ g<2J}} (\Psi_1(k,g)+\Psi_2(k,g)),
\end{equation}
where
\begin{equation*}
\Psi_1(k,g) = \Bigg| \sum_{\substack{\frac{J}{g}\leq m<\frac{2J}{g} \\ (m,2\alpha d)=1\\ 3\nmid m }} \frac{\mu^2(m) \delta_{gm}}{\sqrt{g}}  \mathcal{G}_{0}(1+w;k,g,\alpha, d) \left( \frac{k}{m}\right) \prod_{p|m} \left( 1+ \frac{2}{p^{1+w}}\left( \frac{k_1}{p}\right)\right)^{-1} \Bigg|^2
\end{equation*}
and $\Psi_2(k,g)$ is the same, but with the condition $3|m$ instead of $3\nmid m$. We first bound the contribution of $\Psi_1$. We factor out $g^{-1/2}\mathcal{G}_0(1+w;k,g,\alpha, d)$ and apply the bound from Lemma~\ref{lem: nu-sum as an Euler product} to deduce that
\begin{equation}\label{eq: Psi 1 after factoring out G0}
\Psi_1(k,g) \ll_{\varepsilon} (\alpha d K)^{\varepsilon}g^{1+\varepsilon}\Bigg| \sum_{\substack{\frac{J}{g}\leq m<\frac{2J}{g} \\ (m,6\alpha d)=1 }} \mu^2(m) \delta_{gm}  \left( \frac{k}{m}\right) \prod_{p|m} \left( 1+ \frac{2}{p^{1+w}}\left( \frac{k_1}{p}\right)\right)^{-1} \Bigg|^2.
\end{equation}
If $\left( \frac{k}{m}\right)\neq 0$, then
\begin{align*}
\prod_{p|m} \left( 1+ \frac{2}{p^{1+w}}\left( \frac{k_1}{p}\right)\right)^{-1} 
& = \prod_{p|m} \left( 1- \frac{4}{p^{2+2w}}\right)^{-1} \prod_{p|m} \left( 1- \frac{2}{p^{1+w}}\left( \frac{k_1}{p}\right)\right) \\
& = \prod_{p|m} \left( 1- \frac{4}{p^{2+2w}}\right)^{-1}\sum_{j|m} \frac{\mu(j) d_2(j)}{j^{1+w}} \left( \frac{k_1}{j}\right).
\end{align*}
We insert this into \eqref{eq: Psi 1 after factoring out G0}, interchange the order of summation, and apply Cauchy's inequality to see that
\begin{equation*}
\Psi_1(k,g) \ll_{\varepsilon} (\alpha d K)^{\varepsilon}g^{1+\varepsilon} \sum_{j<\frac{2J}{g} } \Bigg| \sum_{\substack{\frac{J}{g}\leq m<\frac{2J}{g} \\ (m,6\alpha d)=1 \\ j|m }} \mu^2(m) \delta_{gm} \left( \frac{k}{m}\right) \prod_{\substack{p|m }} \left( 1- \frac{4}{p^{2+2w}} \right)^{-1}\Bigg|^2.
\end{equation*}
We next relabel $m$ as $jm$, factor out $\mu^2(j)\left( \frac{k}{j}\right) \prod_{\substack{p|j }} \left( 1- \frac{4}{p^{2+2w}} \right)^{-1}$ from the $m$-sum, and observe that $\prod_{\substack{p|j }} \left( 1- \frac{4}{p^{2+2w}} \right)^{-1}\ll_{\varepsilon} j^{\varepsilon}$ because Re$(w)\geq -\frac{1}{2}+\varepsilon$ and $p>3$ for all $p|m$. The result is
\begin{equation}\label{eq: after relabeling m as jm in Psi 1}
\Psi_1(k,g) \ll_{\varepsilon} (\alpha d JK)^{\varepsilon}g^{1+\varepsilon} \sum_{j<\frac{2J}{g} } \Bigg| \sum_{\substack{\frac{J}{gj}\leq m<\frac{2J}{gj} \\ (m,6\alpha dj)=1}} \mu^2(m) \delta_{gjm} \left( \frac{k}{m}\right) \prod_{\substack{p|m }} \left( 1- \frac{4}{p^{2+2w}} \right)^{-1}\Bigg|^2.
\end{equation}
Now, by \eqref{eq: defn of a(k) and b(k)}, $g|a(k)$ implies $b(g)|k$. Thus we may interchange the order of summation to write
$$
\sum_{K\leq |k|< 2K}\frac{1}{k_2}  \sum_{\substack{g|a(k) \\ g<2J}} \Psi_1(k,g) \leq \sum_{g<2J} \sum_{\substack{K\leq |k|< 2K \\ b(g)|k} }\frac{1}{k_2}  \Psi_1(k,g) = \sum_{g<2J}\sum_{\frac{K}{b(g)}\leq |f|< \frac{2K}{b(g)}} \frac{1}{k_2}  \Psi_1(f b(g),g),
$$
where we have relabeled $k$ in the last sum as $f b(g)$, so that, by \eqref{eq: defn of k1 and k2}, $k_2>0$ satisfies $4 f b(g) d_1=k_1k_2^2$, with $k_1$ a fundamental discriminant. From this and \eqref{eq: after relabeling m as jm in Psi 1}, we arrive at
\begin{equation}\label{eq: after relabeling k as fb(g)}
\begin{split}
\sum_{K\leq |k|< 2K}\frac{1}{k_2}  \sum_{\substack{g|a(k) \\ g<2J}} \Psi_1(k,g) \ll_{\varepsilon} (\alpha d J K)^{\varepsilon}\sum_{g<2J}g \sum_{\frac{K}{b(g)}\leq |f|< \frac{2K}{b(g)}} \frac{1}{k_2} \\
\times \sum_{j<\frac{2J}{g} } \Bigg| \sum_{\substack{\frac{J}{gj}\leq m<\frac{2J}{gj} \\ (m,6\alpha dj)=1}} \mu^2(m) \delta_{gjm} \left( \frac{f b(g)}{m}\right) \prod_{\substack{p|m }} \left( 1- \frac{4}{p^{2+2w}} \right)^{-1}\Bigg|^2 .
\end{split}
\end{equation}
If $4f=f_1f_2^2$, with $f_1$ a fundamental discriminant and $f_2$ a positive integer, then the equation $4 f b(g) d_1=k_1k_2^2$ implies that $f_2|2k_2$, and thus $k_2^{-1}\ll f_2^{-1}$. Hence it follows from \eqref{eq: after relabeling k as fb(g)} and Lemma~\ref{lem: estimates for character sums} that
$$
\sum_{K\leq |k|< 2K}\frac{1}{k_2}  \sum_{\substack{g|a(k) \\ g<2J}} \Psi_1(k,g) \ll_{\varepsilon} (\alpha d J K)^{\varepsilon} J(J+K).
$$
This proves the desired bound for the sum of $\Psi_1(k,g)$ in \eqref{eq: split G0 sum into 3|m and 3 not divides m}. To bound the sum of $\Psi_2(k,g)$, we argue in the same way, but instead of \eqref{eq: G0 in terms of g and m} we use
\begin{equation*}
\mathcal{G}_0(1+w;k,\ell,\alpha, d)=\sqrt{m}\left( \frac{k}{m}\right) \prod_{\substack{p|m\\ p>3}} \left( 1+ \frac{2}{p^{1+w}}\left( \frac{k_1}{p}\right)\right)^{-1} \mathcal{G}_{0}^*(1+w;k,g,\alpha, d),
\end{equation*}
where
$$
\mathcal{G}^*_{0}(1+w;k,g,\alpha, d) \ = \ \left( 1-\frac{1}{3^{1+w}} \left( \frac{k_1}{3}\right)\right)^2\prod_{p\neq 3} \mathcal{G}_{0,p}(1+w;k,g,\alpha, d),
$$
with $\mathcal{G}_{0,p}$ as defined in Lemma~\ref{lem: nu-sum as an Euler product}.
\end{proof}
We now estimate the contribution of $\mathcal{R}_1$. From the first bound of Lemma~\ref{lem: version of Lemma 5.4 of Sound}, we see that the sum of the right-hand side of \eqref{eq:using moment estimates of Heath Brown} over all $K=2^j>\alpha^2 V J(1+|w|)(\log X)^4$ is negligible. On the other hand, if $K\leq \alpha^2 V J(1+|w|)(\log X)^4$ then it follows from \eqref{eq:using moment estimates of Heath Brown} and the second bound in Lemma~\ref{lem: version of Lemma 5.4 of Sound} that
\begin{equation*}
\begin{split}
\sum_{\substack{ K\leq |k|<2K \\ k\text{ odd}}}U(\alpha,k,w)
& \ll_{\varepsilon} (1+|w|)^{\frac{1}{2}+\varepsilon}|\check{\Phi}(w)|(\alpha  JKVX)^{\varepsilon} \Bigg(  \frac{ \alpha^3 V(JK+J^2)}{X} \Bigg)^{\frac{1}{2}} \\
& \ll_{\varepsilon} (1+|w|)^{1+\varepsilon}|\check{\Phi}(w)|(\alpha  JKVX)^{\varepsilon} \frac{\alpha^{\frac{5}{2}} VJ}{X^{\frac{1}{2}}}.
\end{split}
\end{equation*}
We sum this over all $K=2^j$, $j$ a positive integer, with $K\leq \alpha^2 V J(1+|w|)(\log X)^4$, and then multiply the resulting sum by $\alpha^{-2}$. We then integrate over all $w$ with Re$(w)=-\frac{1}{2}+\varepsilon$ and sum over all integers $\alpha\leq Y$ to deduce from \eqref{eq: bound for R over dyadic intervals} that 
\begin{equation}\label{eq: 2nd bound for R over dyadic intervals}
\sum_{\substack{d=V \\ (d,2)=1}}^{2V-1}\sum_{\substack{\ell=J \\ (\ell,2d)=1}}^{2J-1} |\mathcal{R}(\ell,d)| \ll \frac{ V^{1+\varepsilon}J^{1+\varepsilon} Y^{\frac{3}{2}+\varepsilon}}{X^{\frac{1}{2}-\varepsilon}}.
\end{equation}
Recall from \eqref{eq: defn of R1} and \eqref{eq: defn of R} that $\mathcal{R}_1=m_1m_2\mathcal{R}(m_1m_2,d)$. Since $\lambda_d\ll d^{\varepsilon}$ by \eqref{lambda} and $b_m\ll 1$ by \eqref{eq: defn of mollifier coeffs bm}, it thus follows from \eqref{eq: 2nd bound for R over dyadic intervals} that
\begin{equation}\label{eq: bound for contribution of R1}
\sum_{\substack{ d\leq D \\ d \text{ odd} }} \mu^2(d)\lambda_d \mathop{\sum\sum}_{\substack{m_1,m_2\leq M \\ (m_1m_2,2d)=1}} \frac{b_{m_1}b_{m_2}}{(m_1m_2)^{3/2}} \,|\mathcal{R}_1 | \ll \frac{ D^{1+\varepsilon}M^{1+\varepsilon} Y^{\frac{3}{2}+\varepsilon}}{X^{\frac{1}{2}-\varepsilon}}.
\end{equation}

\subsection{Conditions for the parameters}\label{subsection: conditions for the parameters}

From \eqref{eq: defn of Q1 star}, \eqref{Q14}, \eqref{P14}, and \eqref{eq: bound for contribution of R1}, we see that the total contribution of the sum with $\mathcal{Q}_1$ to $\mathcal{B}$ in \eqref{R0split} is
\begin{equation}\label{Q15}
\begin{split}
& \frac{X}{(\sqrt{2}-1)^4}\sum_{\substack{ d\leq D \\ d \text{ odd} }} \mu^2(d)\lambda_d \mathop{\sum\sum}_{\substack{m_1,m_2\leq M \\ (m_1m_2,2d)=1}} \frac{b_{m_1}b_{m_2}}{(m_1m_2)^{3/2}} \sum_{\substack{ \nu=1 \\ (\nu,2d)=1}}^{\infty} \frac{d_2(\nu) }{\nu^{3/2}} \ \sum_{\substack{ \alpha\leq Y \\ (\alpha,2m_1m_2\nu)=1}} \frac{\mu(\alpha)}{\alpha^2d_1}\, \mathcal{Q}_1 \\
& = \left(\frac{1+i}{2}\right)\frac{\sqrt{2}X}{2(\sqrt{2}-1)^4}\sum_{\substack{ d\leq D \\ d \text{ odd} }} \mu^2(d)\lambda_d \mathop{\sum\sum}_{\substack{m_1,m_2\leq M \\ (m_1m_2,2d)=1}} \frac{b_{m_1}\overline{b_{m_2}}}{(m_1m_2)^{3/2}} \,  \underset{w=0}{\mbox{Res}} \   X^w \frac{1}{2\pi i} \int_{\left( \frac{1}{\log X}\right)}    \Gamma_2(s-w) \\
& \ \times  8^{s-w}\left(1-\frac{1}{4^{s-w}}\right) \mathcal{K}(s,w;m_1m_2, d)\,\frac{ds}{s} + O\Bigg( \frac{X^{1+\varepsilon}D^{\varepsilon}M^{\varepsilon}  }{Y^{1-\varepsilon}} + X^{\frac{1}{2}+\varepsilon}D^{1+\varepsilon}M^{1+\varepsilon} Y^{\frac{3}{2}+\varepsilon} \Big).
\end{split}
\end{equation}
Recall the definition \eqref{eq:outline section, defn of M, length of mollifier} of $M$. Also, recall the definitions \eqref{Ddef} and \eqref{eq:sieve section, defn of R} of $D$ and $R$ of $D$, respectively. So that the error terms in \eqref{Q15} are $O(X^{1-\varepsilon})$, we assume the parameters $\theta$ and $\vartheta$ satisfy
\begin{equation*}
\theta+2\vartheta<\frac{1}{2},
\end{equation*}
and we take the parameter $Y$ in \eqref{mu2approx} to be
\begin{equation*}
Y=X^{\delta}
\end{equation*}
with $\delta = \delta(\theta,\vartheta)$ sufficiently small.

\subsection{Evaluating the sums of the other terms with $k\neq 0$}

The procedure for evaluating the sum with $\mathcal{Q}_2$ in \eqref{R0split} is largely similar to the above process for $\mathcal{Q}_1$, with only a few differences. The main difference arises from the negative sign in the character $\left( \frac{-2d_1}{m_1m_2\nu}\right)$ in \eqref{Q2}. This causes the residues in the versions of \eqref{Q14} and \eqref{eq: defn of P1} for $\mathcal{Q}_2$ to have each $-kd_1$ equal to a perfect square instead of $kd_1=\square$. This means sgn$(k)=-1$. Hence, because of the factor sgn$(\xi)$ in \eqref{eq: integral in lemma for h(xi,w)}, the version of \eqref{P12} for $\mathcal{Q}_2$ has the function
\begin{equation*}
(2\pi)^{-u} \Gamma(u) ( \cos \left(\tfrac{\pi}{2}u\right)+i \sin \left(\tfrac{\pi}{2}u\right) )
\end{equation*}
in place of the function $\Gamma_2(u)$ defined by \eqref{eq: defn of Gamma 2}. These lead to a version of \eqref{Q15} for $\mathcal{Q}_2$ that we may combine with \eqref{Q15} using the identity
\begin{equation}\label{eq: identity with cos and sine}
\left( \frac{1+i}{2}\right) (\cos u - i\sin u) + \left( \frac{1-i}{2}\right) (\cos u + i\sin u) =\cos u + \sin u.
\end{equation}
The result is
\begin{equation}\label{eq: combining Q1 and Q2}
\begin{split}
& \frac{X}{(\sqrt{2}-1)^4}\sum_{\substack{ d\leq D \\ d \text{ odd} }} \mu^2(d)\lambda_d \mathop{\sum\sum}_{\substack{m_1,m_2\leq M \\ (m_1m_2,2d)=1}} \frac{b_{m_1}b_{m_2}}{(m_1m_2)^{3/2}} \sum_{\substack{ \nu=1 \\ (\nu,2d)=1}}^{\infty} \frac{d_2(\nu) }{\nu^{3/2}} \ \sum_{\substack{ \alpha\leq Y \\ (\alpha,2m_1m_2\nu)=1}} \frac{\mu(\alpha)}{\alpha^2d_1}\, (\mathcal{Q}_1+\mathcal{Q}_2) \\
& = \frac{\sqrt{2}X}{2(\sqrt{2}-1)^4}\sum_{\substack{ d\leq D \\ d \text{ odd} }} \mu^2(d)\lambda_d \mathop{\sum\sum}_{\substack{m_1,m_2\leq M \\ (m_1m_2,2d)=1}} \frac{b_{m_1}b_{m_2}}{(m_1m_2)^{3/2}} \, \underset{w=0}{\mbox{Res}} \   X^w \frac{1}{2\pi i} \int_{\left( \frac{1}{\log X}\right)}    \Gamma_1(s-w) \\
& \ \times  8^{s-w}\left(1-\frac{1}{4^{s-w}}\right) \mathcal{K}(s,w;m_1m_2, d)\,\frac{ds}{s} + O(X^{1-\varepsilon}),
\end{split}
\end{equation}
where
\begin{equation}\label{eq: defn of Gamma 1}
\Gamma_1(u)=(2\pi)^{-u} \Gamma(u) ( \cos \left(\tfrac{\pi}{2}u\right)+ \sin \left(\tfrac{\pi}{2}u\right) )
\end{equation}
and the bound $O(X^{1-\varepsilon})$ for the error term is guaranteed by the conditions in Subsection \ref{subsection: conditions for the parameters}.

The evaluation of the sums in \eqref{R0split} with $\mathcal{Q}_3$ and $\mathcal{Q}_4$ defined by \eqref{Q3} and \eqref{Q4} is similar. The version of \eqref{eq: defn of P1} for $\mathcal{Q}_3$ has an extra $-1$ factor because the Kronecker symbol $\left(\frac{-2}{kd_1}\right)$ equals $-1$ when $-kd_1$ is an odd perfect square. The resulting expression for the sums in \eqref{R0split} with $\mathcal{Q}_3$ and $\mathcal{Q}_4$ is exactly the same as the right-hand side of \eqref{eq: combining Q1 and Q2}. Therefore
\begin{equation}\label{eq: combining Q1 Q2 Q3 Q4}
\begin{split}
& \frac{X}{(\sqrt{2}-1)^4}\sum_{\substack{ d\leq D \\ d \text{ odd} }} \mu^2(d)\lambda_d \mathop{\sum\sum}_{\substack{m_1,m_2\leq M \\ (m_1m_2,2d)=1}} \frac{b_{m_1}b_{m_2}}{(m_1m_2)^{3/2}} \sum_{\substack{ \nu=1 \\ (\nu,2d)=1}}^{\infty} \frac{d_2(\nu) }{\nu^{3/2}} \ \sum_{\substack{ \alpha\leq Y \\ (\alpha,2m_1m_2\nu)=1}} \frac{\mu(\alpha)}{\alpha^2d_1} \sum_{j=1}^4\mathcal{Q}_j\\
& = \frac{\sqrt{2}X}{(\sqrt{2}-1)^4}\sum_{\substack{ d\leq D \\ d \text{ odd} }} \mu^2(d)\lambda_d \mathop{\sum\sum}_{\substack{m_1,m_2\leq M \\ (m_1m_2,2d)=1}} \frac{b_{m_1}b_{m_2}}{(m_1m_2)^{3/2}} \, \underset{w=0}{\mbox{Res}} \   X^w \frac{1}{2\pi i} \int_{\left( \frac{1}{\log X}\right)}    \Gamma_1(s-w) \\
& \ \times  8^{s-w}\left(1-\frac{1}{4^{s-w}}\right) \mathcal{K}(s,w;m_1m_2, d)\,\frac{ds}{s} + O(X^{1-\varepsilon}).
\end{split}
\end{equation}

To estimate the sum with $\mathcal{U}_1$ in \eqref{R0split}, we first relabel $k$ in \eqref{U1} as $2k$, now with $k$ odd, to write
\begin{equation}\label{U1 relabeled}
\mathcal{U}_1 = \left(\frac{1+i}{2}\right)\left( \frac{2d_1}{m_1m_2\nu}\right)    \sum_{\substack{k \in\mathbb{Z} \\ k \text{ odd}}}e\left(\frac{k\overline{\alpha^2d_1m_1m_2\nu}}{4}\right)\hat{F}_{\nu}\left( \frac{kX}{4\alpha^2d_1m_1m_2\nu}\right) G_{2k}(m_1m_2\nu).
\end{equation}
From the definition \eqref{eq: defn of Gk} of $G_k(n)$, we see that $G_{2k}(n)=\left( \frac{2}{n}\right) G_k(n)$ for all odd integers $n$. Also, the orthogonality of Dirichlet characters modulo $4$ implies that $e(\frac{h}{4})=i(\frac{-1}{h})$ for odd $h$. It follows from these and \eqref{U1 relabeled} that
\begin{equation*}
\mathcal{U}_1 = i\left(\frac{1+i}{2}\right)\left( \frac{-d_1}{m_1m_2\nu}\right)    \sum_{\substack{k \in\mathbb{Z} \\ k \text{ odd}}}\left(\frac{-1}{kd_1 }\right)\hat{F}_{\nu}\left( \frac{kX}{4\alpha^2d_1m_1m_2\nu}\right) G_{k}(m_1m_2\nu).
\end{equation*}
We then proceed as we did for $\mathcal{Q}_1$. We treat the sum with $\mathcal{U}_2$, defined by \eqref{U2}, in a similar way. We combine the resulting expressions using the identity \eqref{eq: identity with cos and sine}, and we arrive at
\begin{equation}\label{U1U2}
\begin{split}
& \frac{X}{(\sqrt{2}-1)^4}\sum_{\substack{ d\leq D \\ d \text{ odd} }} \mu^2(d)\lambda_d \mathop{\sum\sum}_{\substack{m_1,m_2\leq M \\ (m_1m_2,2d)=1}} \frac{b_{m_1}b_{m_2}}{(m_1m_2)^{3/2}} \sum_{\substack{ \nu=1 \\ (\nu,2d)=1}}^{\infty} \frac{d_2(\nu) }{\nu^{3/2}} \ \sum_{\substack{ \alpha\leq Y \\ (\alpha,2m_1m_2\nu)=1}} \frac{\mu(\alpha)}{\alpha^2d_1}\, (\mathcal{U}_1+\mathcal{U}_2) \\
& = \frac{X}{(\sqrt{2}-1)^4}\sum_{\substack{ d\leq D \\ d \text{ odd} }} \mu^2(d)\lambda_d \mathop{\sum\sum}_{\substack{m_1,m_2\leq M \\ (m_1m_2,2d)=1}} \frac{b_{m_1}b_{m_2}}{(m_1m_2)^{3/2}} \,  \underset{w=0}{\mbox{Res}} \   X^w \frac{1}{2\pi i} \int_{\left( \frac{1}{\log X}\right)}    \Gamma_1(s-w) \\
& \ \times  4^{s-w}\left(1-\frac{1}{4^{s-w}}\right) \mathcal{K}(s,w;m_1m_2, d)\,\frac{ds}{s} + O(X^{1-\varepsilon}).
\end{split}
\end{equation}

Next, to evaluate the sum with $\mathcal{V}$ in \eqref{R0split}, we relabel $k$ in \eqref{V} as $4k$, now with $k$ odd, to see that
\begin{equation*}
\mathcal{V} = -\left( \frac{2d_1}{m_1m_2\nu}\right)    \sum_{\substack{k \in\mathbb{Z} \\ k \text{ odd}}}\hat{F}_{\nu}\left( \frac{kX}{2\alpha^2d_1m_1m_2\nu}\right) {\tau}_k(m_1m_2\nu)
\end{equation*}
since $e(h/2)=-1$ for odd $h$ and $\tau_{4k}(n)=\tau_k(n)$ for odd $n$ by \eqref{Gaussdef}. Into this we insert the second expression for $\tau_k (n)$ in \eqref{Gaussdef}. Since $\left( \frac{-1}{n}\right) G_k(n)=G_{-k}(n)$ by \eqref{eq: defn of Gk}, we may split our sum expression for $\mathcal{V}$ into two, one with $G_k(n)$ and the other with $G_{-k}(n)$. We relabel $k$ as $-k$ in the latter and combine the result with the former to arrive at
\begin{equation}\label{V1}
\mathcal{V} = -\left( \frac{2d_1}{m_1m_2\nu}\right)    \sum_{\substack{k \in\mathbb{Z} \\ k \text{ odd}}}\tilde{F}_{\nu}\left( \frac{kX}{2\alpha^2d_1m_1m_2\nu}\right) G_k(m_1m_2\nu),
\end{equation}
where $\tilde{F}(\xi)$ is defined by
\begin{equation*}
\tilde{F}(\xi)=\frac{1+i}{2}\hat{F}(\xi)+ \frac{1-i}{2}\hat{F}(-\xi)=\int_{-\infty}^{\infty} (\cos(2\pi \xi x)+\sin(2\pi \xi x))F(x)\,dx.
\end{equation*}
We then proceed as we did for $\mathcal{Q}_1$, using \cite[Lemma 5.2]{Sou00} instead of Lemma~\ref{lem: properties of h(xi,w)}. We arrive at versions of \eqref{Q14}, \eqref{eq: defn of P1}, and \eqref{eq: defn of R1} which show that the residue at $w=0$ equals zero because $2kd_1 \neq \square$ when $kd_1$ is odd. This leads to
\begin{equation}\label{Vestimated}
\frac{X}{(\sqrt{2}-1)^4}\sum_{\substack{ d\leq D \\ d \text{ odd} }} \mu^2(d)\lambda_d \mathop{\sum\sum}_{\substack{m_1,m_2\leq M \\ (m_1m_2,2d)=1}} \frac{b_{m_1}b_{m_2}}{(m_1m_2)^{3/2}} \sum_{\substack{ \nu=1 \\ (\nu,2d)=1}}^{\infty} \frac{d_2(\nu) }{\nu^{3/2}} \ \sum_{\substack{ \alpha\leq Y \\ (\alpha,2m_1m_2\nu)=1}} \frac{\mu(\alpha)}{\alpha^2d_1}\, \mathcal{V} =O(X^{1-\varepsilon})
\end{equation}
under the conditions in Subsection \ref{subsection: conditions for the parameters}.

Lastly, to estimate the sum with $\mathcal{W}$ in \eqref{R0split}, we relabel $k$ in \eqref{W} as $8k$ to write
\begin{equation*}
\mathcal{W} =\left( \frac{d_1}{m_1m_2\nu}\right)    \sum_{\substack{k \in\mathbb{Z} \\ k\neq 0}}\hat{F}_{\nu}\left( \frac{kX}{\alpha^2d_1m_1m_2\nu}\right) {\tau}_k(m_1m_2\nu)
\end{equation*}
using the fact that $e(h)=1$ for any integer $h$ and $\tau_{8k}(n)=(\frac{2}{n})\tau_k(n)$ for odd $n$ by \eqref{Gaussdef}. Into this we insert the second expression for $\tau_k(n)$ in \eqref{Gaussdef}, apply $\left( \frac{-1}{n}\right) G_k(n)=G_{-k}(n)$, and recombine the $k$ and $-k$ terms as we did for $\mathcal{V}$ in \eqref{V1} to deduce that
\begin{equation*}
\mathcal{W} =\left( \frac{d_1}{m_1m_2\nu}\right)    \sum_{\substack{k \in\mathbb{Z} \\ k\neq 0}}\tilde{F}_{\nu}\left( \frac{kX}{\alpha^2d_1m_1m_2\nu}\right) G_k(m_1m_2\nu).
\end{equation*}
We then proceed as we did for $\mathcal{Q}_1$, using \cite[Lemma 5.2]{Sou00} instead of Lemma~\ref{lem: properties of h(xi,w)}. Since we are now summing over all nonzero integers $k$ and not just the odd ones, instead of \eqref{eq: write inner j sum as Euler product} we use
\begin{align*}
\sum_{j=1}^{\infty}j^{-2s+2w} \mathcal{G}(1+w;j^2,m_1m_2,\alpha d)
= \prod_{p} \sum_{b=0}^{\infty} p^{2b(w-s)} \mathcal{G}_p(1+w;p^{2b},m_1m_2,\alpha d) \\
= (m_1m_2)^{1-s+w} \ell_1^{s-w-\frac{1}{2}} \zeta(2s-2w)\zeta(2s+1) \mathcal{H}_1(s-w,1+w;m_1m_2,\alpha d).
\end{align*}
We arrive at
\begin{equation}\label{Westimated}
\begin{split}
\frac{X}{(\sqrt{2}-1)^4}\sum_{\substack{ d\leq D \\ d \text{ odd} }} \mu^2(d)\lambda_d \mathop{\sum\sum}_{\substack{m_1,m_2\leq M \\ (m_1m_2,2d)=1}} \frac{b_{m_1}b_{m_2}}{(m_1m_2)^{3/2}} \sum_{\substack{ \nu=1 \\ (\nu,2d)=1}}^{\infty} \frac{d_2(\nu) }{\nu^{3/2}} \ \sum_{\substack{ \alpha\leq Y \\ (\alpha,2m_1m_2\nu)=1}} \frac{\mu(\alpha)}{\alpha^2d_1}\, \mathcal{W} \\
= \frac{X}{(\sqrt{2}-1)^4}\sum_{\substack{ d\leq D \\ d \text{ odd} }} \mu^2(d)\lambda_d \mathop{\sum\sum}_{\substack{m_1,m_2\leq M \\ (m_1m_2,2d)=1}} \frac{b_{m_1}b_{m_2}}{(m_1m_2)^{3/2}} \,  \underset{w=0}{\mbox{Res}} \   X^w \frac{1}{2\pi i} \int_{\left( \frac{1}{\log X}\right)}    \Gamma_1(s-w) \\
\times  \mathcal{K}(s,w;m_1m_2, d)\,\frac{ds}{s} + O(X^{1-\varepsilon}).
\end{split}
\end{equation}

\subsection{Putting together the estimates}

From \eqref{R0split}, \eqref{eq: combining Q1 Q2 Q3 Q4}, \eqref{U1U2}, \eqref{Vestimated}, and \eqref{Westimated}, we deduce that
\begin{equation*}
\begin{split}
& \mathcal{B}= \frac{X}{(\sqrt{2}-1)^4}\sum_{\substack{ d\leq D \\ d \text{ odd} }} \mu^2(d)\lambda_d \mathop{\sum\sum}_{\substack{m_1,m_2\leq M \\ (m_1m_2,2d)=1}} \frac{b_{m_1}b_{m_2}}{(m_1m_2)^{3/2}} \, \underset{w=0}{\mbox{Res}} \   X^w \frac{1}{2\pi i} \int_{\left( \frac{1}{\log X}\right)}    \Gamma_1(s-w) \\
& \ \times  \Big(8^{s-w}\sqrt{2} +4^{s-w}-2^{s-w}\sqrt{2}\Big) \mathcal{K}(s,w;m_1m_2, d)\,\frac{ds}{s} + O(X^{1-\varepsilon}).
\end{split}
\end{equation*}
We next evaluate the residue at $w=0$. Note that, for fixed $s$, the integrand has a pole of order at most 2 at $w=0$.  We use \eqref{eq: residue as derivative} with $n=2$ to write
\begin{equation}\label{eq: B with logarithmic derivatives}
\begin{split}
\mathcal{B}= \frac{X}{(\sqrt{2}-1)^4}\sum_{\substack{ d\leq D \\ d \text{ odd} }} \mu^2(d)\lambda_d \mathop{\sum\sum}_{\substack{m_1,m_2\leq M \\ (m_1m_2,2d)=1}} 
& \frac{b_{m_1}b_{m_2}}{(m_1m_2)^{3/2}} \,  \frac{1}{2\pi i} \int_{\left( \frac{1}{\log X}\right)}    \Gamma_1(s) \Big(8^{s}\sqrt{2} +4^{s}-2^{s}\sqrt{2}\Big) \\
\times   \mathcal{K}(s,0;m_1m_2, d) \Bigg\{ \log X -\frac{\Gamma_1'(s)}{\Gamma_1(s)} -
& (\log 2)\frac{3\cdot 8^{s}\sqrt{2} +2\cdot 4^{s}-2^{s}\sqrt{2}}{8^{s}\sqrt{2} +4^{s}-2^{s}\sqrt{2}} \\
& \ \ \ \ \ \ \ \ + \frac{\frac{\partial}{\partial w}\mathcal{K}(s,w;m_1m_2, d)}{\mathcal{K}(s,w;m_1m_2, d)} \Bigg|_{w=0} \Bigg\}\,\frac{ds}{s} + O(X^{1-\varepsilon}).
\end{split}
\end{equation}
From the definitions \eqref{eq: defn of K for 2nd moment analysis} and \eqref{eq: defn of H 1} of $\mathcal{K}$ and $\mathcal{H}$, we see that, after some simplification,
\begin{equation}\label{sweeping under the rug 1}
\begin{split}
& \Big(8^{s}\sqrt{2} +4^{s}-2^{s}\sqrt{2}\Big)  \mathcal{K}(s,0;m_1m_2, d) \\
& = \frac{\check{\Phi}(0)}{4} \frac{\Gamma \left( \frac{s}{2} + \frac{1}{4}\right)^2}{\Gamma \left( \frac{1}{4}\right)^2} \left(\frac{4}{\pi}\right)^s\zeta(2s)\zeta(2s+1) \left(1 - \frac{1}{2^{\frac{1}{2}+s}} \right)\left(1 - \frac{1}{2^{\frac{1}{2}-s}} \right)\left(\frac{5}{2} - 4^s-4^{-s}   \right) \\
& \ \ \ \ \times \frac{\varphi(d m_1m_2)^2}{d^3m_1m_2\sqrt{\ell_1}}\sum_{ab=\ell_1} \left( \frac{a}{b}\right)^s \prod_{\substack{p \mid m_1m_2 \\ p \nmid \ell_1}} \left(1 + \frac{1}{p}\right) \prod_{p|d} \left( 1-\frac{1}{p^{1+2s}}\right) \left( 1-\frac{1}{p^{1-2s}}\right) \\
& \ \ \ \ \times \prod_{p\nmid 2m_1m_2 d} \left\{ \left( 1-\frac{1}{p}\right)^2 \left( 1+\frac{2}{p}+\frac{1}{p^3}-\frac{1}{p^{2-2s}}-\frac{1}{p^{2+2s}}\right)  \right\},
\end{split}
\end{equation}
where $\ell_1$ is defined by \eqref{eq: defn of ell 1}, and
\begin{equation}\label{sweeping under the rug 2}
\begin{split}
& - (\log 2) \frac{3\cdot 8^{s}\sqrt{2} +2\cdot 4^{s}-2^{s}\sqrt{2}}{8^{s}\sqrt{2} +4^{s}-2^{s}\sqrt{2}} + \frac{\frac{\partial}{\partial w}\mathcal{K}(s,w;m_1m_2, d)}{\mathcal{K}(s,w;m_1m_2, d)} \Bigg|_{w=0}  \\
& = 2\gamma+ \frac{(\check{\Phi})'(0)}{\check{\Phi}(0)}  - \log (2\ell_1) - 2\frac{\zeta'}{\zeta}(2s) + 2 \frac{\zeta'}{\zeta}(2s+1) +\frac{\log 2}{\left(\sqrt{2}+2^s\right)\left(\sqrt{2}+2^{-s}\right)} \\
& \ \ \ \ +\sum_{p \mid d}\left(\frac{2\log p}{p-1} + \frac{2\log p}{p^{1+2s}-1} + \frac{2\log p}{p^{1-2s}-1} \right)  +\sum_{p \mid m_1m_2} \frac{2\log p}{p-1} - \sum_{\substack{p \mid m_1m_2 \\ p\nmid \ell_1}} \frac{2\log p}{p+1} \\
& \ \ \ \  +\sum_{p \nmid 2 m_1m_2 d}\left(\frac{2\log p}{p-1} - \left( \frac{2\log p}{p}\right)\frac{1 + \frac{2}{p^2} -\frac{1}{p}\left( p^{2s}+p^{-2s}\right) }{1 + \frac{2}{p} + \frac{1}{p^3} - \frac{1}{p^2}\left( p^{2s}+p^{-2s} \right) } \right).
\end{split}
\end{equation}
Now the definition \eqref{eq: defn of Gamma 1} of $\Gamma_1(u)$, the Legendre duplication formula, the functional equation of $\zeta(s)$, and the identity $\Gamma(z)\Gamma(1-z)=\pi \csc(\pi z)$ imply that the functions
$$
\frac{\Gamma^2 \left( \frac{s}{2} + \frac{1}{4}\right)}{\Gamma^2 \left( \frac{1}{4}\right)} \left( \frac{4}{\pi}\right)^s \Gamma_1(s) \zeta(2s)\zeta(2s+1)
$$
and
$$
- \frac{\Gamma_1'(s)}{\Gamma_1(s)} - 2\frac{\zeta'}{\zeta}(2s) + 2 \frac{\zeta'}{\zeta}(2s+1)
$$
are even functions of $s$. Hence \eqref{sweeping under the rug 1} and \eqref{sweeping under the rug 2} are even functions of $s$. It follows that the integrand in \eqref{eq: B with logarithmic derivatives} is an odd function of $s$. We move the line of integration in \eqref{eq: B with logarithmic derivatives} to Re$(s)=-\frac{1}{\log X}$, leaving a residue at $s=0$. In the new integral, we make a change of variables $s\mapsto -s$ to see that, since its integrand is odd, it equals the negative of the original integral in \eqref{eq: B with logarithmic derivatives}. Therefore twice the original integral equals the residue at $s=0$. We write this residue as an integral along the circle $|s|=\frac{1}{\log X}$, taken in the positive direction, and arrive at
\begin{equation}\label{eq: B as a residue}
\begin{split}
\mathcal{B}= \frac{X}{(\sqrt{2}-1)^4}\sum_{\substack{ d\leq D \\ d \text{ odd} }} \mu^2(d)\lambda_d \mathop{\sum\sum}_{\substack{m_1,m_2\leq M \\ (m_1m_2,2d)=1}} 
& \frac{b_{m_1}b_{m_2}}{(m_1m_2)^{3/2}} \,       \frac{1}{4\pi i} \oint_{|s|=\frac{1}{\log X}}\Gamma_1(s) \Big(8^{s}\sqrt{2} +4^{s}-2^{s}\sqrt{2}\Big) \\
\times   \mathcal{K}(s,0;m_1m_2, d) \Bigg\{ \log X -\frac{\Gamma_1'(s)}{\Gamma_1(s)} -
& (\log 2)\frac{3\cdot 8^{s}\sqrt{2} +2\cdot 4^{s}-2^{s}\sqrt{2}}{8^{s}\sqrt{2} +4^{s}-2^{s}\sqrt{2}} \\
& \ \ \ \ \ \ \ \ + \frac{\frac{\partial}{\partial w}\mathcal{K}(s,w;m_1m_2, d)}{\mathcal{K}(s,w;m_1m_2, d)} \Bigg|_{w=0} \Bigg\} \,\frac{ds}{s}+ O(X^{1-\varepsilon}).
\end{split}
\end{equation}

The next step is to carry out the summation over $d$. From \eqref{sweeping under the rug 1} and \eqref{sweeping under the rug 2}, we see that we need to evaluate the sums $\Sigma_1$ and $\Sigma_2$ defined by
\begin{equation}\label{eq: defn of Sigma 1}
\begin{split}
\Sigma_1 = \sum_{\substack{ d\leq D \\ (d,2m_1m_2)=1 }} \mu^2(d)\lambda_d
& \frac{\varphi(d)^2}{d^3} \prod_{p|d} \left( 1-\frac{1}{p^{1+2s}}\right) \left( 1-\frac{1}{p^{1-2s}}\right) \\
& \times \prod_{p\nmid 2m_1m_2 d} \left\{ \left( 1-\frac{1}{p}\right)^2 \left( 1+\frac{2}{p}+\frac{1}{p^3}-\frac{1}{p^{2-2s}}-\frac{1}{p^{2+2s}}\right)  \right\}
\end{split}
\end{equation}
and
\begin{equation}\label{eq: defn of Sigma 2}
\begin{split}
\Sigma_2 = \sum_{\substack{ d\leq D \\ (d,2m_1m_2)=1 }}
& \mu^2(d)\lambda_d\frac{\varphi(d)^2}{d^3} \prod_{p|d} \left( 1-\frac{1}{p^{1+2s}}\right) \left( 1-\frac{1}{p^{1-2s}}\right)  \\
& \times \prod_{p\nmid 2m_1m_2 d} \left\{ \left( 1-\frac{1}{p}\right)^2 \left( 1+\frac{2}{p}+\frac{1}{p^3}-\frac{1}{p^{2-2s}}-\frac{1}{p^{2+2s}}\right)  \right\} \sum_{p|d}J(p,s), 
\end{split}
\end{equation}
where
\begin{equation}\label{eq: defn of J}
J(p,s) = \frac{2\log p}{p^{1+2s}-1} + \frac{2\log p}{p^{1-2s}-1} + \left( \frac{2\log p}{p}\right)\frac{1 + \frac{2}{p^2} -\frac{1}{p}\left( p^{2s}+p^{-2s}\right) }{1 + \frac{2}{p} + \frac{1}{p^3} - \frac{1}{p^2}\left( p^{2s}+p^{-2s} \right) }
\end{equation}
and $|s|=\frac{1}{\log X}$. We only estimate $\Sigma_1$ since $\Sigma_2$ may be treated in the same way, except using Lemma~\ref{sievewithsum} instead of Lemma~\ref{sieve}. We rearrange the factors in \eqref{eq: defn of Sigma 1} to write $\Sigma_1$ as
\begin{equation}\label{eq: Sigma 1 after rearrangement}
\begin{split}
\Sigma_1 = \prod_{p\nmid 2m_1m_2 } \left\{ \left( 1-\frac{1}{p}\right)^2 \left( 1+\frac{2}{p}+\frac{1}{p^3}-\frac{1}{p^{2-2s}}-\frac{1}{p^{2+2s}}\right)  \right\} \sum_{\substack{ d\leq D \\ (d,2m_1m_2)=1 }} \frac{\mu^2(d)\lambda_d}{d} \\
\times \prod_{p|d} \left( 1-\frac{1}{p^{1+2s}}\right) \left( 1-\frac{1}{p^{1-2s}}\right) \left( 1+\frac{2}{p}+\frac{1}{p^3}-\frac{1}{p^{2-2s}}-\frac{1}{p^{2+2s}}\right)^{-1}.
\end{split}
\end{equation}
Now recall the definition \eqref{eq:sieve section, defn of z0} of $z_0$ and the definition \eqref{lambda} of $\lambda_d$. Factoring out the product over primes $p>z_0$, we see that
\begin{equation*}
\begin{split}
\prod_{p\nmid 2m_1m_2 } & \left\{ \left( 1-\frac{1}{p}\right)^2 \left( 1+\frac{2}{p}+\frac{1}{p^3}-\frac{1}{p^{2-2s}}-\frac{1}{p^{2+2s}}\right)  \right\} \\
& = \left( 1+ O\left( \frac{1}{z_0}\right)\right) \prod_{ \substack{p\nmid 2m_1m_2 \\ p\leq z_0} } \left\{ \left( 1-\frac{1}{p}\right)^2 \left( 1+\frac{2}{p}+\frac{1}{p^3}-\frac{1}{p^{2-2s}}-\frac{1}{p^{2+2s}}\right)  \right\}.
\end{split}
\end{equation*}
From this, \eqref{eq: Sigma 1 after rearrangement}, Lemma~\ref{sieve}, and some simplification, we deduce that
\begin{equation}\label{eq: Sigma 1 after sieve Lemma}
\begin{split}
\Sigma_1 = \left( 1+O\left( \frac{1}{z_0}\right)\right) \frac{1+o(1)}{\log R} \prod_{ \substack{p\nmid 2m_1m_2 \\ p\leq z_0 }} \left( 1-\frac{1}{p^2}\right)  \prod_{ \substack{p| 2m_1m_2 \\ p\leq z_0 }} \left( 1-\frac{1}{p}\right)^{-1}+ O\left( \frac{1}{(\log R)^{2018}}\right).
\end{split}
\end{equation}
The condition $p\leq z_0$ may be omitted because $\prod_{p>z_0}(1+O(\frac{1}{p^2}))=1+O(\frac{1}{z_0})$ and
$$
\prod_{ \substack{p| 2m_1m_2 \\ p> z_0 }} \left( 1-\frac{1}{p}\right)^{-1} = \left( 1+O\left( \frac{1}{z_0}\right)\right)^{O(\log X)} = 1+O\left( \frac{\log X}{z_0}\right).
$$
The contributions of the error terms $O\left(\frac{1}{z_0}\right)$ and $O\left(\frac{\log X}{z_0}\right)$ are negligible. From these and \eqref{eq: Sigma 1 after sieve Lemma}, we arrive at
\begin{equation}\label{eq: Sigma 1 simplified}
\Sigma_1 =  \frac{2m_1m_2}{\varphi(m_1m_2)}\prod_{p\nmid 2m_1m_2 } \left( 1-\frac{1}{p^2}\right)   \frac{1+o(1)}{\log R} + O\left( \frac{1}{(\log R)^{2018}}\right).
\end{equation}
In a similar way, but using Lemma~\ref{sievewithsum} instead of Lemma~\ref{sieve}, we deduce from \eqref{eq: defn of Sigma 2} that
\begin{equation}\label{eq: Sigma 2 simplified}
\begin{split}
\Sigma_2 =  - & \frac{2m_1m_2}{\varphi(m_1m_2)}\prod_{p\nmid 2m_1m_2 } \left( 1-\frac{1}{p^2}\right)   \frac{1+o(1)}{\log R} \\
& \times \sum_{ p\nmid 2m_1m_2 } \frac{J(p,s)}{p+1}\left( 1-\frac{1}{p^{1+2s}}\right) \left( 1-\frac{1}{p^{1-2s}}\right)  + O\left( \frac{1}{(\log R)^{2018}}\right).
\end{split}
\end{equation}
In view of the expressions \eqref{sweeping under the rug 1} and \eqref{sweeping under the rug 2} and the definitions \eqref{eq: defn of Sigma 1} and \eqref{eq: defn of Sigma 2}, it now follows from \eqref{eq: B as a residue}, \eqref{eq: Sigma 1 simplified} and \eqref{eq: Sigma 2 simplified} that
\begin{equation}\label{eq: B after evaluation of Sigma 1 and 2}
\begin{split}
& \mathcal{B}= \frac{X\check{\Phi}(0)}{3\zeta(2)(\sqrt{2}-1)^4} \frac{1+o(1)}{\log R}  \mathop{\sum\sum}_{\substack{m_1,m_2\leq M \\ (m_1m_2,2)=1}}  \frac{b_{m_1}b_{m_2}}{\sqrt{m_1m_2\ell_1}}   \prod_{p|\ell_1}\left( \frac{p}{p+1}\right)\\ 
& \times \frac{1}{2\pi i} \oint_{|s|=\frac{1}{\log X}}\sum_{ab=\ell_1} \left( \frac{a}{b}\right)^s  \Gamma_1(s)\frac{\Gamma \left( \frac{s}{2} + \frac{1}{4}\right)^2}{\Gamma \left( \frac{1}{4}\right)^2} \left(\frac{4}{\pi}\right)^s\zeta(2s)\zeta(2s+1)\left(1 - \frac{1}{2^{\frac{1}{2}+s}} \right)\left(1 - \frac{1}{2^{\frac{1}{2}-s}} \right)\\
&\times  \left(\frac{5}{2} - 4^s-4^{-s}   \right)    \Bigg\{ \log \left( \frac{X}{2\ell_1}\right)  +2\gamma+ \frac{(\check{\Phi})'(0)}{\check{\Phi}(0)} -\frac{\Gamma_1'(s)}{\Gamma_1(s)} - 2\frac{\zeta'}{\zeta}(2s) + 2 \frac{\zeta'}{\zeta}(2s+1) \\
&   +\frac{\log 2}{\left(\sqrt{2}+2^s\right)\left(\sqrt{2}+2^{-s}\right)} +\sum_{p\neq 2} \eta_1(p,s) +\sum_{\substack{p|m_1m_2 \\ p\nmid \ell_1}} \eta_2(p,s)+ \sum_{p|\ell_1} \eta_3(p,s)\Bigg\} \,\frac{ds}{s}+ O\left( \frac{X}{(\log R)^{2018}}\right),
\end{split}
\end{equation}
where
\begin{equation*}
\begin{split}
\eta_1(p,s) = \frac{2\log p}{p-1} - \left( \frac{2\log p}{p}\right)\frac{1 + \frac{2}{p^2} -\frac{1}{p}\left( p^{2s}+p^{-2s}\right) }{1 + \frac{2}{p} + \frac{1}{p^3} - \frac{1}{p^2}\left( p^{2s}+p^{-2s} \right) } - \frac{J(p,s)}{p+1}\left( 1-\frac{1}{p^{1+2s}}\right) \left( 1-\frac{1}{p^{1-2s}}\right),
\end{split}
\end{equation*}
\begin{equation}\label{eq: defn of eta 2}
\eta_2(p,s) = \frac{2\log p}{p-1} - \frac{2\log p}{p+1} -\eta_1(p,s),
\end{equation}
and
\begin{equation*}
\eta_3(p,s) = \frac{2\log p}{p-1} - \eta_1(p,s),
\end{equation*}
with $J(p,s)$ defined by \eqref{eq: defn of J}.

Next, we carry out the summation over $m_1,m_2$. We see from \eqref{eq: B after evaluation of Sigma 1 and 2} that we need to evaluate the sums $\Upsilon_1$, $\Upsilon_2$, $\Upsilon_3$, and $\Upsilon_4$ defined by
\begin{equation}\label{eq: defn of Upsilon 1}
\Upsilon_1=\mathop{\sum\sum}_{\substack{m_1,m_2\leq M \\ (m_1m_2,2)=1}}  \frac{b_{m_1}b_{m_2}}{\sqrt{m_1m_2\ell_1}}   \prod_{p|\ell_1}\left( \frac{p}{p+1}\right)\sum_{ab=\ell_1} \left( \frac{a}{b}\right)^s,
\end{equation}
\begin{equation}\label{eq: defn of Upsilon 2}
\Upsilon_2=-\mathop{\sum\sum}_{\substack{m_1,m_2\leq M \\ (m_1m_2,2)=1}}  \frac{b_{m_1}b_{m_2}}{\sqrt{m_1m_2\ell_1}}   \prod_{p|\ell_1}\left( \frac{p}{p+1}\right)\sum_{ab=\ell_1} \left( \frac{a}{b}\right)^s \log \ell_1,
\end{equation}
\begin{equation}\label{eq: defn of Upsilon 3}
\Upsilon_3=\mathop{\sum\sum}_{\substack{m_1,m_2\leq M \\ (m_1m_2,2)=1}}  \frac{b_{m_1}b_{m_2}}{\sqrt{m_1m_2\ell_1}}   \prod_{p|\ell_1}\left( \frac{p}{p+1}\right)\sum_{ab=\ell_1} \left( \frac{a}{b}\right)^s \sum_{\substack{p | \ell_1}} \eta_2(p,s),
\end{equation}
and
\begin{equation}\label{eq: defn of Upsilon 4}
\Upsilon_4=\mathop{\sum\sum}_{\substack{m_1,m_2\leq M \\ (m_1m_2,2)=1}}  \frac{b_{m_1}b_{m_2}}{\sqrt{m_1m_2\ell_1}}   \prod_{p|\ell_1}\left( \frac{p}{p+1}\right)\sum_{ab=\ell_1} \left( \frac{a}{b}\right)^s \sum_{\substack{p|m_1m_2 \\ p\nmid \ell_1}} \eta_3(p,s),
\end{equation}
with $|s|=\frac{1}{\log X}$. 

To estimate $\Upsilon_1$, observe that if $m_1$ and $m_2$ are square-free then \eqref{eq: defn of ell 1} implies
\begin{equation}\label{eq: ell 1 in terms of m1 and m2}
\ell_1 = \frac{m_1m_2}{(m_1,m_2)^2}
\end{equation}
and
\begin{equation}\label{eq: ab sum as an Euler product}
\sum_{ab=\ell_1} \left( \frac{a}{b}\right)^s = \prod_{p|\ell_1} (p^s+p^{-s}).
\end{equation}
From these, the definition \eqref{eq: defn of mollifier coeffs bm} of $b_m$, and the Fourier inversion formula \eqref{eq: write H as Fourier integral}, we deduce from \eqref{eq: defn of Upsilon 1} that
\begin{equation*}
\Upsilon_1 = \int_{-\infty}^{\infty}\int_{-\infty}^{\infty} h(z_1)h(z_2) \mathop{\sum \sum}_{\substack{ (m_1m_2,2)=1}} \frac{\mu(m_1)\mu(m_2)(m_1,m_2)}{m_1^{1+\frac{1+iz_1}{\log M}}m_2^{1+\frac{1+iz_2}{\log M}}}   \prod_{ \substack{p|m_1m_2 \\ p\nmid (m_1,m_2)}}(p^s+p^{-s})\left( \frac{p}{p+1}\right) \, dz_1dz_2.
\end{equation*}
Thus, writing the sum as an Euler product, we see that
\begin{align*}
\Upsilon_1=\int_{-\infty}^{\infty}\int_{-\infty}^{\infty} h(z_1)h(z_2) \prod_{p>2} & \Bigg(1-\frac{1}{p^{1+\frac{1+iz_1}{\log M}}}(p^s+p^{-s})\left( \frac{p}{p+1}\right) \\
& \ \ \ \ -\frac{1}{p^{1+\frac{1+iz_2}{\log M}}}(p^s+p^{-s})\left( \frac{p}{p+1}\right) + \frac{1}{p^{1+\frac{2+iz_1+iz_2}{\log M}}}\Bigg) \, dz_1dz_2.
\end{align*}
We write this as
\begin{equation}\label{eq: Upsilon with W}
\Upsilon_1 = \int_{-\infty}^{\infty}\int_{-\infty}^{\infty}   \frac{h(z_1)h(z_2)\zeta\left(1+\frac{2+iz_1+iz_2}{\log M}\right)  W(s,z_1,z_2,\tfrac{1}{\log M})\, dz_1dz_2}{ \zeta\left( 1+\frac{1+iz_1}{\log M}+s\right) \zeta\left( 1+\frac{1+iz_1}{\log M}-s\right) \zeta\left( 1+\frac{1+iz_2}{\log M}+s\right) \zeta\left( 1+\frac{1+iz_2}{\log M}-s\right) },
\end{equation}
where $W(s,z_1,z_2,\tfrac{1}{\log M})$ is an Euler product that is bounded and holomorphic for $|s|\leq \varepsilon$ and complex $z_1,z_2$ with $|\text{Im}(z_1)|,|\text{Im}(z_2)|\leq \varepsilon \log M$. Note that this definition of $W$ implies
\begin{equation}\label{eq: W at 0000}
W(0,0,0,0) = 8\prod_{p>2}\left( 1-\frac{4}{p+1}+\frac{1}{p}\right)\left(1-\frac{1}{p}\right)^{-3} = 6\zeta(2),
\end{equation}
a fact we use shortly. By \eqref{eq: h has rapid decay}, we may truncate the integrals in \eqref{eq: Upsilon with W} to the range $|z_1|,|z_2|\leq \sqrt{\log M}$, introducing a negligible error. On this range of $z_1$ and $z_2$, the function $W$ and the zeta-functions in \eqref{eq: Upsilon with W} may be written as Laurent series. The contributions of the terms other than the first terms of these Laurent expansions are a factor of $(\log X)^{1-\varepsilon}$ smaller than the contribution of the first terms. The first term of the Laurent expansion of $W$ is given by \eqref{eq: W at 0000}. We thus arrive at
\begin{equation*}
\begin{split}
\Upsilon_1 = 6\zeta(2) \mathop{\int\int}_{|z_i| \leq \sqrt{\log M}} h(z_1)h(z_2) \left(\frac{\log M}{2+iz_1+iz_2}\right) \left( \frac{1+iz_1}{\log M}-s\right)\left( \frac{1+iz_1}{\log M}+s\right) \\
\times \left( \frac{1+iz_2}{\log M}+s\right)\left(\frac{1+iz_2}{\log M}-s\right)\, dz_1dz_2 + O\left( \frac{1}{(\log X)^{4-\varepsilon}}\right).
\end{split}
\end{equation*}
By \eqref{eq: h has rapid decay}, we may extend the range of integration to $\mathbb{R}^2$, introducing a negligible error. We then apply \eqref{eq: integrals of derivatives of H} to deduce that
\begin{equation}\label{eq: Upsilon 1 evaluated}
\begin{split}
\Upsilon_1 = 6\zeta(2) \Bigg(\frac{1}{\log^3 M} \int_0^1 H''(t)^2\,dt  - \frac{2s^2}{\log M}\int_0^1 H(t)H''(t)\,dt \\
+  s^4 \log M \int_0^1 H(t)^2\,dt\Bigg) + O\left( \frac{1}{(\log X)^{4-\varepsilon}}\right).
\end{split}
\end{equation}

Having evaluated $\Upsilon_1$, we next estimate $\Upsilon_2$. Using the residue theorem, we write
\begin{equation*}
-\log \ell_1 = \frac{1}{2\pi i }\oint_{|y|=\frac{1}{2\log X}} \ell_1^{-y}\,\frac{dy}{y^2}.
\end{equation*}
From this, \eqref{eq: defn of Upsilon 2}, \eqref{eq: ell 1 in terms of m1 and m2}, \eqref{eq: ab sum as an Euler product}, the definition \eqref{eq: defn of mollifier coeffs bm} of $b_m$, and the Fourier inversion formula \eqref{eq: write H as Fourier integral}, it follows that
\begin{equation*}
\begin{split}
\Upsilon_2 = \frac{1}{2\pi i }\oint_{|y|=\frac{1}{2\log X}}\int_{-\infty}^{\infty}\int_{-\infty}^{\infty} h(z_1)h(z_2) \mathop{\sum \sum}_{\substack{(m_1m_2,2)=1}} \frac{\mu(m_1)\mu(m_2)(m_1,m_2)^{1+2y}}{m_1^{1+\frac{1+iz_1}{\log M}+y}m_2^{1+\frac{1+iz_2}{\log M}+y}}  \\
\times  \prod_{ \substack{p|m_1m_2 \\ p\nmid (m_1,m_2)}}(p^s+p^{-s})\left( \frac{p}{p+1}\right) \, dz_1dz_2\frac{dy}{y^2}.
\end{split}
\end{equation*}
We express the sum as an Euler product to see that
\begin{equation*}
\begin{split}
\Upsilon_2 = \frac{1}{2\pi i }\oint_{|y|=\frac{1}{2\log X}} \int_{-\infty}^{\infty}\int_{-\infty}^{\infty} h(z_1)h(z_2) \prod_{p>2} \Bigg(1-\frac{1}{p^{1+\frac{1+iz_1}{\log M}+y}}(p^s+p^{-s})\left( \frac{p}{p+1}\right) \\
-\frac{1}{p^{1+\frac{1+iz_2}{\log M}+y}}(p^s+p^{-s})\left( \frac{p}{p+1}\right) + \frac{1}{p^{1+\frac{2+iz_1+iz_2}{\log M}}}\Bigg) \, dz_1dz_2 \,\frac{dy}{y^2}.
\end{split}
\end{equation*}
Write this as
\begin{equation*}
\begin{split}
\Upsilon_2 = \frac{1}{2\pi i }\oint_{|y|=\frac{1}{2\log X}} &\int_{-\infty}^{\infty}\int_{-\infty}^{\infty}  h(z_1)h(z_2) \zeta\left(1+\tfrac{2+iz_1+iz_2}{\log M}\right) V(s,z_1,z_2,\tfrac{1}{\log M},y)\\
&\times \zeta^{-1}\left( 1+\tfrac{1+iz_1}{\log M}+y+s\right) \zeta^{-1}\left( 1+\tfrac{1+iz_1}{\log M}+y-s\right)  \\
&\times \zeta^{-1}\left( 1+\tfrac{1+iz_2}{\log M}+y+s\right) \zeta^{-1}\left( 1+\tfrac{1+iz_2}{\log M}+y-s\right) \,dz_1dz_2\frac{dy}{y^2},
\end{split}
\end{equation*}
where $V(s,z_1,z_2,\tfrac{1}{\log M},y)$ is an Euler product that is bounded and holomorphic for $|s|,|y|\leq \varepsilon$ and complex $z_1,z_2$ with $|\text{Im}(z_1)|,|\text{Im}(z_2)|\leq \varepsilon \log M$. This definition of $V$ implies that $V(0,0,0,0,0)=6\zeta(2)$. As in our treatment of $\Upsilon_1$, we use \eqref{eq: h has rapid decay} to truncate the integrals. Then we write the function $V$ and the zeta-functions as Laurent series. The main contribution arises from the first terms of the Laurent expansions, and we arrive at
\begin{equation*}
\begin{split}
\Upsilon_2 = \frac{6\zeta(2)}{2\pi i }\oint_{|y|=\frac{1}{2\log X}}\mathop{\int\int}_{|z_i| \leq \sqrt{\log M}}   h(z_1)h(z_2) \left(\frac{\log M}{2+iz_1+iz_2}\right) \left( \frac{1+iz_1}{\log M}+y-s\right) \\
\times \left( \frac{1+iz_1}{\log M}+y+s\right)\left( \frac{1+iz_2}{\log M}+y+s\right)\left(\frac{1+iz_2}{\log M}+y-s\right)\, dz_1dz_2 \frac{dy}{y^2}+ O\left( \frac{1}{(\log X)^{3-\varepsilon}}\right).
\end{split}
\end{equation*}
We carry out the integration over $y$ by applying the formula \eqref{eq: residue as derivative} with $n=2$ and deduce that
\begin{equation*}
\begin{split}
\Upsilon_2 = 
& 6\zeta(2)\mathop{\int\int}_{|z_i| \leq \sqrt{\log M}}   h(z_1)h(z_2) \left(\frac{\log M}{2+iz_1+iz_2}\right)   \\
& \times \Bigg\{ \left( \frac{1+iz_1}{\log M}+s\right) \left( \frac{(1+iz_2)^2}{(\log M)^2}-s^2\right) + \left( \frac{1+iz_1}{\log M}-s\right) \left( \frac{(1+iz_2)^2}{(\log M)^2}-s^2\right)  \\
&  + \left( \frac{1+iz_2}{\log M}+s\right) \left( \frac{(1+iz_1)^2}{(\log M)^2}-s^2\right) + \left( \frac{1+iz_2}{\log M}-s\right) \left( \frac{(1+iz_1)^2}{(\log M)^2}-s^2\right) \Bigg\}\, dz_1dz_2 \\
& \hphantom{\left( \frac{1+iz_2}{\log M}+s\right) \left( \frac{(1+iz_1)^2}{(\log M)^2}-s^2\right) + \left( \frac{1+iz_2}{\log M}-s\right)} + O\left( \frac{1}{(\log X)^{3-\varepsilon}}\right).
\end{split}
\end{equation*}
We extend the integral and apply \eqref{eq: integrals of derivatives of H}. After simplifying, we arrive at
\begin{equation}\label{eq: Upsilon 2 evaluated}
\Upsilon_2 = 6\zeta(2) \left( -\frac{4}{\log^2 M} \int_0^1 H'(t)H''(t)\,dt + 4s^2 \int_0^1 H(t)H'(t)\,dt \right)  + O\left( \frac{1}{(\log X)^{3-\varepsilon}}\right).
\end{equation}

We next estimate $\Upsilon_3$ defined by \eqref{eq: defn of Upsilon 3}. We interchange the order of summation over $m_1,m_2$ and over $p$. From \eqref{eq: ell 1 in terms of m1 and m2}, we see for a prime $q$ and square-free $m_1$ and $m_2$ that $q|\ell_1$ if and only if $q$ divides exactly one of $m_1$ or $m_2$. If $q$ divides $m_2$ and not $m_1$, then we may relabel $m_1$ as $m_2$ and vice versa. Hence
\begin{equation*}
\Upsilon_3 = 2\sum_{2<q\leq M}\eta_2(q,s)\mathop{\sum\sum}_{\substack{m_1,m_2\leq M \\ (m_1m_2,2)=1 \\ q|m_1, \ q\nmid m_2 }}  \frac{b_{m_1}b_{m_2}}{\sqrt{m_1m_2\ell_1}}   \prod_{p|\ell_1}\left( \frac{p}{p+1}\right)\sum_{ab=\ell_1} \left( \frac{a}{b}\right)^s . 
\end{equation*}
From this, the definition \eqref{eq: defn of mollifier coeffs bm} of $b_m$, \eqref{eq: ell 1 in terms of m1 and m2}, and \eqref{eq: ab sum as an Euler product}, it follows that
\begin{equation*}
\begin{split}
\Upsilon_3 = 2\sum_{2<q\leq M}\eta_2(q,s)\mathop{\sum\sum}_{\substack{m_1,m_2\leq M \\ (m_1m_2,2)=1 \\ q|m_1, \ q\nmid m_2 }}  \frac{\mu(m_1)\mu(m_2)}{[m_1,m_2]}   \prod_{\substack{p|m_1m_2 \\ p\nmid (m_1,m_2)} }\left( \frac{p}{p+1}\right)(p^s+p^{-s}) \\
\times H\left(\frac{\log m_1}{\log M}\right)H\left(\frac{\log m_2}{\log M}\right). 
\end{split}
\end{equation*}
We relabel $m_1$ as $qm_1$ to write this as
\begin{equation*}
\begin{split}
\Upsilon_3 = -2\sum_{2<q\leq M}\frac{\eta_2(q,s)}{q+1}(q^s+q^{-s})\sum_{m_1\leq \frac{M}{q}}\sum_{\substack{ m_2\leq M \\ (m_1m_2,2q)=1}}  \frac{\mu(m_1)\mu(m_2)}{[m_1,m_2]}   \\
\times \prod_{\substack{p|m_1m_2 \\ p\nmid (m_1,m_2)} }\left( \frac{p}{p+1}\right)(p^s+p^{-s}) H\left(\frac{\log qm_1}{\log M}\right)H\left(\frac{\log m_2}{\log M}\right). 
\end{split}
\end{equation*}
We insert the Fourier inversion formula \eqref{eq: write H as Fourier integral}, interchange the order of summation, and then write the $m_1,m_2$-sum as an Euler product to deduce that
\begin{align*}
\Upsilon_3=-2\int_{-\infty}^{\infty}\int_{-\infty}^{\infty}\sum_{2<q\leq M}\frac{\eta_2(q,s)(q^s+q^{-s})}{(q+1)q^{\frac{1+iz_1}{\log M}}}   h(z_1)h(z_2)\prod_{p\nmid 2q}\Bigg(1-\frac{1}{p^{1+\frac{1+iz_1}{\log M}}}(p^s+p^{-s})\left( \frac{p}{p+1}\right)  \\
 -\frac{1}{p^{1+\frac{1+iz_2}{\log M}}}(p^s+p^{-s})\left( \frac{p}{p+1}\right) + \frac{1}{p^{1+\frac{2+iz_1+iz_2}{\log M}}}\Bigg)   dz_1dz_2.
\end{align*}
We may express the Euler product in terms of zeta-functions to write
\begin{equation}\label{eq: Upsilon 3 with Euler product}
\begin{split}
\Upsilon_3=-2\int_{-\infty}^{\infty}\int_{-\infty}^{\infty}\sum_{2<q\leq M}\frac{\eta_2(q,s)(q^s+q^{-s})}{(q+1)q^{\frac{1+iz_1}{\log M}}}   h(z_1)h(z_2) \zeta\left(1+\tfrac{2+iz_1+iz_2}{\log M}\right) \zeta^{-1}\left( 1+\tfrac{1+iz_1}{\log M}+s\right) \\
\times \zeta^{-1}\left( 1+\tfrac{1+iz_1}{\log M}-s\right) \zeta^{-1}\left( 1+\tfrac{1+iz_2}{\log M}+s\right) \zeta^{-1}\left( 1+\tfrac{1+iz_2}{\log M}-s\right)U_q(s,z_1,z_2,\tfrac{1}{\log M}) dz_1dz_2,
\end{split}
\end{equation}
where $U_q(s,z_1,z_2,\tfrac{1}{\log M})$ is an Euler product that is uniformly bounded for $2<q\leq M$ prime, $|s|\leq \varepsilon$, and real $z_1,z_2$. Using \eqref{eq: h has rapid decay}, we may truncate the integrals to the range $|z_1|,|z_2|\leq \sqrt{\log M}$ and introduce only a negligible error. In this range, and for $|s|=\frac{1}{\log X}$, the quotient of zeta-functions in \eqref{eq: Upsilon 3 with Euler product} is $(\log M)^{-3+\varepsilon}$. Moreover, \eqref{eq: defn of eta 2} implies $\eta_2(q,s)\ll by q^{-1+\varepsilon}$ for $2<q\leq M$ and $|s|=\frac{1}{\log X}$. It thus follows that
\begin{equation}\label{eq: Upsilon 3 bounded}
\Upsilon_3 \ll \frac{1}{(\log X)^{3-\varepsilon}}.
\end{equation}
A similar argument applies to $\Upsilon_4$ defined by \eqref{eq: defn of Upsilon 4}, except we use the fact that, for a prime $q$, $q|m_1m_2$ and $q\nmid \ell_1$ both hold if and only if $q$ divides both $m_1$ and $m_2$, by \eqref{eq: ell 1 in terms of m1 and m2}. This leads to
\begin{equation}\label{eq: Upsilon 4 bounded}
\Upsilon_4 \ll \frac{1}{(\log X)^{3-\varepsilon}}.
\end{equation}

It now follows from \eqref{eq: B after evaluation of Sigma 1 and 2}, the definitions \eqref{eq: defn of Upsilon 1} through \eqref{eq: defn of Upsilon 4} of $\Upsilon_1,\Upsilon_2,\Upsilon_3,\Upsilon_4$, and the estimates \eqref{eq: Upsilon 1 evaluated}, \eqref{eq: Upsilon 2 evaluated}, \eqref{eq: Upsilon 3 bounded}, and \eqref{eq: Upsilon 4 bounded} that
\begin{equation*}
\begin{split}
& \mathcal{B}= \frac{2X\check{\Phi}(0)}{(\sqrt{2}-1)^4} \frac{1+o(1)}{\log R} \ \frac{1}{2\pi i} \oint_{|s|=\frac{1}{\log X}} \Gamma_1(s)\frac{\Gamma \left( \frac{s}{2} + \frac{1}{4}\right)^2}{\Gamma \left( \frac{1}{4}\right)^2} \left(\frac{4}{\pi}\right)^s\zeta(2s)\zeta(2s+1) \\ 
& \times \left(1 - \frac{1}{2^{\frac{1}{2}+s}} \right)\left(1 - \frac{1}{2^{\frac{1}{2}-s}} \right)\left(\frac{5}{2} - 4^s-4^{-s}   \right)    \Bigg\{ \log \left( \frac{X}{2}\right)  +2\gamma+ \frac{(\check{\Phi})'(0)}{\check{\Phi}(0)} -\frac{\Gamma_1'(s)}{\Gamma_1(s)} - 2\frac{\zeta'}{\zeta}(2s) \\
&   + 2 \frac{\zeta'}{\zeta}(2s+1)   +\frac{\log 2}{\left(\sqrt{2}+2^s\right)\left(\sqrt{2}+2^{-s}\right)} +\sum_{p\neq 2} \eta_1(p,s) \Bigg\} \Bigg\{ \frac{1}{\log^3 M} \int_0^1 H''(t)^2\,dt  \\
& - \frac{2s^2}{\log M}\int_0^1 H(t)H''(t)\,dt  +  s^4 \log M \int_0^1 H(t)^2\,dt -\frac{4}{\log^2 M} \int_0^1 H'(t)H''(t)\,dt \\
&  + 4s^2 \int_0^1 H(t)H'(t)\,dt\Bigg\} \,\frac{ds}{s}+ O\left( \frac{X}{(\log X)^{2-\varepsilon}}\right).
\end{split}
\end{equation*}
Evaluating the $s$-integral as a residue, we deduce that
\begin{equation*}
\begin{split}
\mathcal{B} = \frac{X\check{\Phi}(0)}{4\left(1-\frac{1}{\sqrt{2}}\right)^2}\frac{1+o(1)}{\log R}\Bigg\{ \frac{\log X}{2\log M} \int_0^1 H(t)H''(t)\,dt -\int_0^1 H(t)H'(t)\,dt \Bigg\} \\ 
 +  O\left(X(\log X)^{-2+\varepsilon}\right).
\end{split} 
\end{equation*}
From this, \eqref{eq: T0 after mollifier}, \eqref{S+SMSR}, \eqref{eq:bound on S_R^+}, and \eqref{SM3}, it now follows that
\begin{equation*}
\begin{split}
S^+ = 
& \frac{X}{8\left(1-\frac{1}{\sqrt{2}}\right)^2} \frac{1+o(1)}{\log R} \Bigg\{ \frac{1}{24}\left(\frac{\log X}{\log M}\right)^3 \int_0^1 H''(t)^2\,dt  \\
& -  \frac{1}{2}\left(\frac{\log X}{\log M}\right)^2 \int_0^1 H'(t) H''(t)\,dt + \frac{\log X}{\log M} \int_0^1 H(t) H''(t)\,dt + \frac{\log X}{\log M} \int_0^1 H'(t)^2 \,dt  \\
& - \ \ 2\int_0^1 H(t)H'(t) \,dt  \Bigg\} + O\left( \frac{X}{(\log X)^{2-\varepsilon}}+\frac{X^{1+\varepsilon}}{Y}+ X^{\frac{1}{2}+\varepsilon}M\right).
\end{split}
\end{equation*}
The error terms are acceptable by the choices in Subsection \ref{subsection: conditions for the parameters}, and this yields Proposition \ref{prop:upper bound for S2}. 

\section{Choosing the mollifier: finishing the proof of Theorem \ref{thm: pos prop}}\label{sec:finish proof of pos prop thm}

In this section we complete the proof of Theorem \ref{thm: pos prop} by making an optimal choice for the smooth function $H(x)$ (see \eqref{eq: defn of mollifier},\eqref{eq: defn of mollifier coeffs bm}).

By \eqref{eq: Cauchy Schwarz first second moment inequality}, Proposition \ref{prop: asymptotic for S1}, and Proposition \ref{prop:upper bound for S2}, one derives the inequality
\begin{align}\label{eq:first pos prop inequality}
\sum_{\substack{p \equiv 1 \, (\text{mod }8) \\ L(\frac{1}{2},\chi_p) \neq 0}} (\log p)\Phi \left( \frac{p}{X}\right) &\geq \frac{X}{(1+\delta_0)8} \cdot \vartheta\frac{\left(H(0) - \frac{1}{2\theta}H'(0)\right)^2}{\mathfrak{I}},
\end{align}
where $\delta_0 > 0$ is sufficiently small and fixed. We also have the upper bound
\begin{align*}
\sum_{\substack{p \equiv 1 \, (\text{mod }8) \\ L(\frac{1}{2},\chi_p) \neq 0}} (\log p)\Phi \left( \frac{p}{X}\right) \leq (\log X) \sum_{\substack{X/2 < p \leq X \\ p \equiv 1 \, (\text{mod }8) \\ L(\frac{1}{2},\chi_p) \neq 0}} 1.
\end{align*}
The right side of \eqref{eq:first pos prop inequality} is an increasing function of $\vartheta$, and so $\vartheta$ should be as large as possible. The hypotheses of Proposition \ref{prop:upper bound for S2} allow $\vartheta = \frac{1}{2}(\frac{1}{2} - \theta) - \varepsilon$, and therefore
\begin{align}\label{eq:percent lower bound with cal Z}
\sum_{\substack{X/2 < p \leq X \\ p \equiv 1 \, (\text{mod }8) \\ L(\frac{1}{2},\chi_p) \neq 0}} 1 \ &\geq \frac{X}{(1+2\delta_0)8 \log X} \cdot \varrho,
\end{align}
where
\begin{align*}
\varrho &:= \frac{1}{2}\left(\frac{1}{2}-\theta \right)\frac{\left(H(0) - \frac{1}{2\theta}H'(0)\right)^2}{\mathfrak{I}}.
\end{align*}
We seek a choice of $H(x)$ which maximizes $\varrho$.

As $H(x)$ is a smooth function supported in $[-1,1]$, we have $H(1) = H'(1) = 0$. For notational simplicity we set $H(0) = A, -H'(0) = B$. Since
\begin{align*}
\int_0^1 H(x)H'(x)dx &= -\frac{1}{2}A^2, \\
\int_0^1 H(x)H''(x)dx &= AB - \int_0^1 H'(x)^2 dx, \\
\int_0^1 H'(x)H''(x) dx &= -\frac{1}{2}B^2,
\end{align*}
we have
\begin{align*}
\mathfrak{I} &= \left(A + \frac{1}{2\theta}B\right)^2 + \frac{1}{24\theta^3}\int_0^1 H''(x)^2 dx.
\end{align*}
We choose $H(x)$ such that on $[0,1]$ it is a smooth approximation to the optimal function $H_*(x)$ which minimizes the integral
\begin{align}\label{eq:H star integral}
\int_0^1 H_*''(x)^2 dx
\end{align}
among all $H_1 \in \mathcal{C}^3([0,1])$ satisfying the boundary conditions $H_1(0) = A, -H_1'(0) = B, H_1(1) = H_1'(1) = 0$. We may choose $H(x)$ such that
\begin{align*}
(1+\delta_0)\int_0^1 H_*''(x)^2 dx \geq \int_0^1 H''(x)^2 dx.
\end{align*}

By the Euler-Lagrange equation, we find that an $H_*(x)$ which minimizes \eqref{eq:H star integral} must satisfy
\begin{align*}
H_*^{(4)}(x) = 0.
\end{align*}
Thus, $H_*(x)$ is a polynomial of degree at most three. Recalling the boundary conditions, we find
\begin{align*}
H_*(x) = (2A-B)x^3+(2B-3A)x^2-Bx+A.
\end{align*}
By direct computation we obtain
\begin{align*}
\int_0^1 H_*''(x)^2 dx = 3A^2 + (2B-3A)^2,
\end{align*}
and therefore
\begin{align*}
\varrho &\geq \frac{1-O(\delta_0)}{2} \left(\frac{1}{2} - \theta\right) \left(1 + \frac{3A^2 + (2B-3A)^2}{24\theta^3 (A + \frac{1}{2\theta}B)^2} \right)^{-1}.
\end{align*}
It is now a straightforward, but tedious, calculus exercise to find that
\begin{align*}
A = \frac{B(4\theta+3)}{6(\theta+1)}
\end{align*}
is an optimal choice. Thus
\begin{align}\label{eq:lower bound on cal Z}
\varrho \geq \frac{1-O(\delta_0)}{2}\left(\frac{1}{2} - \theta\right) \frac{2\theta (3+6\theta+4\theta^2)}{(1+2\theta)^3}.
\end{align}
With this choice of $A$ we have
\begin{align*}
H_*(x) = \frac{2B \theta}{6(\theta+1)}(1-x)^2 \left(2 + \frac{3}{2\theta}+x \right).
\end{align*}
Since $\varrho$ is invariant under multiplication of $H$ by scalars, we arrive at the convenient expression
\begin{align}\label{eq:best choice of H star}
H_*(x) = (1-x)^2\left(2 + \frac{3}{2\theta}+x \right).
\end{align}
If we set $x = \frac{\log m}{\log M}$ in \eqref{eq:best choice of H star}, we obtain that the mollifier coefficients $b_m$ satisfy
\begin{align*}
b_m \approx \mu(m)\frac{\log^2(M/m)}{\log^2 M}\frac{\log(X^{3/2}M^2m)}{\log M}.
\end{align*}
One might wish to compare this with the description of $\lambda(\ell)$ in \cite[p. 449]{Sou00}.

Define
\begin{align*}
\rho(\theta) := \frac{1}{2}\left(\frac{1}{2} - \theta\right) \frac{2\theta (3+6\theta+4\theta^2)}{(1+2\theta)^3} = \frac{1}{2} \left(\frac{1}{2} - \theta \right)\left(1 - \frac{1}{(1+2\theta)^3} \right).
\end{align*}
By \eqref{eq:percent lower bound with cal Z} and \eqref{eq:lower bound on cal Z}, we obtain
\begin{align}\label{eq:percent lower bound in terms of Z theta}
\sum_{\substack{X/2 < p \leq X \\ p \equiv 1 \, (\text{mod }8) \\ L(\frac{1}{2},\chi_p) \neq 0}} 1 \ &\geq \frac{X}{(1+O(\delta_0))8 \log X} \cdot \rho(\theta).
\end{align}
The maximum of $\rho(\theta)$ on $(0,\frac{1}{2})$ occurs at the unique positive root $\theta_0$ of the polynomial $16\theta^4 + 32\theta^3+24\theta^2+12\theta-3.$ By numerical calculation we find
\begin{align*}
\theta_0 = 0.17409\ldots
\end{align*}
and
\begin{align}\label{eq:val of Z theta at special theta}
\rho(\theta_0) = 0.09645\ldots.
\end{align}
We then choose $\theta = \theta_0$. Since
\begin{align*}
\sum_{\substack{X/2 < p \leq X \\ p \equiv 1 \, (\text{mod }8)}} 1 = (1+o(1))\frac{X}{8(\log X)},
\end{align*}
we deduce Theorem \ref{thm: pos prop} from \eqref{eq:percent lower bound in terms of Z theta} and \eqref{eq:val of Z theta at special theta} upon summing over dyadic intervals.

\section{The second moment of $L(\frac{1}{2},\chi_p)$}\label{sec:second moment theorems}

In this section we prove Theorems \ref{thm: 2nd moment upper bound} and \ref{thm: 2nd moment GRH}. We first consider separately the upper and lower bounds for Theorem \ref{thm: 2nd moment upper bound}.

\subsection{The upper bound in Theorem \ref{thm: 2nd moment upper bound}}

We define
\begin{align}\label{eq:defn of second moment}
M_2 := \sum_{p \equiv 1 \, (\text{mod }8)} (\log p) \Phi \left( \frac{p}{X}\right) L \left( \tfrac{1}{2},\chi_p \right)^2.
\end{align}
In this subsection we prove
\begin{align}\label{eq:upper bound for second moment}
M_2 &\leq (4\mathfrak{c} + o(1)) \frac{X}{8} (\log X)^3.
\end{align}
The upper bound of Theorem \ref{thm: 2nd moment upper bound} then follows from \eqref{eq:upper bound for second moment} upon summation over dyadic intervals.

The proof of \eqref{eq:upper bound for second moment} follows the lines of the proof of Proposition \ref{prop:upper bound for S2}, taking $M(p) = 1$. We employ positivity to replace $\log p$ by $\log X$ and then introduce an upper bound sieve. After applying the approximate functional equation we split $\mu^2(n) = N_Y(n) + R_Y(n)$, and employ the bound $\eqref{eq:bound on S_R^+}$.

We follow the argument of Section \ref{sec:mollified second moment} down to \eqref{SM3}, obtaining
\begin{align*}
S_N^+ = \mathcal{T}_0 + \mathcal{B}.
\end{align*}
Since we have no mollifier here, we find
\begin{equation*}
\begin{split}
\mathcal{T}_0 =  \frac{2X}{(\sqrt{2}-1)^4} \frac{1+o(1)}{\log R}  \sum_{\substack{ \nu=1 \\ (\nu,2)=1 \\ \nu=\square }}^{\infty} \frac{d_2(\nu) }{\sqrt{\nu}} \hat{F}_{\nu}(0)  + O\left( \frac{X}{(\log R)^{2018}}\right)   + O\left( \frac{X^{1+\varepsilon}}{Y}\right).
\end{split}
\end{equation*}
We insert into this the definitions \eqref{Fdef} and \eqref{eq: defn of omega j} of $F_{\nu}$ and $\omega_2$, interchange the order of summation, and then write the sum on $\nu$ as an Euler product. The result is
\begin{equation*}
\begin{split}
& \mathcal{T}_0 =  \frac{2X}{(\sqrt{2}-1)^4} \frac{1+o(1)}{\log R}  \frac{1}{2\pi i} \int_{(c)} \frac{\Gamma(\frac{s}{2}+\frac{1}{4})^2}{\Gamma(\frac{1}{4})^2}\left(1-\frac{1}{2^{\frac{1}{2}-s}}\right)^2\left( \frac{X}{ \pi} \right)^{s} \check{\Phi}(s) \left( 1-\frac{1}{2^{1+2s}}\right)^3 \\
&\times \zeta(1+2s)^3\left( 1-\frac{1}{2^{2+4s}}\right)^{-1} \zeta(2+4s)^{-1} \,\frac{ds}{s} + O\left( \frac{X}{(\log R)^{2018}}+ \frac{X^{1+\varepsilon}}{Y}\right).
\end{split}
\end{equation*}
As before, we truncate the integral to the range $|\text{Im}(s)|\leq (\log X)^2$, and then deform the path of integration to the path made up of the line segments $L_1,L_2,L_3$ defined above \eqref{eq: zeta bound in zero free region} to see that the main contribution arises from the residue of the integrand at $s=0$. We evaluate the residue using \eqref{eq: residue as derivative} and arrive at
\begin{equation*}
\begin{split}
\mathcal{T}_0 = \left(144\zeta(2)\left(1-\frac{1}{\sqrt{2}}\right)^2 \right)^{-1}\frac{X\check{\Phi}(0)}{4} \frac{1+o(1)}{\log R} (\log X)^3+ O\left( X\log X+ \frac{X^{1+\varepsilon}}{Y}\right).
\end{split}
\end{equation*}
Recalling the definition of $\mathfrak{c}$, we have
\begin{align}\label{eq:T0 simplified further}
\mathcal{T}_0 &\leq (\mathfrak{c}+\varepsilon) \frac{X}{8} \frac{(\log X)^3}{\log R} + O \left(X \log X + \frac{X^{1+\varepsilon}}{Y} \right).
\end{align}
Moreover, we see from \eqref{eq: B after evaluation of Sigma 1 and 2} that if $M=1$ and $b_1=1$, then
\begin{equation}\label{eq: B specialized with M equals 1}
\mathcal{B} \ll X \frac{\log X}{\log R} \ll X
\end{equation}
since we may deform the path of integration in \eqref{eq: B after evaluation of Sigma 1 and 2} to a circle $|s|=\varepsilon$. The condition $\theta+2\vartheta<\frac{1}{2}$ in Subsection~\ref{subsection: conditions for the parameters} with $\theta=0$ allows us to take $\vartheta=\frac{1}{4}-\varepsilon$ in \eqref{eq:T0 simplified further}. We then set $Y = X^\delta$, for some small, fixed $\delta > 0$. We see that the upper bound \eqref{eq:upper bound for second moment} then follows from \eqref{eq:T0 simplified further} and \eqref{eq: B specialized with M equals 1} after sending $\varepsilon$ to zero sufficiently slowly.

\subsection{The lower bound in Theorem \ref{thm: 2nd moment upper bound}}

Recall the definition \eqref{eq:defn of second moment} of $M_2$. Our goal is to prove the following result.
\begin{prop}\label{prop: 2nd moment lower bound}
For large $X$ we have
\begin{align*}
M_2 &\geq \frac{1}{2}(\mathfrak{c} - o(1)) \frac{X}{4} (\log X)^3,
\end{align*}
where $\mathfrak{c}$ is the positive constant defined in Theorem \ref{thm: 2nd moment upper bound}, and $o(1)$ is some quantity that goes to zero as $X \rightarrow \infty$.
\end{prop}
The lower bound for Theorem \ref{thm: 2nd moment upper bound} easily follows from Proposition \ref{prop: 2nd moment lower bound} by summing over dyadic intervals.

The main idea in the proof of Proposition \ref{prop: 2nd moment lower bound} is a standard one. For any Dirichlet polynomial $A(p)$, the Cauchy-Schwarz inequality implies
\begin{align}\label{eq:lower bound 2nd moment cauchy}
M_2 &\geq \frac{\left(\sum_{p \equiv 1 \, (\text{mod }8)} (\log p) \Phi \left( \frac{p}{X}\right) L \left( \frac{1}{2},\chi_p\right) A(p) \right)^2}{\sum_{p \equiv 1 \, (\text{mod }8)} (\log p) \Phi \left( \frac{p}{X}\right) A(p)^2}.
\end{align}
Clearly, we should choose $A(p)$ to be an approximation to $L(\frac{1}{2},\chi_p)$. Our choice is inspired by the approximate functional equation in Lemma \ref{lem: approx func eq}. For a positive real number $\alpha$, we define
\begin{align}\label{eq:lower bound defn of A alpha}
A_{\alpha}(p) := \frac{2}{\left(1 - \frac{1}{\sqrt{2}}\right)^2} \sum_{n \text{ odd}} \frac{\chi_p(n)}{\sqrt{n}} \omega_1\left(n \sqrt{\frac{\pi}{p^{\alpha}}} \right).
\end{align}
With $\varepsilon_0 > 0$ small and fixed, we then choose $A(p)$ in \eqref{eq:lower bound 2nd moment cauchy} to be
\begin{align}\label{eq:lower bound choice for A}
A(p) := A_{1-\varepsilon_0}(p).
\end{align}
Observe that taking $\alpha=1$ in \eqref{eq:lower bound defn of A alpha} yields
\begin{align}\label{eq:lower bound L is A 1}
A_1(p) = L\left(\tfrac{1}{2},\chi_p\right).
\end{align}

\begin{prop}\label{prop:2nd moment dir poly A}
Let $\varepsilon_0 > 0$ be small. Let $\alpha_1 \leq \alpha_2$ be real numbers with $\alpha_1,\alpha_2 \in \{1-\varepsilon_0,1\}$, and $(\alpha_1,\alpha_2) \neq (1,1)$. Then
\begin{align*}
M_{\alpha_1,\alpha_2}:= \sum_{p \equiv 1 \, (\textup{mod }8)} (\log p) \Phi \left( \frac{p}{X}\right)A_{\alpha_1}(p)A_{\alpha_2}(p) = \frac{1}{2}(\mathfrak{c} + O(\varepsilon_0)) \frac{X}{4} (\log X)^3.
\end{align*}
\end{prop}

\begin{proof}[Proof of Proposition \ref{prop: 2nd moment lower bound} assuming Proposition \ref{prop:2nd moment dir poly A}] 

By \eqref{eq:lower bound 2nd moment cauchy}, \eqref{eq:lower bound choice for A}, and \eqref{eq:lower bound L is A 1}, we have
\begin{align*}
M_2 &\geq \frac{M_{1-\varepsilon_0,1}^2}{M_{1-\varepsilon_0,1-\varepsilon_0}}.
\end{align*}
We apply Proposition \ref{prop:2nd moment dir poly A} to obtain
\begin{align*}
M_2 &\geq \frac{1}{2}(\mathfrak{c} + O(\varepsilon_0))\frac{X}{4}(\log X)^3.
\end{align*}
The proposition follows upon letting $\varepsilon_0 = \varepsilon_0(X)$ go to zero sufficiently slowly as $X \rightarrow \infty$.
\end{proof}

We devote the rest of this subsection to the proof of Proposition \ref{prop:2nd moment dir poly A}.

\begin{proof}[Proof of Proposition \ref{prop:2nd moment dir poly A}]
By definition,
\begin{align*}
M_{\alpha_1,\alpha_2} &= \frac{4}{(1 - \frac{1}{\sqrt{2}})^4} \sum_{p \equiv 1 \, (\text{mod }8)} (\log p) \Phi \left( \frac{p}{X}\right) \mathop{\sum \sum}_{m,n \text{ odd}} \frac{\chi_p(mn)}{\sqrt{mn}} \omega_1\left(m \sqrt{\frac{\pi}{p^{\alpha_1}}} \right)\omega_1\left(n \sqrt{\frac{\pi}{p^{\alpha_2}}} \right).
\end{align*}
Let $M_{\neq}$ denote the contribution to $M_{\alpha_1,\alpha_2}$ from $mn \neq \square$. An application of Lemma \ref{lem:gen char sum over primes} shows that $M_{\neq} \ll X$, say. We note that for bounding $M_{\neq}$ it is crucial that $\alpha_1 = 1-\varepsilon_0$.

We therefore have
\begin{align*}
M_{\alpha_1,\alpha_2} = &\frac{4}{(1 - \frac{1}{\sqrt{2}})^4} \sum_{p \equiv 1 \, (\text{mod }8)} (\log p) \Phi \left( \frac{p}{X}\right) \mathop{\sum \sum}_{\substack{(mn,2p)=1 \\ mn = \square}} \frac{1}{\sqrt{mn}} \omega_1\left(m \sqrt{\frac{\pi}{p^{\alpha_1}}} \right)\omega_1\left(n \sqrt{\frac{\pi}{p^{\alpha_2}}} \right) \\
&+O(X).
\end{align*}
We use Lemma \ref{lem: properties of omega j} to remove the condition $(mn,p) = 1$ at the cost of a negligible error. We then open $\omega_1$ using its definition as an integral, and interchange the order of summation and integration. After some simplification we arrive at
\begin{align*}
M_{\alpha_1,\alpha_2} = \frac{4}{(1-\frac{1}{\sqrt{2}})^4} \frac{1}{(2\pi i)^2} &\int_{(c_1)}\int_{(c_2)} K(s_1,s_2) \left(\frac{X^{\alpha_1}}{\pi} \right)^{\frac{s_1}{2}}\left(\frac{X^{\alpha_2}}{\pi} \right)^{\frac{s_2}{2}} \zeta(1+2s_1)\zeta(1+2s_2) \\
&\times \zeta(1+s_1+s_2) \left(\sum_{p \equiv 1 \, (\text{mod }8)} (\log p) \Phi \left( \frac{p}{X}\right) \left(\frac{p}{X} \right)^{\frac{\alpha_1s_1 + \alpha_2s_2}{2}} \right) \frac{ds_1 ds_2}{s_1 s_2} \\
&+O(X),
\end{align*}
where $c_\ell = \text{Re}(s_\ell)$ is a positive real number, and
\begin{align*}
K(s_1,s_2) &= \zeta^{-1}(2+2s_1+2s_2) \left(1 + \frac{1}{2^{1+s_1+s_2}}\right)^{-1} \\ 
&\times\prod_{\ell=1}^2 \frac{\Gamma \left( \frac{s_\ell}{2}+\frac{1}{4}\right)}{\Gamma \left( \frac{1}{4}\right)} \left(1 - \frac{1}{2^{\frac{1}{2}-s_\ell}}\right) \left(1 - \frac{1}{2^{1+2s_\ell}}\right).
\end{align*}
For the moment we choose $c_1 = c_2 = \frac{1}{\log X}$. By the rapid decay of $K(s_1,s_2)$ in vertical strips, we may truncate to $|\text{Im}(s_\ell)| \leq (\log X)^2$ at the cost of a negligible error. With this condition in place, we use the prime number theorem in arithmetic progressions to obtain that the sum on $p$ is
\begin{align*}
\frac{X}{4}\int_0^\infty \Phi(x) x^{\frac{\alpha_1s_1+\alpha_2s_2}{2}}dx + O \left(X \exp(-c\sqrt{\log X}) \right).
\end{align*}
The error term clearly makes an acceptable contribution to $M_{\alpha_1,\alpha_2}$. We then remove the condition on $\text{Im}(s_\ell)$ by the same means we installed it and obtain
\begin{align*}
M_{\alpha_1,\alpha_2} = \frac{4}{(1-\frac{1}{\sqrt{2}})^4} \frac{X}{4}\int_0^\infty \Phi(x) \frac{1}{(2\pi i)^2} &\int_{(c_1)}\int_{(c_2)} K(s_1,s_2) \left(\frac{(xX)^{\alpha_1}}{\pi} \right)^{\frac{s_1}{2}}\left(\frac{(xX)^{\alpha_2}}{\pi} \right)^{\frac{s_2}{2}} \\ 
&\times \zeta(1+2s_1)\zeta(1+2s_2)\zeta(1+s_2+s_2) \frac{ds_1 ds_2}{s_1s_2} dx \\
&+O(X).
\end{align*}

We wish to separate the variables $s_1$ and $s_2$. Since $c_\ell > 0$ we expand $\zeta(1+s_1+s_2)$ as an absolutely convergent Dirichlet series. Interchanging the order of summation and integration, we obtain
\begin{align*}
M_{\alpha_1,\alpha_2} = \frac{4}{(1-\frac{1}{\sqrt{2}})^4} &\frac{X}{4}\int_0^\infty \Phi(x) \sum_{n=1}^\infty \frac{1}{n}\frac{1}{(2\pi i)^2} \int_{(c_1)}\int_{(c_2)} K(s_1,s_2) \left(\frac{(xX)^{\alpha_1}}{\pi n^{2}} \right)^{\frac{s_1}{2}}\left(\frac{(xX)^{\alpha_2}}{\pi n^{2}} \right)^{\frac{s_2}{2}} \\ 
&\times \zeta(1+2s_1)\zeta(1+2s_2) \frac{ds_1 ds_2}{s_1s_2} dx +O(X).
\end{align*}
To truncate the summation over $n$, first we move the contours of integration to the right to $c_1 = c_2 = 1$. By trivial estimation we deduce that the contribution from $n \gg X^{\frac{\alpha_1 + \alpha_2}{4}}$ is $O(X)$. For $n$ in the range $X^{\frac{\alpha_1}{2}} \ll n \ll X^{\frac{\alpha_1 + \alpha_2}{4}}$, we move $\text{Re}(s_2)$ to $c_2 = \frac{1}{\log X}$ and estimate trivially, getting an error term of $O(X (\log X)^2)$. With $n \ll X^{\frac{\alpha_1}{2}}$ we then move $c_1$ to $\frac{1}{\log X}$, obtaining
\begin{align*}
M_{\alpha_1,\alpha_2} = \frac{4}{(1-\frac{1}{\sqrt{2}})^4} \frac{X}{4}&\int_0^\infty \Phi(x) \sum_{n \leq \sqrt{(xX)^{\alpha_1}/\pi}} \frac{1}{n} \\
&\times \frac{1}{(2\pi i)^2} \int_{(\frac{1}{\log X})}\int_{(\frac{1}{\log X})} K(s_1,s_2) \left(\frac{(xX)^{\alpha_1}}{\pi n^{2}} \right)^{\frac{s_1}{2}}\left(\frac{(xX)^{\alpha_2}}{\pi n^{2}} \right)^{\frac{s_2}{2}} \\ 
&\times \zeta(1+2s_1)\zeta(1+2s_2) \frac{ds_1 ds_2}{s_1s_2} dx + O(X (\log X)^{2}).
\end{align*}
The variables $s_1$ and $s_2$ are almost separated, except they are entangled inside of $K(s_1,s_2)$. We move the lines of integration to $\text{Re}(s_1) = \text{Re}(s_2) = -\delta$, for some small, fixed $\delta > 0$. In doing so we pick up contributions from the poles at $s_1,s_2 = 0$. The contribution from the integrals on $\text{Re}(s_\ell) = -\delta$ is trivially bounded by $O(X\log X)$. We write the contributions from the poles at $s_\ell = 0$ as contour integrals around small circles, thereby obtaining
\begin{align*}
M_{\alpha_1,\alpha_2} = \frac{4}{(1-\frac{1}{\sqrt{2}})^4} &\frac{X}{4}\int_0^\infty \Phi(x) \sum_{n \leq \sqrt{(xX)^{\alpha_1}/\pi}} \frac{1}{n} \\ 
&\times\frac{1}{(2\pi i)^2} \mathop{\oint \oint}_{|s_\ell| = (\log X)^{-1}} K(s_1,s_2) \left(\frac{(xX)^{\alpha_1}}{\pi n^{2}} \right)^{\frac{s_1}{2}}\left(\frac{(xX)^{\alpha_2}}{\pi n^{2}} \right)^{\frac{s_2}{2}} \\ 
&\times \zeta(1+2s_1)\zeta(1+2s_2) \frac{ds_1 ds_2}{s_1s_2} dx + O(X (\log X)^{2}).
\end{align*}
Since $|s_\ell|=(\log X)^{-1}$ we have
\begin{align*}
K(s_1,s_2) = K(0,0) + O \left( \frac{1}{\log X}\right) = \frac{1}{6\zeta(2)}\left(1 - \tfrac{1}{\sqrt{2}}\right)^2 + O\left( \frac{1}{\log X}\right),
\end{align*}
and therefore
\begin{align*}
M_{\alpha_1,\alpha_2} = \frac{2}{3\zeta(2)(1 - \frac{1}{\sqrt{2}})^2}\frac{X}{4}\int_0^\infty \Phi(x) &\sum_{n \leq \sqrt{(xX)^{\alpha_1}/\pi}} \frac{1}{n}\frac{1}{(2\pi i)^2} \mathop{\oint \oint}_{|s_\ell| = (\log X)^{-1}} \left(\frac{(xX)^{\alpha_1}}{\pi n^{2}} \right)^{\frac{s_1}{2}}\left(\frac{(xX)^{\alpha_2}}{\pi n^{2}} \right)^{\frac{s_2}{2}} \\ 
&\times \zeta(1+2s_1)\zeta(1+2s_2) \frac{ds_1 ds_2}{s_1s_2} dx + O(X (\log X)^{2}).
\end{align*}
Expanding in Laurent and power series yields
\begin{align*}
\frac{1}{2\pi i}\oint_{|s_\ell|=(\log X)^{-1}} \left(\frac{(xX)^{\alpha_\ell}}{\pi n^{2}} \right)^{\frac{s_\ell}{2}} \zeta(1+2s_\ell) \frac{ds_\ell}{s_\ell} = \frac{1}{2}\log \left(\frac{1}{n}\sqrt{\frac{(xX)^{\alpha_\ell}}{\pi}}  \right) + O(1),
\end{align*}
and hence
\begin{align*}
M_{\alpha_1,\alpha_2} &= \frac{1}{6\zeta(2)(1-\frac{1}{\sqrt{2}})^2} \frac{X}{4}\int_0^\infty \Phi(x) \sum_{n \leq \sqrt{(xX)^{\alpha_1}/\pi}} \frac{1}{n}\log \left(\frac{1}{n}\sqrt{\frac{(xX)^{\alpha_1}}{\pi}}  \right)\log \left(\frac{1}{n}\sqrt{\frac{(xX)^{\alpha_2}}{\pi}}  \right) \\
&+O(X(\log X)^{2}).
\end{align*}
Partial summation yields
\begin{align*}
\sum_{n \leq \sqrt{(xX)^{\alpha_1}/\pi}} \frac{1}{n}\log \left(\frac{1}{n}\sqrt{\frac{(xX)^{\alpha_1}}{\pi}}  \right)\log \left(\frac{1}{n}\sqrt{\frac{(xX)^{\alpha_2}}{\pi}}  \right) = \frac{1+O(\varepsilon_0)}{24}(\log X)^3,
\end{align*}
and therefore
\begin{align*}
M_{\alpha_1,\alpha_2} = \frac{1}{2}(\mathfrak{c} + O(\varepsilon_0)) \frac{X}{4}(\log X)^3.
\end{align*}
\end{proof}

\subsection{Proof of Theorem \ref{thm: 2nd moment GRH}}

We turn now to the proof of Theorem \ref{thm: 2nd moment GRH}. Throughout this subsection we set $\eta := 100 \log \log X/\log X$. Recalling the definition \eqref{eq:lower bound defn of A alpha} of $A_{\alpha}(p)$, we then have
\begin{align*}
L\left( \tfrac{1}{2},\chi_p \right) = A_{1-\eta}(p) + B(p),
\end{align*}
say. Thus
\begin{align}\label{eq:GRH second moment defn}
M_2 &= \sum_{p \equiv 1 \, (\text{mod }8)} (\log p) \Phi \left( \frac{p}{X}\right) \left\{ A_{1-\eta}(p)^2 + O(|A_{1-\eta}(p)B(p)| +|B(p)|^2)  \right\}.
\end{align}
We prove, on GRH, that
\begin{align}\label{eq:GRH main term}
\sum_{p \equiv 1 \, (\text{mod }8)} (\log p) \Phi \left( \frac{p}{X}\right) A_{1-\eta}(p)^2 = \mathfrak{c} \frac{X}{8}(\log X)^3 + O(X(\log X)^{2+\varepsilon})
\end{align}
and
\begin{align}\label{eq:GRH error term}
\sum_{p \equiv 1 \, (\text{mod }8)} (\log p) \Phi \left( \frac{p}{X}\right) |B(p)|^2 \ll X(\log X)^{5/2}.
\end{align}
Theorem \ref{thm: 2nd moment GRH} then follows from \eqref{eq:GRH second moment defn}, \eqref{eq:GRH main term}, and \eqref{eq:GRH error term} after applying Cauchy-Schwarz and summing over dyadic ranges.

We may easily prove \eqref{eq:GRH main term}, since the treatment is substantially similar to the proof of Proposition \ref{prop:2nd moment dir poly A}. Applying the approximate functional equation, the main term of \eqref{eq:GRH main term} is
\begin{align*}
\frac{4}{(1 - \frac{1}{\sqrt{2}})^4} \sum_{p \equiv 1 \, (\text{mod }8)} (\log p) \Phi \left( \frac{p}{X}\right) \mathop{\sum \sum}_{m,n \text{ odd}} \frac{\chi_p(mn)}{\sqrt{mn}} \omega_1\left(m \sqrt{\frac{\pi}{p^{1-\eta}}} \right)\omega_1\left(n \sqrt{\frac{\pi}{p^{1-\eta}}} \right).
\end{align*}
We argue as in Proposition \ref{prop:2nd moment dir poly A} and obtain that the contribution from $mn = \square$ is
\begin{align*}
\mathfrak{c} \frac{X}{8}(\log X)^3 + O(X(\log X)^{2+\varepsilon}).
\end{align*}
The following standard result implies that the contribution to \eqref{eq:GRH main term} from $mn \neq \square$ is $O(X/\log X)$, say.
\begin{lem}\label{lem:char sum bound on GRH}
Let $\chi$ be a non-principal Dirichlet character modulo $q$. Let $\chi^*$ be the primitive character inducing $\chi$, and assume that GRH holds for $L(s,\chi^*)$. If $q \leq X^M$ for some fixed positive constant $M$, then
\begin{align*}
\sum_{p \leq X} \chi(p)(\log p) \ll_M X^{1/2} (\log X)^2.
\end{align*}
\end{lem}

The proof of \eqref{eq:GRH error term} is more subtle. Here the method of proof is that of Soundararajan and Young \cite{SY10}. As the arguments are very similar, our exposition will be sparse, and we refer the reader to \cite{SY10} for more details. We perform some initial manipulations, and then we state the main proposition which will yield \eqref{eq:GRH error term}.

By definition, we have
\begin{align}\label{eq:GRH defn of B}
B(p) &= \frac{1}{2\pi i} \int_{(c)} g(s) L\left( \frac{1}{2}+s,\chi_p\right) \frac{p^{s/2} - p^{(1-\eta)s/2}}{s}ds,
\end{align}
where $c>0$ and
\begin{align*}
g(s) = \frac{2}{(1-\frac{1}{\sqrt{2}})^2} \frac{\Gamma \left( \frac{s}{2} + \frac{1}{4}\right)}{\Gamma \left( \frac{1}{4}\right)} \left(1 - \frac{1}{2^{\frac{1}{2}-s}}\right) \left(1 - \frac{1}{2^{\frac{1}{2}+s}}\right)\pi^{-s/2}.
\end{align*}
The function $(p^{s/2}-p^{(1-\eta)s/2})/s$ is entire, so we may move the line of integration in \eqref{eq:GRH defn of B} to $\text{Re}(s) = 0$. On the line $\text{Re}(s) = 0$ we have the bound $|(p^{s/2}-p^{(1-\eta)s/2})/s| \ll \log \log X$, and hence the left side of \eqref{eq:GRH error term} is
\begin{align}\label{eq:GRH upper bound on sum of B squared}
\ll (\log X)^{1+\varepsilon} \int_\mathbb{R} \int_\mathbb{R} |g(it_1)g(it_2)| \sum_{\substack{p \leq X \\ p \equiv 1 \, (\text{mod }8)}} \left|L \left( \tfrac{1}{2}+it_1,\chi_p\right)L \left( \tfrac{1}{2}+it_2,\chi_p\right) \right| dt_1 \ dt_2.
\end{align}
To state the proposition we need, we first establish some notation, following \cite[Section 6]{SY10}. Given $x \geq 10$, say, and a complex number $z$, we define
\begin{align*}
\mathcal{L}(z,x) = 
\begin{cases}
\log \log x, \ \ \ \ \ &|z| \leq (\log x)^{-1}, \\
-\log |z|, &(\log x)^{-1} \leq |z| \leq 1, \\
0, &|z| \geq 1.
\end{cases}
\end{align*}
For complex numbers $z_1$ and $z_2$ we define
\begin{align*}
\mathcal{M}(z_1,z_2,x) = \frac{1}{2}(\mathcal{L}(z_1,x) + \mathcal{L}(z_2,x)),
\end{align*}
and
\begin{align*}
\mathcal{V}(z_1,z_2,x) = \frac{1}{2}[\mathcal{L}(2z_1,x)&+\mathcal{L}(2z_2,x) + \mathcal{L}(2\text{Re}(z_1),x)+\mathcal{L}(2\text{Re}(z_2),x)) \\
&+2\mathcal{L}(z_1+z_2,x) + 2\mathcal{L}(z_1 + \overline{z_2},x)].
\end{align*}
It is helpful to know that for the values of $z_1$ and $z_2$ we consider, we have $\log \log X \leq \mathcal{V}(z_1,z_2,X)  \leq 4 \log \log X$.

The following result, an analogue of \cite[Theorem 6.1]{SY10}, is the key input we need.
\begin{prop}\label{prop:GRH moments}
Let $X$ be large, and let $z_1$ and $z_2$ be complex numbers with $0 \leq \textup{Re}(z_i) \leq \frac{1}{\log X}$ and $|z_i| \leq X$. Assume the Riemann Hypothesis for the Riemann zeta function $\zeta(s)$ and for all Dirichlet $L$-functions $L(s,\chi_p)$ with $p \equiv 1 \pmod{8}$. Then for any $r >0$ in $\mathbb{R}$ and any $\varepsilon > 0$ we have
\begin{align*}
\sum_{\substack{p \leq X \\ p \equiv 1 \, (\textup{mod }8)}}&\left|L \left( \tfrac{1}{2}+z_1,\chi_p\right)L \left( \tfrac{1}{2}+z_2,\chi_p\right) \right|^r \\
&\ll_{r,\varepsilon} \frac{X}{(\log X)^{1-\varepsilon}} \exp \left(r \mathcal{M}(z_1,z_2,X) + \frac{r^2}{2} \mathcal{V}(z_1,z_2,X) \right).
\end{align*}
\end{prop}

\begin{proof}[Proof of \eqref{eq:GRH error term} assuming Proposition \ref{prop:GRH moments}]
Recall \eqref{eq:GRH upper bound on sum of B squared}. If $t_1$ or $t_2$ satisfies $|t_i| > X$ we use Cauchy-Schwarz, Lemma \ref{lem: moment estimates}, and the rapid decay of $g$ to get a negligible error. 

We may therefore assume that $|t_i| \leq X$. We then consider, for a parameter $0 < \alpha < 1$ at our disposal, two cases: (1) both $t_1$ and $t_2$ satisfy $|t_i| \leq (\log X)^{-\alpha}$, or (2) one of $t_1,t_2$ satisfies $|t_i| \geq (\log X)^{-\alpha}$. In case (1) we use the trivial bounds
\begin{align*}
\mathcal{M}(it_1,it_2,X) &\leq \log \log X, \\
\mathcal{V}(it_1,it_2,X) &\leq 4 \log \log X,
\end{align*}
while in case (2) we use the bounds
\begin{align*}
\mathcal{M}(it_1,it_2,X) &\leq \frac{1+\alpha}{2} \log \log X, \\
\mathcal{V}(it_1,it_2,X) &\leq \frac{7+\alpha}{2}\log \log X + O(1).
\end{align*}
Since $|g(it)| \ll (1+t^2)^{-1}$ we obtain by Proposition \ref{prop:GRH moments} that the quantity in \eqref{eq:GRH upper bound on sum of B squared} is
\begin{align*}
&\ll X(\log X)^\varepsilon \left((\log X)^{3-2\alpha} + (\log X)^{9/4+3\alpha/4} \right) =X(\log X)^{27/11+\varepsilon} \leq X(\log X)^{5/2}
\end{align*}
upon choosing $\alpha = 3/11$.
\end{proof}

To prove Proposition \ref{prop:GRH moments} we establish estimates for how often $\left|L \left( \frac{1}{2}+z_1,\chi_p\right)L \left( \frac{1}{2}+z_2,\chi_p\right) \right|$ can be large. The following is very similar to \cite[Proposition 6.2]{SY10}.

\begin{prop}\label{prop:large values}
Assume the hypotheses of Proposition \ref{prop:GRH moments}. Let $\mathcal{N}(V;z_1,z_2,X)$ denote the number of primes $p \leq X$, $p \equiv 1 \pmod{8}$, such that $\log \left|L \left( \frac{1}{2}+z_1,\chi_p\right)L \left( \frac{1}{2}+z_2,\chi_p\right) \right| \geq V + \mathcal{M}(z_1,z_2,X)$. In the range $3 \leq V\leq 4r\mathcal{V}(z_1,z_2,X)$ we have
\begin{align*}
\mathcal{N}(V;z_1,z_2,X) &\ll \frac{X}{(\log X)^{1-o_r(1)}} \exp \left(-\frac{V^2}{2\mathcal{V}(z_1,z_2,X)} \right),
\end{align*}
and for larger $V$ we have
\begin{align*}
\mathcal{N}(V;z_1,z_2,X) \ll \frac{X}{(\log X)^{1-o_r(1)}}\exp(-4rV).
\end{align*}
\end{prop}

\begin{proof}[Proof of Proposition \ref{prop:GRH moments}]
We have
\begin{align*}
\sum_{\substack{p \leq X \\ p \equiv 1 \, (\textup{mod }8)}}&\left|L \left( \tfrac{1}{2}+z_1,\chi_p\right)L \left( \tfrac{1}{2}+z_2,\chi_p\right) \right|^r \\ 
&= r \int_{-\infty}^\infty \exp(rV + r\mathcal{M}(z_1,z_2,X)) \mathcal{N}(V;z_1,z_2,X) dV.
\end{align*}
Then use Proposition \ref{prop:large values}.
\end{proof}

We use the following lemma to determine how frequently a Dirichlet polynomial can be large. We write $\log_2 X$ for $\log \log X$.

\begin{lem}\label{lem:mean val Dir poly over primes}
Let $X$ and $y$ be real numbers and $k$ a natural number with $y^k \leq X^{\frac{1}{2} - \frac{1}{\log_2 X}}$. For any complex numbers $a(q)$ we have
\begin{align*}
\sum_{\substack{p \leq X \\ p \equiv 1 \, (\textup{mod }8)}} \left|\sum_{2 < q \leq y} \frac{a(q)\chi_p(q)}{q^{\frac{1}{2}}} \right|^{2k} \ll \frac{X \log_2 X}{\log X} \frac{(2k!)}{2^k k!} \left( \sum_{q \leq y} \frac{|a(q)|^2}{q}\right)^k,
\end{align*}
where the implied constant is absolute.
\end{lem}
\begin{proof}
This result is similar to \cite[Lemma 6.3]{SY10}, so we give only a sketch. Since we are assuming GRH we could use Lemma \ref{lem:char sum bound on GRH}, but we get an unconditional result that is almost as good by appealing to sieve theory.

Since $p \equiv 1 \pmod{8}$, we have $\chi_p(q) = \chi_{q^*}(p)$, where for an odd integer $n$ we define $n^*=(-1)^\frac{n-1}{2} n$. Observe that $\chi_{q^*}$ is a primitive character with conductor $\leq 4q$. We then introduce an upper bound sieve supported on $d \leq D = X^{\frac{1}{\log_2 X}}$. With the upper bound sieve in place we drop the congruence condition modulo 8 and the condition that $p$ is a prime. Opening the square and using the P\'olya-Vinogradov inequality, the sum in question is then
\begin{align*}
&\ll \sum_{n \leq X} \left(\sum_{d \mid n} \lambda_d \right)\left|\sum_{2 < q \leq p} \frac{a(q)\chi_{q^*}(n)}{q^{\frac{1}{2}}} \right|^{2k} \\ 
&\ll \sum_{\substack{q_i \leq y \\ q_1\cdots q_{2k} = \square}} \frac{|a(q_1)\cdots a(q_{2k})|}{\sqrt{q_1\cdots q_{2k}}} \sum_{n \leq X} \left(\sum_{d \mid n} \lambda_d \right) + D\log(y^{2k})\sum_{q_1,\ldots,q_{2k} \leq y} |a(q_1)\cdots a(q_{2k})|.
\end{align*}
For the first term we obtain
\begin{align*}
\sum_{\substack{q_i \leq y \\ q_1\cdots q_{2k} = \square}} \frac{|a(q_1)\cdots a(q_{2k})|}{\sqrt{q_1\cdots q_{2k}}} \sum_{n \leq X} \left(\sum_{d \mid n} \lambda_d \right) \ll \frac{X\log_2 X}{\log X}\frac{(2k!)}{2^kk!} \left(\sum_{q \leq y} \frac{|a(q)|^2}{q} \right)^k,
\end{align*}
and for the second term we use Cauchy-Schwarz to obtain
\begin{align*}
D\log(y^{2k})\sum_{q_1,\ldots,q_{2k} \leq y} |a(q_1)\cdots a(q_{2k})| \ll X\frac{k \log X}{D}\left(\sum_{q \leq y} \frac{|a(q)|^2}{q} \right)^k.
\end{align*}
\end{proof}

\begin{proof}[Proof of Proposition \ref{prop:large values}]
Assume GRH for $L(s,\chi_p)$. A modification of the proof of the Proposition in \cite{Sou09} then yields
\begin{align*}
\log \left|L \left( \tfrac{1}{2}+z_1,\chi_p\right)L \left( \tfrac{1}{2}+z_2,\chi_p\right) \right| &\leq \text{Re}\left(\sum_{q^\ell \leq x}\frac{\chi_p(q^\ell)}{\ell q^{\ell(\frac{1}{2} + \frac{1}{\log x})}}(p^{-\ell z_1} + p^{-\ell z_2})\frac{\log(x/q^\ell)}{\log x} \right) \\
&+ 2\frac{\log X}{\log x} + O\left( \frac{1}{\log x}\right).
\end{align*}
The terms with $\ell \geq 3$ contribute $O(1)$. For $\ell = 2$ we use the Riemann hypothesis for $\zeta(s)$ (see \cite[(6.4)]{SY10}) and obtain
\begin{align*}
\frac{1}{2}\sum_{q \leq \sqrt{x}} \frac{1}{q^{1 + \frac{2}{\log x}}} (q^{-2z_1} + q^{-2z_2})\frac{\log(x/q^2)}{\log x} = \mathcal{M}(z_1,z_2,x) + O(\log \log \log X).
\end{align*}
Since $\mathcal{M}(z_1,z_2,x) \leq \mathcal{M}(z_1,z_2,X) + 2 \frac{\log X}{\log x}$, we obtain
\begin{align}\label{eq:GRH upper bound for log L}
\log \left|L \left( \tfrac{1}{2}+z_1,\chi_p\right)L \left( \tfrac{1}{2}+z_2,\chi_p\right) \right| &\leq \text{Re} \ \sum_{2 < q \leq x} \frac{\chi_p(q)}{q^{\frac{1}{2} + \frac{1}{\log x}}} (q^{-z_1} + q^{-z_2}) \\
&+\mathcal{M}(z_1,z_2,X) + 4 \frac{\log X}{\log x} + O(\log \log \log X). \nonumber
\end{align}

We put $\mathcal{V} = \mathcal{V}(z_1,z_2,X)$, and define
\begin{align*}
T = 
\begin{cases}
\frac{1}{2} \log \log \log X, \ \ \ \ \ \ &V \leq \mathcal{V}, \\
\frac{\mathcal{V}}{2V} \log \log \log X, &\mathcal{V} < V \leq \frac{1}{16} \mathcal{V} \log \log \log X, \\
8, &V > \frac{1}{16} \mathcal{V} \log \log \log X.
\end{cases}
\end{align*}
We take $x = X^{T/V}$, and $z = x^{1/\log \log X}$.

Taking $x = \log X$ in \eqref{eq:GRH upper bound for log L} and estimating trivially, we may assume $V \leq \frac{5 \log X}{\log \log X}$. In \eqref{eq:GRH upper bound for log L} we then have
\begin{align*}
\log \left|L \left( \tfrac{1}{2}+z_1,\chi_p\right)L \left( \tfrac{1}{2}+z_2,\chi_p\right) \right| &\leq S_1 + S_2 + \mathcal{M}(z_1,z_2,X) + 5\frac{V}{T},
\end{align*}
where $S_1$ is the sum on $q$ truncated to $q \leq z$, and $S_2$ is the remainder of the sum. Since $\log \left|L \left( \frac{1}{2}+z_1,\chi_p\right)L \left( \frac{1}{2}+z_2,\chi_p\right) \right| \geq V + \mathcal{M}(z_1,z_2,X)$ we have
\begin{align*}
S_2 \geq \frac{V}{T} \ \ \ \ \ \ \ \ \text{ or }  \ \ \ \ \ \ \ \ S_1 \geq V \left(1 - \frac{6}{T}\right) =: V_1.
\end{align*}
We take $k = \lfloor (\frac{1}{2} - \frac{1}{\log_4 X}) \frac{V}{T}\rfloor - 1$ in Lemma \ref{lem:mean val Dir poly over primes} and apply the usual Chebyshev-type maneuver to deduce that the number of $p \leq X$ with $S_2 \geq V/T$ is
\begin{align*}
\ll \frac{X \log_2 X}{\log X} \exp \left(-\frac{V}{4T} \log V \right).
\end{align*}

It remains to bound the number of $p$ for which $S_1$ is large. By Lemma \ref{lem:mean val Dir poly over primes}, for any $k \leq (\frac{1}{2} - \frac{1}{\log_2 X})\frac{V \log \log X }{T}$ the number of $p \leq X$ with $S_1 \geq V_1$ is
\begin{align*}
\ll \frac{X \log_2 X}{\log X} \left(\frac{2k \mathcal{V}(z_1,z_2,X) + O(\log \log\log X)}{eV_1^2} \right)^k.
\end{align*}
For $V \leq (\log \log X)^2$ we take $k= \lfloor V_1^2/2 \mathcal{V}\rfloor$, and for $V > (\log \log X)^2$ we take $k = \lfloor 10 V \rfloor$. It follows that the number of $p$ for which $S_1 \geq V_1$ is
\begin{align*}
\ll \frac{X \log_2 X}{\log X} \exp\left(-\frac{V_1^2}{2\mathcal{V}} \left(1 + O \left(\frac{\log \log \log X}{\log \log X} \right) \right) \right) + \frac{X\log_2 X}{\log X} \exp(-V \log V).
\end{align*}
\end{proof}

\section{Proof of Theorem \ref{thm: third moment}}\label{sec:third moment}

The proof of Theorem \ref{thm: third moment} breaks naturally into two parts: the lower bound, and the upper bound. The argument for the lower bound is very similar to that in \cite{RuS06}, and we therefore give only a sketch. The argument for the upper bound is similar to that in Section \ref{sec:mollified second moment}. In either case, we crucially use the assumption that the central values are non-negative.

\subsection{The lower bound}

Let $d_{1/2}(n)$ be the multiplicative function with $(d_{1/2}\star d_{1/2})(n) = 1$. For a prime $p \equiv 1 \pmod{4}$ and large $X$ define
\begin{align*}
R(p) := \sum_{n \leq X^{1/500}} \frac{d_{1/2}(n) \chi_p(n)}{\sqrt{n}}.
\end{align*}
By H\"older's inequality and the assumption $L(\frac{1}{2},\chi_p) \geq 0$ we have
\begin{align*}
\sum_{p \equiv 1 \, (\text{mod }8)} (\log p) \Phi \left( \frac{p}{X}\right) L \left( \frac{1}{2},\chi_p\right)^3 \geq \frac{T_1^3}{T_2^2},
\end{align*}
where
\begin{align*}
T_1 &:= \sum_{p \equiv 1 \, (\text{mod }8)} (\log p) \Phi \left( \frac{p}{X}\right) L \left( \frac{1}{2},\chi_p\right) R(p)^4, \\
T_2 &:= \sum_{p \equiv 1 \, (\text{mod }8)} (\log p) \Phi \left( \frac{p}{X}\right) R(p)^6.
\end{align*}
In $T_2$ we open up $R(p)^6$, and obtain a sum over $n_1,\ldots,n_6$, and $p$. The terms with $n_1 \cdots n_6 = \square$ yield a main term of size $\ll X(\log X)^6$, and the terms with $n_1 \cdots n_6 \neq \square$ are shown to be an error term by using Lemma \ref{lem:gen char sum over primes}. 

For $T_1$, we write $L(\frac{1}{2},\chi_p)$ using Lemma \ref{lem: approx func eq}. After opening $R(p)^4$, we have a sum over $n_1,\ldots,n_4,m$, and $p$, where $m$ is the variable of summation in the approximate functional equation. The main term $mn_1\cdots n_4 = \square$ is of size $\gg X(\log X)^6$, and the error term $mn_1 \cdots n_4 \neq \square$ is small by Lemma \ref{lem:gen char sum over primes}. This gives the lower bound.

\subsection{The upper bound}

Assuming that $L(\frac{1}{2},\chi_n) \geq 0$ for all square-free $n \equiv 1 \pmod{8}$, we can use an upper bound sieve and positivity to write
\begin{align*}
M_3 := &\sum_{p \equiv 1 \, (\text{mod }8)} (\log p) \Phi \left( \frac{p}{X}\right) L\left( \tfrac{1}{2},\chi_p\right)^3 \\ 
&\leq (\log X) \sum_{n \equiv 1 \, (\text{mod }8)} \mu^2(n)\Bigg(\sum_{\substack{d \mid n \\ d \leq D}} \lambda_d \Bigg) \Phi \left( \frac{n}{X}\right) L \left( \tfrac{1}{2},\chi_n\right)^3
\end{align*}
The coefficients $\lambda_d$ of the sieve are given, as before, by \eqref{lambda}. We take $R$ to be a sufficiently small power of $X$.

We use the approximate functional equation
\begin{align*}
L(\tfrac{1}{2},\chi_n)^3 \ = \ \frac{16}{(\sqrt{2}-1)^6} \sum_{\begin{subarray}{c} \nu=1 \\ \nu \mbox{\scriptsize{ odd}}\end{subarray}}^{\infty} \frac{d_3(\nu) \left( \frac{\nu}{n}\right)}{\sqrt{\nu}} \omega_3\left(\nu\left( \frac{\pi}{n}\right)^{3/2}\right),
\end{align*}
where $\omega_3(\xi)$ is defined by taking $j=3$ in \eqref{eq: defn of omega j}. Our function $\omega_3(\xi)$ is not the same as $\omega_3(\xi)$ in \cite{Sou00}. After using the approximate functional equation to represent $L(\frac{1}{2},\chi_n)^3$, we write $\mu^2(n) = N_Y(n) + R_Y(n)$. The contribution from $R_Y(n)$ is bounded using arguments similar to those in Subsection \ref{subsec:contrib of R_Y}. For $N_Y(n)$ we use Poisson summation as before. Up to negligible error, we therefore have the upper bound
\begin{align*}
M_3 &\leq (\log X) \frac{16}{(\sqrt{2}-1)^6}\sum_{\begin{subarray}{c} d\leq D \\ d|P(z) \\ d \mbox{\scriptsize{ odd}}\end{subarray}} \lambda_d \sum_{\begin{subarray}{c}\nu=1 \\ \nu \mbox{\scriptsize{ odd}}\end{subarray}}^{\infty} \frac{d_3(\nu) }{\sqrt{\nu}}  \sum_{\begin{subarray}{c} \alpha\leq Y \\ \alpha \mbox{\scriptsize{ odd}}\end{subarray}} \mu(\alpha) \\
& \ \ \times  \left( \frac{2[\alpha^2,d]}{\nu}\right)   \frac{X}{[\alpha^2,d]8\nu} \sum_{\begin{subarray}{c} k \in\mathbb{Z} \end{subarray}}e\left(\frac{k\overline{[\alpha^2,d]\nu}}{8}\right)\hat{F}_\nu\left( \frac{kX}{[\alpha^2,d]8\nu}\right) {\tau}_k(\nu),
\end{align*}
where
\begin{align*}
F_{\nu}(t) \ = \ \Phi(t) \omega_3\left( \nu \left(\frac{\pi}{tX}\right)^{3/2}\right).
\end{align*}
We treat separately the contributions from $k = 0$ and $k \neq 0$. The calculations are somewhat easier in that ultimately we seek only upper bounds, not asymptotic formulas.

The contribution from $k = 0$ is treated as in Subsection \ref{subsec:k = 0}, and is
\begin{align*}
\ll X \frac{\log X}{\log R} (\log X)^6 \ll X (\log X)^6.
\end{align*}

For $k \neq 0$ the presence of the additive character necessitates a splitting of $k$ into residue classes modulo $8$. When necessary, we write the additive character as a linear combination of multiplicative characters. We use the identity
\begin{align*}
\tau_k(n) \ = \ \left( \frac{1+i}{2} + \left(\frac{-1}{n}\right)\left(\frac{1-i}{2}\right)\right) G_k(n)
\end{align*}
and treat the two terms separately. We then follow the method of Section \ref{sec:mollified second moment} to obtain that the contribution from $k \neq 0$ is
\begin{align*}
\ll X \frac{\log X}{\log R} (\log X)^6 \ll X (\log X)^6.
\end{align*}

One difference that arises is in proving analogues of Lemma \ref{lem: properties of h(xi,w)}. Here we have $\check{\Phi}(w + \frac{s}{2})$ inside of an integral, instead of just $\check{\Phi}(w)$ outside of an integral. It is helpful to use the bound
\begin{align*}
\check{\Phi}(y) \ll_j \left( \frac{\log X}{|y|}\right)^j.
\end{align*}

Another difference is that we have a factor of $X^{s/2}$ in the integrals, whereas this factor disappeared for the $k \neq 0$ terms in Section \ref{sec:mollified second moment}. We therefore do not need to concern ourselves with any symmetry properties of the integrand (cf. the symmetry argument yielding \eqref{eq: B as a residue}).

\section*{Acknowledgements}

The authors would like to thank Steve Gonek and Matthew Young for helpful comments which have improved the clarity of our exposition.

The second author was supported by NSF grant DMS-1501982, and by the National Science Foundation Graduate Research Program under grant number DGE-1144245.

\end{document}